\global\def\draftcontrol{1}
\documentclass[12pt,a4paper]{amsart}

\usepackage{amsfonts,amsmath,amssymb}
\usepackage{latexsym,fullpage,amsfonts,amssymb,amsmath,amscd,graphics}

\usepackage[all]{xy}
\usepackage{amssymb,amsthm,amsxtra}
\usepackage[dvipsnames]{xcolor}
\usepackage{amscd}
\usepackage{tkz-euclide}
\usepackage{stmaryrd}
\usepackage{mathrsfs}
\usepackage{url}
\usepackage{bbm}
\usepackage{graphicx}
\usepackage{caption}
\usepackage{subcaption}
\usepackage{wasysym}
\usepackage{enumitem}
\setenumerate{label=\alph*),leftmargin=.6cm}
\usepackage[vcentermath]{youngtab}%
\usepackage{framed}
\usepackage{extarrows}

\usepackage{upgreek}

\usepackage[citestyle=numeric,bibstyle=numeric, giveninits=true,doi=false,url=false,isbn=false,
parentracker=false, backend=biber, maxnames=10, date=year, sorting=nty]{biblatex} 

\usepackage{hyperref}
\hypersetup{ unicode, bookmarksnumbered, linktoc = all, hidelinks
}
\usepackage[capitalise]{cleveref}

\input{hacks}

\usepackage{tikz}
\usetikzlibrary{decorations.markings}
\usetikzlibrary{cd}
\usetikzlibrary{patterns}
\usetikzlibrary{shapes}

\tikzset{anchorbase/.style={baseline={([yshift=-0.5ex]current bounding box.center)}}}
\tikzset{wipe/.style={white,line width=4pt}}

\tikzset{->-/.style={decoration={ markings, mark=at position #1 with
  {\arrow{>}}},postaction={decorate}}} \tikzset{-<-/.style={decoration={ markings, mark=at position
  #1 with {\arrow{<}}},postaction={decorate}}}

 \tikzset{darkg/.style={green!70!black}}

\theoremstyle{plain}
\newtheorem*{theorem*}{Theorem}
\newtheorem*{theorema*}{\Alph{Theorem}}
\newtheorem*{remark*}{Remark}
\newtheorem*{example*}{Example}
\newtheorem{lemma}{Lemma}[subsection]
\newtheorem{proposition}[lemma]{Proposition}
\newtheorem{corollary}[lemma]{Corollary}
\newtheorem{theorem}[lemma]{Theorem}

\newtheorem*{conjecture*}{Conjecture}
\newtheorem{thm}[lemma]{Theorem}

\newtheorem{example}[lemma]{Example}

\theoremstyle{definition}
\newtheorem{definition}[lemma]{Definition}

\theoremstyle{remark}
\newtheorem{remark}[lemma]{Remark}

\newcommand{\tr}{\operatorname{tr}} 
\newcommand{\HC}{\operatorname{HC}}

\newcommand{\Hom}{\operatorname{Hom}}

\newcommand{\sgn}{\operatorname{sgn}}
\newcommand{\sdim}{\operatorname{sdim}} 
\newcommand{\str}{\operatorname{str}}

\newcommand{\rank}{\operatorname{rank}}
\newcommand{\s}{\operatorname{ss}}
\newcommand{\aut}{\operatorname{aut}}
\newcommand{\autbar}{\overline{\operatorname{aut}}}

\newcommand{\ad}{\operatorname{ad}}

\renewcommand{\Im}{\operatorname{im}}
\newcommand{\Ker}{\operatorname{ker}} 

\newcommand{\Ext}{\operatorname{Ext}}

\newcommand{\DC}{\operatorname{H}_{\Dirac}}
\newcommand{\DI}{\operatorname{I}}
\newcommand{\bra}{\langle}
\newcommand{\ket}{\rangle}
\newcommand{\bracket}{\langle \cdot,\cdot \rangle}

\def\ZG{{\mathfrak{Z}(\gg)}}
\def\ou{{\overline{u}}}
\def\Sbar{{\overline{S}}}
\def\spin{{S^{\gg,\ll}}}
\def\spinbar{{\overline{S}^{\gg,\ll}}}
\def\spineven{{S^{\gg,\ll}_{0}}} 
\def\spinodd{{S^{\gg,\ll}_{1}}} 
\def\Weyl{{\mathscr W(\ss_{1})}} 
 
\def\CC{{\mathbb C}} 
\def\RR{{\mathbb R}}

\def\UE{{\mathfrak U}}

\def\HH{{\mathcal{H}}}
\def\hh{{\mathfrak h}} 
\def\pp{{\mathfrak p}} 
\def\kk{{\mathfrak k}} 
\def\uu{{\mathfrak u}}
\def\ll{{\mathfrak l}}

\def\glmn{{\mathfrak{gl}(m\vert n)}}

\def\qq{{\mathfrak q}}

\def\aa{{\mathfrak a}} 

\def\bb{{\mathfrak b}} 
\def\even{{\mathfrak{g}_{0}}} 
\def\odd{{\mathfrak{g}_{1}}} 
\def\pp{{\mathfrak p}} 
 
\def\tt{{\mathfrak t}} 
\def\NN{{\mathbb N}} 
\def\nn {{\mathfrak{n}}} 
\def\pp{{\mathfrak{p}}}
 
\def\gg{{\mathfrak{g}}} 
\def\S{{\mathrm{S}}}

\def\calA{{\mathcal{A}}}
\def\partiale{{\frac{\partial}{\partial x_{i}}}}
\def\partialo{{\frac{\partial}{\partial \eta_{j}}}}

\def\ss{{\mathfrak{s}}}
 
\def\ZZ{{\mathbb Z}} 

\def\ubar{{\overline{\mathfrak{u}}}}
\def\gsmod{{\mathfrak{g}}\textbf{-smod}}

\def\gmod{{\mathfrak{g}_{0}}\textbf{-mod}}

\def\calC{{\mathcal C}}

\newcommand{\End}{\operatorname{End}} 
\newcommand{\gr}{\operatorname{gr}}

\newcommand{\Sym}{\operatorname{Sym}}

\newcommand{\Dirac}{\operatorname{D}(\gg,\ll)} 

\newcommand\del[1]{\frac{\partial}{\partial #1}}

\newcommand{\res}{\operatorname{res}}

\newcommand{\ev}{\mathrm{ev}}

\newcommand{\comment}[1]{}

\newcommand{\be}{\begin{equation}}
\newcommand{\ee}{\end{equation}}
\newcommand{\bbar}{\overline{b}}
\newcommand{\psibar}{\overline{\psi}}

\def\Dbar{\overline{D}}
\def\Abar{{\overline A}} 
\def\Cbar{{\overline C}} 
\def\calO{\mathcal O}
\def\calOp{\mathcal{O}^{\pp}}

\def\osp{{\mathfrak{osp}}} 

 \def\id{{\rm id}}

\DeclareFontFamily{U}{stix2bb}{} 
\DeclareFontShape{U}{stix2bb}{m}{n} {<-> stix2-mathbb}{}

\begin{document}

\title{Cubic Dirac Operators and Dirac Cohomology for Basic Classical Lie Superalgebras}

\author{Simone Noja}
 \address{Simone Noja, Institut für Mathematik der Univerit\"at Heidelberg, Im Neuenheimer Feld 205, 69120 Heidelberg, Deutschland}
 \email{noja@mathi.uni-heidelberg.de}
 
  \author{Steffen Schmidt}
 \address{Steffen Schmidt, Institut für Mathematik der Univerit\"at Heidelberg \\Im Neuenheimer Feld 205 \\
69120 Heidelberg, Deutschland}
 \email{stschmidt@mathi.uni-heidelberg.de}

\author{Raphael Senghaas}
 \address{Raphael Senghaas, Institut für Mathematik and Institut f\"ur Theoretische Physik der Univerit\"at Heidelberg, Im Neuenheimer Feld 205,
69120 Heidelberg, Deutschland}
 \email{rsenghaas@mathi.uni-heidelberg.de}
 
\begin{abstract}

We study the Dirac cohomology of supermodules over basic classical Lie superalgebras, formulated in terms of cubic Dirac operators associated with parabolic subalgebras. Specifically, we establish a super-analog of the Casselman--Osborne theorem for supermodules with an infinitesimal character and use it to show that the Dirac cohomology of highest-weight supermodules is always non-trivial. In particular, we explicitly compute the Dirac cohomology of finite-dimensional simple supermodules for basic Lie superalgebras of type 1 with a typical highest weight, as well as of simple supermodules in the parabolic BGG category. We further investigate the relationship between Dirac cohomology and Kostant (co)homology, proving that, under suitable conditions, Dirac cohomology embeds into Kostant (co)homology. Moreover, we show that this embedding lifts to an isomorphism when the supermodule is unitarizable.

\end{abstract}

\maketitle

\setcounter{tocdepth}{1}

\tableofcontents

\section{Introduction}

\subsection{Vue d'ensemble} 
Dirac operators are first-order differential operators whose square equals the Laplacian \cite{DiracOp}. Their birth, in 1928, is deeply rooted in the history of quantum mechanics, as P.A.M.\ Dirac sought a differential equation capable of describing the behavior of relativistic spin-$\frac{1}{2}$ particles in a way consistent with the principles of the then-nascent quantum mechanics \cite{DiracYesTod}. Since their first appearance, the theory of Dirac operators has developed dramatically, becoming an exciting branch of modern mathematics with achievements and applications spanning topology, global analysis, and theoretical physics.\footnote{Indeed, the history of Dirac operators intersects with that of the celebrated Atiyah-Singer index theorem, which is regarded as one of the most important results in 20th-century mathematics \cite{Lawson}.}

The theory of Dirac operators took an algebraic turn in the 1970s when Parthasarathy, Atiyah, and Schmid constructed the first instance of algebraic Dirac operators in the context of the representation theory of symmetric pairs \cite{atiyah1977geometric, parthasarathy1972dirac}. Specifically, they realized discrete series representations as kernels of Dirac operators and formulated an inequality (the Dirac inequality) that proved crucial for the classification of unitary highest-weight modules, which was completed shortly afterward in \cite{enright1983classification}. Building on these results, Vogan \cite{vogan1997dirac} later introduced a cohomology theory, known as Dirac cohomology, which captures the infinitesimal character of an irreducible representation.

Shortly thereafter, Kostant \cite{kostant1999cubic} generalized the notion of (algebraic) Dirac operators and Dirac cohomology to quadratic Lie algebras, which include the original case of symmetric pairs treated in \cite{atiyah1977geometric} and \cite{parthasarathy1972dirac} as a special case. Crucially, in order to extend the theory beyond symmetric pairs, a cubic term was necessary, leading Kostant to introduce the notion of cubic Dirac operators -- the main characters of the present paper.

In more recent years, the algebraic theory of Dirac operators and Dirac cohomology has been further developed and applied to the study of representations of Lie superalgebras by Huang and Pandžić \cite{huang2005dirac}. Analogous to the classical case, they showed that Dirac cohomology determines the infinitesimal character of an irreducible representation. In this superalgebraic setting, the relationship between Dirac cohomology and unitarity was subsequently investigated by Xiao \cite{xiao2015dirac} and the second author \cite{schmidt}.

Kostant’s analog of cubic Dirac operators for quadratic Lie superalgebras was recently studied by Kang and Chen \cite{kang2021dirac} and Meyer \cite{Meyer}, who extended the main properties of cubic Dirac operators from the classical to the super case. However, these works do not address the corresponding Dirac cohomology.

\subsection{Main results}






The aim of this paper is to study the Dirac cohomology associated with cubic Dirac operators, focusing on its applications to the representation theory of supermodules, where the formalism proves to be particularly powerful.

In particular, we focus on basic classical Lie superalgebras $\gg$, together with a choice of parabolic subalgebra $\pp \subset \gg$, which induces a decomposition $\gg = \ll \oplus \ss$, where $\ll$ is the Levi subalgebra and $\ss$ is its orthogonal complement in $\gg$.

In this setting, Dirac cohomology fully demonstrates its strength when the input 
$\gg$-supermodule admits an infinitesimal character $\chi_\lambda : \ZG \to \CC $, \emph{i.e.}, if there exist $\lambda \in \hh^\ast$ such that  every element $z \in \ZG$ in the center of $\gg$ acts on the supermodule by a scalar multiple of the identity $\chi_\lambda (z) \in \CC$. In this case, we establish a super-analog of the Casselman--Osborne lemma, which we rephrase here in the following form.

\begin{theorem} Let $M$ be a $\gg$-supermodule with infinitesimal character $\chi_\lambda$ for $\lambda \in \hh^\ast$.
Then $z \in \ZG$ acts on the Dirac cohomology $\DC(M)$ as $\eta_{\, \ll} (z) \in \mathfrak{Z} (\ll)$ for a uniquely defined algebra homomorphism $\eta_{\, \ll} : \ZG \to \mathfrak{Z}(\ll).$ \\
In particular, if $V$ is a sub $\ll$-supermodule of $\DC(M)$ with infinitesimal character $\chi_{\mu}^{\ll}$ for $\mu \in \hh^{\ast}$, then $\chi_{\lambda} = \chi_{\mu}^{\ll} \circ \eta_{\ll}$.   
\end{theorem}

The relevant properties of the homomorphism $\eta_{\, \ll} : \ZG \to \mathfrak{Z}(\ll),$ together with its relation to the Harish-Chandra homomorphism are discussed in Section \ref{subsec::infchar}.\\
Supermodules that admit an infinitesimal character encompass the important class of highest-weight supermodules. Building on the super Casselman--Osborne lemma, we demonstrate that the Dirac cohomology of this distinguished class is always non-trivial. 


\begin{theorem}
    Let $M$ be a highest weight $\gg$-supermodule. Then its Dirac cohomology $\DC(M)$ is non-trivial.
\end{theorem}

Furthermore, in the case of finite-dimensional supermodules, we explicitly compute relevant examples of Dirac cohomology for Lie superalgebras of type 1, \emph{i.e.}, $\gg = \glmn$, $A(m|n)$, or $C(n)$. More precisely, leaving the details of the notations to Section \ref{subsec:fin_dim}, we prove that the Dirac cohomology decomposes as follows.


\begin{theorem} 
    Let $M$ be an admissible finite-dimensional simple $(\gg, \ll)$-supermodule with typical highest weight $\Lambda$. Then its Dirac cohomology reads
    \[
    \DC(M) = \bigoplus_{w \in W_{\Lambda + \rho^{\uu}}^{\ll,1}} L_{\ll}(w(\Lambda + \rho) - \rho^{\ll}),
    \]
where $\rho$ and $\rho^\ll$ denote the respective Weyl vectors.
\end{theorem}

  In turn, this allows to compute the Dirac cohomology of finite-dimensional simple objects in the parabolic BGG category $\calO^{\pp}$, for some parabolic subalgebra $\pp = \ll \ltimes \uu$ with reductive Lie algebra $\ll_{0}$.  

\begin{theorem} 
    Let $M$ be a finite-dimensional simple object in the parabolic BGG category $\calO^{\pp}$ with highest weight $\Lambda$. Then its Dirac cohomology reads 
    \[
    \DC(M) = \bigoplus_{w \in W_{\Lambda + \rho^{\uu}}^{\ll_{0},1}} L_{\ll_{0}}(w(\Lambda + \rho^{\uu} + \rho^{\ll_{0}})-\rho^{\ll_{0}}),
    \]
where $\rho^{\uu}$ and $\rho^{\ll_0}$ denote the respective Weyl vectors.
\end{theorem}

Finally, we examine the relationship between the Dirac cohomology and the Lie algebra cohomology of supermodules. In particular, we prove that Dirac cohomology embeds into Kostant's $\uu$-cohomology. Moreover, when the supermodule satisfies suitable conditions (\emph{e.g.}, it is unitarizable), we show that this embedding is, in fact, an isomorphism. More specifically, we establish the following result.

\begin{theorem} \label{thm::inclusion_DC}
    Let $\pp = \ll \ltimes \uu$ be a parabolic subalgebra, and let $M$ be an admissible simple $(\gg,\ll)$-supermodule, that is, $M$ is $\ll$-semisimple. Then there exist injective $\ll$-supermodule morphisms 
    \[
    \DC(M) \hookrightarrow H^{\ast}(\uu,M), \qquad \DC(M) \hookrightarrow H_{\ast}(\ubar,M).
    \]
    Moreover, if $M$ is unitarizable, the above maps are isomorphisms of $\ll$-supermodules. 
\end{theorem}
It is interesting to observe that the results of this work suggest that a generic supermodule may be entirely determined by its Dirac cohomology. This and other related questions are currently under investigation and will be addressed in future work.

Furthermore, with a view toward explicit computations of the Dirac cohomology of $\gg$-supermodules, we expect that similar results to those presented above will also hold for type 2 Lie superalgebras -- hence all basic classical Lie superalgebras.

\subsection{Leitfaden} 

More specifically, the paper is structured as follows. Our conventions are established at the end of Section 1. In Section 2, we provide a brief introduction to the fundamental definitions, along with some preliminary notions and results. In particular, as mentioned above, after introducing basic classical Lie superalgebras $\gg$, we discuss their parabolic subalgebras $\pp \subset \gg$ and the induced decomposition $\gg = \ll \oplus \ss$, where $\ll$ is the Levi subalgebra and $\ss$ is its orthogonal complement in $\gg$. We then examine the role of parabolic subalgebras in representation theory of Lie superalgebras, with a particular focus on highest-weight theory. Building on the preliminary notions to be utilized in subsequent sections of the manuscript, we then introduce real forms of $\gg$ and unitarizable supermodules. In the final part of the section, we turn our attention to Clifford superalgebras and exterior superalgebras. In particular, we explicitly construct an embedding of the Levi subalgebra $\ll$ into the Clifford superalgebra $C (\ss)$, which will play a significant role in the subsequent sections of the paper.


In Section 3, we introduce the primary object of study in this paper, the cubic Dirac operator $\Dirac$. We first examine its fundamental properties and establish a decomposition into $\ll$-invariant summands, as stated in Theorem \ref{theorem::decomp_Dirac}. We then proceed to define the Dirac cohomology
$\DC (M)$ of a $\gg$-supermodule $M$ in Section \ref{subsec::DiracCohomo}, where the notion of the oscillator supermodule $\overline{M}(\ss)$ plays a central role.
Subsequently, we focus on the class of supermodules admitting an infinitesimal character, a setting in which the theory demonstrates particular effectiveness. Within this framework, we establish a super-analog of the Casselman--Osborne lemma in Section \ref{subsec::infchar}. Finally, we conclude Section 3 with a concise discussion of certain homological properties of Dirac cohomology.


In Section 4, we investigate the Dirac cohomology of highest weight 
$\gg$-supermodules. First, we establish that the Dirac cohomology of highest-weight supermodules is always non-trivial (Proposition \ref{prop::HW_non-trivial}). We then refine this analysis to certain classes of finite-dimensional supermodules, explicitly computing their cohomology. More precisely, we determine the Dirac cohomology of finite-dimensional supermodules for Lie superalgebras of type 1 with typical highest weight (Theorem \ref{thm::DC_finite_dimensional}) and of simple supermodules in $\calOp$ (Theorem \ref{thm::DC_OP}).


Finally, in Section 5, after introducing Kostant's $\uu$-cohomology and $\bar \uu$-homology, we investigate its relationship with Dirac cohomology. In particular, we establish that Dirac cohomology can always be embedded into Kostant's $\uu$-cohomology and $\bar \uu$-homology.

In the second part of the section, we shift our focus to unitarizable $\gg$-supermodules. For this geometrically significant class of $\gg$-supermodules, we demonstrate a Hodge-like decomposition for the cubic Dirac operator and prove that, as $\ll$-supermodules, their Dirac cohomology is isomorphic — up to a twist — to their Kostant $\bar\uu$-cohomology. Finally, we examine simple weight $\gg$-supermodules and establish that their Dirac cohomology is trivial unless they are of highest weight type. This result can be regarded as a generalization, in the supergeometric setting, of the corresponding classical result for reductive Lie algebras over $\CC$.



\subsection{Conventions} \label{subsec::conventions}
We denote by $\ZZ_{+}$ the set of positive integers. Let $\ZZ_{2} := \ZZ / 2\ZZ$ be the ring of integers modulo $2$. We denote the elements of $\ZZ_{2}$ by $\overline{0}$ (the residue class of even integers) and $\overline{1}$ (the residue class of odd integers). The ground field is $\CC$, unless otherwise stated.

If $V := V_{0} \oplus V_{1}$ is a super vector space and $v \in V$ is a homogeneous element, then $p(v)$ denotes the parity of $v$, meaning $p(v) = 0$ if $v \in V_{0}$ and $p(v) = 1$ if $v \in V_{1}$.

For a given Lie superalgebra $\gg = \even \oplus \odd$, we denote its universal enveloping superalgebra by $\UE(\gg)$. The universal enveloping algebra of the Lie subalgebra $\even$ is denoted by $\UE(\even)$. The centers are denoted by $\ZG$ and $\mathfrak{Z}(\even)$, respectively.

By a (left) $\gg$-supermodule, we mean a super vector space $M = M_{0} \oplus M_{1}$ equipped with a graded linear (left) action of $\gg$, such that $[X, Y]v = X(Yv) - (-1)^{p(X)p(Y)} Y(Xv)$ for all homogeneous $X, Y \in \gg$ and $v \in M$. Analogously, a $\UE(\gg)$-supermodule is defined. A morphism $f: M \to N$ of $\gg$-supermodules is a linear map such that $f(M_{i}) \subset N_{i}$ and $f(Xv) = Xf(v)$ for all $X \in \gg$ and $v \in M$. We write $\gsmod$ for the category of all (left) $\gg$-supermodules. This category is a $\CC$-linear abelian category, equipped with an endofunctor $\Pi$, the \emph{parity reversing functor}. The parity reversing functor is defined by $\Pi(M)_{0} = M_{1}$ and $\Pi(M)_{1} = M_{0}$. Moreover, $\Pi(M)$ is viewed as a $\gg$-supermodule with the new action $X \cdot v := (-1)^{p(X)} Xv$ for any $X \in \gg$ and $v \in M$. In particular, a $\gg$-supermodule $M$ is not necessarily isomorphic to $\Pi(M)$.

We denote by $\gmod$ the category of (left) $\even$-modules. When $\even$ is considered as a purely even Lie superalgebra, the category $\gsmod$ is simply a direct sum of two copies of $\gmod$. We view any $\even$-module as a $\even$-supermodule concentrated in a single parity. Additionally, if we disregard the parity, any $\gg$-supermodule $M$ can be viewed as a $\even$-module, denoted by $M_{\ev}$.

We will not make an explicit distinction between the categories $\gsmod$ and $\UE(\gg)$\textbf{-smod} or $\gmod$ and $\UE(\even)$\textbf{-mod}, which are the categories of all (left) $\UE(\gg)$-supermodules and left $\UE(\even)$-modules, respectively.

Moreover, $\gsmod$ is a tensor category. For any two $\gg$-supermodules $M,N$, we have a $\ZZ_{2}$-graded tensor product $M \hat{\otimes} N$ with
\[
(M \hat{\otimes} N)_{0} := (M_{0} \otimes N_{0}) \oplus (M_{1}\otimes N_{1}), \quad
(M \hat{\otimes} N)_{1} := (M_0 \otimes N_1) \oplus (M_1 \otimes N_0),
\]
where $\otimes$ denotes the ordinary tensor product of vector spaces. 
The $\gg$-action is
\[
X (m \otimes n) := Xm\otimes n + (-1)^{p(X)p(m)}m\otimes Xn, \quad X \in \gg, \ m, n \in M.
\]

Finally, for any Lie subsuperalgebra $\ll \subset \gg$, we denote by $\calC(\gg, \ll)$ the category of all $\gg$-supermodules that are $\ll$-semisimple. An object in $\calC(\gg, \ll)$ is called $(\gg, \ll)$\emph{-admissible}.

\medskip\noindent \textit{Acknowledgments.} We extend special thanks to Johannes Walcher for numerous conversations and collaboration on related projects. 
This work is funded by the Deutsche Forschungsgemeinschaft (DFG, German Research Foundation) under
project number 517493862 (Homological Algebra of Supersymmetry: Locality, Unitary, Duality). 
This work is funded by the Deutsche Forschungsgemeinschaft (DFG, German Research Foundation) under
Germany’s Excellence Strategy EXC 2181/1 — 390900948 (the Heidelberg STRUCTURES Excellence Cluster).

\section{Preliminaries}
In this article, let $(\gg := \even \oplus \odd, [\cdot,\cdot])$ denote a basic classical Lie superalgebra, meaning $\gg$ is simple, $\even$ is a reductive Lie algebra, and $\gg$ possesses a non-degenerate, invariant and consistent bilinear form $(\cdot,\cdot)$. \emph{Consistency} means $(X,Y) = 0$ whenever $p(X) \neq p(Y)$ for $X,Y \in \gg$, while \emph{invariance} means $([X,Y],Z) = (X,[Y,Z])$ for all $X,Y,Z \in \gg$. Such a form is unique up to a complex scalar multiple. Kac established that the complete list includes all simple Lie algebras along with the following types of Lie superalgebras \cite{Kac}:
\[
A(m\vert n), \quad B(m\vert n), \quad C(n), \quad D(m\vert n), \quad F(4), \quad G(3), \quad D(2,1;\alpha).
\]
The general linear Lie superalgebra $\glmn$ for $m,n \geq 1$ is also declared to be basic classical. Throughout, we assume $\odd \neq {0}$ and $\alpha \in \RR$ for $D(2,1;\alpha)$, ensuring that $\gg$ is also a contragredient Lie superalgebra. For the basic classical Lie superalgebras $A(m|n)$ with $m \neq n$, $B(m|n)$, $C(n+1)$, $D(m|n)$ with $m \neq n+1$, $F(4)$, and $G(3)$, we use the \emph{Killing form} as the non-degenerate, invariant, consistent bilinear form $(\cdot,\cdot)$ given by 
\be
(x,y) := \str(\ad_{x} \circ \ad_{y}), \qquad x,y \in \gg,
\ee
where $\ad_{z}(\cdot) := [z,\cdot]$ denotes the adjoint representation of $\gg$ and $\str(\cdot)$ denotes the supertrace. For the remaining basic classical Lie superalgebras, the Killing form vanishes identically, and an alternative form can be constructed in an ad hoc manner \cite{Kac}. We choose the form constructed in \cite[Section 5.4]{Musson}, and for simplicity, we refer to this form as the Killing form. 

Let $\hh \subset \gg$ be a Cartan subalgebra, and denote the set of roots by $\Phi := \Phi(\gg,\hh)$, so that $\gg$ has the root space decomposition 
\be
\gg = \hh \oplus \bigoplus_{\alpha \in \Phi}\gg^{\alpha}
\ee
with corresponding root spaces $\gg^{\alpha} := \{ X \in \gg : [H,X] = \alpha(H)X \ \text{for all} \ H \in \hh\}$. Note that $\hh \subset \even$, since $\gg$ is basic classical. Moreover, we fix some positive system $\Phi^{+}$ and define the \emph{Weyl vector} to be
\be
\rho := \frac{1}{2}\sum_{\alpha \in \Phi^{+}}\alpha.
\ee
The \emph{Weyl group} $W^{\gg}$ of $\gg$ is defined to be the Weyl group of the underlying Lie algebra $\even$. Moreover, for a fixed positive system $\Phi^{+}$, we define the \emph{fundamental system} to be the set of all $\alpha \in \Phi^{+}$ which cannot be written as the sum of two roots in $\Phi^{+}$. Elements belonging to the fundamental system are called \emph{simple}. 

The following proposition collects some important properties of the root system $\Phi$ and the Killing form $(\cdot,\cdot)$.

\begin{proposition}[{\cite[Proposition 2.5.5]{Kac}}]\label{prop::structure_theory_BCLSA} 
The following assertions hold: 
\begin{enumerate}
    \item[a)] If $\alpha \in \Phi$, then $-\alpha \in \Phi$. Moreover, $\Phi = -\Phi$, $\Phi_{0} = -\Phi_{0}$, and $\Phi_{1} = -\Phi_{1}$.
    \item[b)] $\sdim(\gg^{\alpha}) = (1\vert 0)$ for all $\alpha \in \Phi_{0}$, and $\sdim(\gg^{\alpha}) = (0\vert 1)$ for all $\alpha \in \Phi_{1}$.
    \item[c)] $[\gg^{\alpha},\gg^{\beta}] = 0$ if and only if $\alpha, \beta \in \Phi$ and $\alpha + \beta \notin \Phi$, while $[\gg^{\alpha},\gg^{\beta}] \subset \gg^{\alpha + \beta}$ for all $\alpha, \beta \in \Phi$ if $\alpha + \beta \in \Phi$. In particular, $[\gg^{\alpha},\gg^{-\alpha}]$ is a one-dimensional subspace in $\hh$ for all $\alpha \in \Phi$.
    \item[d)] The restriction of the invariant form $(\cdot,\cdot)$ on $\hh \times \hh$ is non-degenerate, and $(\gg^{\alpha},\gg^{\beta}) = 0$ unless $\alpha = -\beta \in \Phi$.
    \item[e)] Fix a nonzero positive root vector $e_{\alpha} \in \gg^{\alpha}$. Then $[e_{\alpha},e_{-\alpha}] = (e_{\alpha},e_{-\alpha})h_{\alpha}$, where $h_{\alpha}$ is the \emph{coroot} determined by
    \[
    (h_{\alpha},h) = \alpha(h) \quad \text{for all } h \in \hh.
    \]
    \item[f)] The bilinear form on $\hh^{\ast}$ defined by $(\lambda,\mu) := (h_{\lambda},h_{\mu})$ is non-degenerate and invariant under the Weyl group of $\even$.
\end{enumerate}
\end{proposition}

The basic classical Lie superalgebras $\gg$ can be further divided into two types, called type 1 and type 2. For $\gg$, one of the following two assertions holds:
    \begin{itemize}
        \item[a)] There is a $\ZZ$-grading $\gg = \gg_{-1} \oplus \gg_0 \oplus \gg_1$, such that $\gg_0 = \even$ and $\gg_{\pm 1}$ are simple $\even$-modules.
        \item[b)] The even part $\even$ is semisimple, and $\odd$ is a simple $\even$-module.
    \end{itemize}
We say that $\gg$ is \emph{of type} 1 (respectively \emph{type} 2) if it satisfies a) (respectively b)). The basic classical Lie superalgebras of type $1$ are $A(m\vert n)$ and $C(n)$ together with $\glmn$, while $B(m\vert n), D(m\vert n), F(4), G(3)$ and $D(2,1;\alpha)$ are of type $2$.

\subsection{Parabolic subalgebras, parabolic induction, highest weight supermodules} 
\subsubsection{Parabolic subalgebras} There are essentially two ways to define parabolic subalgebras of $\gg$, which we refer to as the classical approach via a parabolic set of roots \cite{Bourbaki, Rao}, and the hyperplane approach \cite{DMP}. However, the two approaches are not equivalent, and the latter appears more natural \cite{DFG}. We follow the approach in \cite{DMP}.

Let $X$ be the finite-dimensional real vector space $X := \RR \otimes_{\ZZ} Q$, where $Q$ is the abelian group generated by $\Phi$.

\begin{definition}
    A partition $\Phi = \Phi^{-}_{T} \sqcup \Phi^{0}_{T} \sqcup \Phi^{+}_{T}$ of the set of roots $\Phi$ is called a \emph{triangular decomposition} $T$ if there exists a functional $l : X \to \ZZ$ such that 
    \[
    \Phi^{0}_{T} = \ker l \cap \Phi, \qquad \Phi^{\pm}_{T} = \{ \alpha \in \Phi : l(\alpha) \gtrless 0\}.
    \]
\end{definition}

For any triangular decomposition $T$ of $\Phi$, we obtain a triangular decomposition of $\gg$, that is, a decomposition $\gg = \gg_{T}^{+} \oplus \gg_{T}^{0} \oplus \gg_{T}^{-}$, where
\be
\gg_{T}^{+} := \bigoplus_{\alpha \in \Phi_{T}^{+}}\gg^{\alpha}, \qquad \gg_{T}^{0} := \bigoplus_{\alpha \in \Phi_{T}^{0}} \gg^{\alpha},\qquad \gg_{T}^{-} := \bigoplus_{\alpha \in \Phi_{T}^{-}} \gg^{\alpha}.
\ee
The subset $P_{T} = \Phi^{0}_{T}\sqcup \Phi^{+}_{T}$ is called a \emph{principal parabolic subset}. The following lemma is straightforward.

\begin{lemma}[{\cite{DFG}}] \label{lemm::parabolic_subset}
    Every principal parabolic subset $P \subset \Phi$ is a \emph{parabolic subset}, \emph{i.e.}, the following conditions hold:
    \begin{enumerate}
        \item[a)] $\Phi = P \sqcup -P$, and
        \item[b)] $\alpha, \beta \in P$ with $\alpha + \beta \in \Phi$ implies $\alpha + \beta \in P$.
    \end{enumerate}
\end{lemma}

This leads us to the definition of a parabolic subalgebra.

\begin{definition}
    Let $T$ be a triangular decomposition of $\Phi$. The Lie subsuperalgebra $\pp_{T} := \gg^{0}_{T} \oplus \gg_{T}^{+}$ is called a \emph{parabolic subalgebra} of $\gg$.
\end{definition}

\begin{remark}
    For any triangular decomposition $\pp$, the space $\pp_{T} \cap \even$ is a parabolic subalgebra of $\even$ \cite[Section 5]{DMP}.
\end{remark}

A \emph{root subalgebra} is a Lie subsuperalgebra $\qq \subset \gg$ such that 
\be
\qq = (\qq \cap \hh) \oplus \bigl( \bigoplus_{\alpha \in \Xi}\gg^{\alpha}\bigr)
\ee
for some subset $\Xi \subset \Phi$. In particular, for a given triangular decomposition $T$ with principal parabolic subset $P_{T}$, the parabolic subalgebra is
    \be
    \pp_{T} = \hh \oplus \bigl(\bigoplus_{\alpha \in P_{T}}\gg^{\alpha}\bigr).
    \ee

Any parabolic subalgebra $\pp_{T}$ admits a Levi decomposition. For the parabolic set of roots $P_{T} \subset \Phi$, we define $L_{T} := P_{T} \cap (-P_{T})$ as the \emph{Levi component}, $U_{T} := P_{T} \setminus (-P_{T})$ as the \emph{nilradical}, and $P_{T} = L_{T} \sqcup U_{T}$ as the \emph{Levi decomposition}. The associated root subalgebras 
\be
\ll_{T} := \hh \oplus \bigl(\bigoplus_{\alpha \in L_{T}}\gg^{\alpha}\bigr), \qquad \uu_{T} := \bigoplus_{\alpha \in U_{T}} \gg^{\alpha},
\ee
are called the \emph{Levi subalgebra} and \emph{nilradical} of $\pp_{T}$. A direct calculation yields that $\uu_{T}$ is an ideal in $\pp_{T}$. The \emph{Levi decomposition} of $\pp_{T}$ takes the form of the semidirect product $\pp_{T} := \ll_{T} \ltimes \uu_{T}$. If a parabolic subalgebra $\pp_{T}$ is fixed, we omit the subscript $T$ and denote it simply by $\pp$. 

Any parabolic subalgebra $\pp_{T}$ has an \emph{opposite parabolic subalgebra} $\overline{\pp}_{T}$, such that $\gg = \pp_{T} + \overline{\pp}_{T}$. The opposite parabolic subalgebra is $\overline{\pp}_{T} = \ll_{T} \ltimes \ubar_{T}$, where $\ubar_{T}$ is the root subalgebra corresponding to $U^{-}_{T} := (-P_{T}) \setminus P_{T}$. 

Furthermore, the Levi subalgebra has a proper root system $\Phi(\ll_{T};\hh)$, since $\hh \subset \ll$, which is a subset of $\Phi$. We denote the associated Weyl group by $W^{\ll_{T}}$. The positive system $\Phi^{+}_{T}$ induces a positive system for $P_{T},L_{T}$ and $U_{T}$ by $P^{+}_{T} := P_{T} \cap \Phi^{+}$, $L^{+}_{T} := L_{T} \cap \Phi^{+}$ and $U^{+}_{T} := U_{T} \cap \Phi^{+}$. We set $(U^{+}_{T})_{0,1}$ and $(L^{+}_{T})_{0,1}$ for the associated even and odd parts. Further, we define $\rho^{\ll_{T}} := \rho^{\ll_{T}}_{0} - \rho^{\ll_{T}}_{1}$ and $\rho^{\uu_{T}} := \rho^{\uu_{T}}_{0} - \rho^{\uu_{T}}_{1}$ for 
\be
\rho^{\ll_{T}}_{0,1} := \frac{1}{2} \sum_{\alpha \in (L_{T}^{+})_{0,1}} \alpha, \qquad \rho^{\uu_{T}}_{0,1} := \frac{1}{2} \sum_{\alpha \in (U_{T}^{+})_{0,1}}\alpha.
\ee

Finally, fix a parabolic subalgebra $\pp := \pp_{T}$ for some parabolic set $P$ with Levi decomposition $\pp = \ll \ltimes \uu$ and opposite parabolic subalgebra $\overline{\pp} = \ll \ltimes \ubar$. Set $\ss := \uu \oplus \overline{\uu}$. By construction of $\ll$ and $\ss$, the space $\ss$ is the orthogonal complement of $\ll$ with respect to $(\cdot,\cdot)$, and we have a direct sum decomposition 
\be
\gg = \ll \oplus \ss,
\ee
where we use $(\gg^{\alpha},\gg^{\beta}) = 0$ unless $\alpha = -\beta \in \Phi$ (cf.\ Proposition \ref{prop::structure_theory_BCLSA}).
In particular, the restriction of $(\cdot,\cdot)$ to $\ll$ and $\ss$ remains non-degenerate. In the next section, we will show that $\gg$ is an example of a quadratic Lie superalgebra and introduce a cubic Dirac operator associated with this decomposition.

\subsubsection{Parabolic induction} We fix a parabolic subalgebra $\pp = \ll \ltimes \uu$. We are interested in \emph{weight supermodules}, that is, supermodules $M$ where $\hh$ acts semisimply:
\be
M = \bigoplus_{\mu \in \hh^{\ast}} M^{\mu}, \qquad M^{\mu} = \{m \in M : hm = \mu(h)m \ \text{for every} \ h \in \hh\}.
\ee
The elements $\mu \in \hh^{\ast}$ with $M^{\mu} \neq \{0\}$ are called \emph{weights} of $M$, while $M^{\mu}$ is called \emph{weight space} associated to $\mu$. 

Let $V$ be a weight $\ll$-supermodule. Via the projection $\pp \to \ll$, we naturally extend $V$ to a $\pp$-supermodule, where the nilradical $\uu$ acts trivially on $V$. Conversely, given a non-trivial $\pp$-supermodule $M$, the space of $\uu$-invariants is non-zero and carries the structure of both a $\pp$ and $\ll$-supermodule.

\begin{lemma}
    Any simple weight $\ll$-supermodule $V$ is a simple $\pp$-supermodule with the trivial action of $\uu$. Conversely, if $V$ is a simple weight $\pp$-supermodule, then $\uu$ acts trivially on $V,$ and $V$ is a simple weight $\ll$-supermodule.
\end{lemma}

\begin{proof}
    Given a simple weight $\ll$-supermodule $V$, we can extend $V$ to a simple weight $\pp$-supermodule by let $\uu$ acting trivially, since $\uu \subset \pp$ is an ideal. 

    Let $V$ be a simple weight $\pp$-supermodule, and denote by $T$ the triangular decomposition defining $\pp$ with linear map $l : X \to \ZZ$. Fix a weight $\mu$ of $V$ and define two subsupermodules
    \[
    U := \bigoplus_{\nu \in \hh^{\ast}, \ l(\nu) \geq l(\mu)} V^{\nu}, \qquad W := \bigoplus_{\nu \in \hh^{\ast}, \ l(\nu) > l(\mu)} V^{\nu}.
    \]
    As $\mu$ is a weight of $V$, the subsupermodule $U$ is non-trivial, while $W^{\gg}$ is proper. Since $V$ is simple, we must have $U = V$ and $W = \{0\}$. Consequently, by induction, $\uu$ acts trivially on $V$, and $V$ is a simple $\ll$-supermodule.
\end{proof}

For a simple weight $\ll$-supermodule $V$, considered equivalently as a simple weight $\pp$-supermodule, the \emph{parabolically induced supermodule (or generalized Verma supermodule)} is defined to be
\be
M_{\pp}(V) := \UE(\gg) \otimes_{\UE(\pp)} V.
\ee

The following proposition is standard.

\begin{proposition}
    The $\gg$-supermodule $M_{\pp}(V)$ has a unique maximal proper subsupermodule. In particular, $M_{\pp}(V)$ has a unique simple quotient $L_{\pp}(V)$.
\end{proposition}

The simple weight $\gg$-supermodules $L_{\pp}(V)$ exhaust all simple weight $\gg$-supermodules. To establish this, we introduce the following notations. For a fixed triangular decomposition $T$ of $\Phi$, we write $\Phi^{T} = \Phi^{T}_{0} \sqcup \Phi^{T}_{1}$ for the set of roots of $\gg_{T}^{0}$. A triangular decomposition $T$ of $\Phi$ is called \emph{good}, if the following holds \cite{DMP}: 
\begin{enumerate}
    \item The monoid generated by $\Phi_{0}^{T}$ is a group, denoted by $Q_{0}^{T}$.
    \item For any $\beta \in \Phi_{1}^{T}$, there exists some $m > 0$ such that $m\beta \in Q_{0}^{T}$.
\end{enumerate}
In this case, the Levi subalgebra $\ll_{T} = \gg_{T}^{0}$ is called \emph{good}. The good triangular decompositions of $\gg$ were classified in \cite[Section 7]{DMP}. 

Let $\ll_{T}$ be a good Levi subalgebra. A weight $\ll_{T}$-supermodule $V$ is called \emph{cuspidal} if for any $\alpha \in \Phi^{T}_{0}$ the associated root vector $e_{\alpha}$ acts injectively on $V$. 

\begin{thm}[{\cite[Theorem 6.1]{DMP}}] \label{thm::DMP}
Let $M$ be a simple $\gg$-supermodule. Then there exists a parabolic subalgebra $\pp$ with good Levi subalgebra $\ll$, and a cuspidal $\ll$-supermodule $V$, such that 
\[
M \cong L_{\pp}(V).
\]
\end{thm}

\subsubsection{Parabolic category $\calO^{\pp}$} \label{subsubsec::Op} Categorically, the parabolic induction can be naturally studied in the \emph{parabolic BGG category} $\calOp$, which is a variant of the super BGG category $\calO$, determined by a (fixed) parabolic subalgebra $\pp = \ll \ltimes \uu$. Following \cite{Mazorchuk}, the category $\calOp$ is the full subcategory of $\gsmod$ whose objects are the $\gg$-supermodules satisfying the following three properties:
\begin{enumerate}
    \item[1.] $M$ is a finitely generated $\UE(\gg)$-supermodule. 
    \item[2.] Viewed as a $\UE(\ll_{0})$-module, $M$ decomposes in a direct sum of finite-dimensional simple modules.
    \item[3.] $M$ is locally $\uu$-finite in the sense that $\dim(\UE(\uu)v) < \infty$ for all $v \in M$.
\end{enumerate}
This is an abelian subcategory of $\gsmod$ closed under the parity switching functor $\Pi$. For further categorical properties of $\calO^{\pp}$, we refer to \cite{Mazorchuk}. However, it is important for us to note that the simple objects in $\calOp$ are precisely given by $L_{\pp}(V)$, where $V$ is a simple $\ll$-supermodule.

\subsubsection{Highest weight supermodules} \label{subsubsec::Highest_weight_supermodules} An important class of weight $\gg$-supermodules consists of highest weight $\gg$-supermodules, which we now define. With respect to some choice $\Phi^{+}$ of positive roots, the Lie superalgebra $\gg$ has a \emph{triangular decomposition}
\be
\gg = \nn^{-} \oplus \hh \oplus \nn^{+}, \qquad \nn^{\pm} := \sum_{\alpha \in \Phi^{+}} \gg^{\pm \alpha}.
\ee
The associated \emph{Borel subalgebra} is $\bb := \hh \oplus \nn^{+}$, which is a particular example of a parabolic subalgebra.

\begin{definition}
    A $\gg$-supermodule $M$ is called a \emph{highest weight} $\gg$\emph{-supermodule} with respect to a positive system $\Phi^{+}$, if there exists a nonzero vector $v_{\Lambda} \in M$ with $\Lambda \in \hh^{\ast}$ such that the following holds:
\begin{enumerate}
    \item[a)] $Xv_{\Lambda}=0$ for all $X \in \nn^{+}$,
    \item[b)] $Hv_{\Lambda}=\Lambda(H)v_{\Lambda}$ for all $H \in \hh$, and 
    \item[c)] $\UE(\gg)v_{\Lambda}=M$.
\end{enumerate}
The vector $v_{\Lambda}$ is referred to as the highest weight vector of $M$, and $\Lambda$ is referred to as the \emph{highest weight}.
\end{definition}

Throughout this article, we fix a positive system $\Phi^{+}$, which will remain implicit in the subsequent discussion. The associated Borel subalgebra will be denoted by $\bb$.

\begin{example}
    Any finite-dimensional simple $\gg$-supermodule is a highest weight $\gg$-supermodule. 
\end{example}

We realize any simple highest weight $\gg$-supermodule as the unique simple quotient of certain universal highest weight supermodule, called \emph{Verma supermodule}. For a detailed discussion of Verma supermodules, we refer to \cite[Chapter 9]{Musson}. First, note that  $\UE(\gg)$ is a right $\UE(\bb)$-supermodule with respect to right multiplication. For any $\lambda \in \hh^{\ast}$, we define the \emph{Verma supermodule} by
\be \label{eq::Verma_supermodules}
M_{\bb}(\lambda) := \UE(\gg) \otimes_{\UE(\bb)} \CC_{\lambda},
\ee
where $\CC_{\lambda}$ is the one-dimensional $\bb$-supermodule with trivial action of $\nn^{+}$ and weight $\lambda$. The subscript $\bb$ will be omitted when the Borel subalgebra is fixed.
Then the supermodule $M(\Lambda)$ is a generalized highest weight $\gg$-supermodule with highest weight $\Lambda$, and highest weight vector $[1_{\UE(\gg)} \otimes 1]$. In particular, $\UE(\nn^{-})v_{\Lambda} = M(\Lambda)$, and for each vector $v_{\Lambda}$ of weight $\Lambda$ in a $\gg$-supermodule $M$ satisfying $\nn^{+}v_{\Lambda} = 0$, there exists a uniquely determined $\gg$-supermodule morphism $M(\Lambda) \to M$ sending $[1_{\UE(\gg)} \otimes 1]$ to $v_{\Lambda}$.

The following properties of highest weight $\gg$-supermodules follow directly from their realization as quotients of Verma supermodules.

\begin{lemma}
    Let $M$ be a highest weight $\gg$-supermodule with highest weight $\Lambda \in \hh^{\ast}$. Then the following assertions hold: 
    \begin{enumerate}
        \item[a)] $M$ is a \emph{weight supermodule}.
\item[b)] For all weights $\lambda$ of $M$, we have $\dim(M^{\lambda}) < \infty$, while $\dim(M^{\Lambda})=1$.
\item[c)] Any nonzero quotient of $M$ is again a highest weight supermodule.
\item[d)] $M$ has a unique maximal subsupermodule and unique simple quotient. In particular, $M$ is indecomposable.
\item[e)] Any two simple highest weight $\gg$-supermodules with highest weight $\Lambda$ are isomorphic.
    \end{enumerate}
\end{lemma}

\subsection{Unitarizable supermodules and real forms} \label{subsec::unitarity_real_forms}
In this subsection, we introduce unitarizable supermodules over $\gg$, which in turn are defined with respect to real forms of $\gg$. Our approach and notation are based in particular on \cite{Carmeli_Fioresi_Varadarajan_HW, Chuah_Fioresi_real, Fioresi_real_forms}, which are to be considered the main sources for this subsection.

Let $V = V_{0} \oplus V_{1}$ be a complex super vector space. A \emph{super Hermitian form} on $V$ is a sesquilinear map $\bracket : V \times V \to \CC$, linear in the first variable and conjugate-linear in the second, such that 
\begin{align}
\bra v,w \ket = (-1)^{p(v)p(w)}\overline{\bra w,v\ket}
\end{align}
for all homogeneous $v,w \in V$. Additionally, we assume that $\bracket$ is consistent, meaning $\bra v,w\ket = 0$ whenever $p(v) \neq p(w)$. 

A super Hermitian form $\bracket$ can be decomposed as $\bracket = \bracket_0 + i\bracket_1$, where $\bracket_s = (-1)^s\bracket\vert_{V_s \times V_s}$ for $s \in \ZZ_2$. Here, $\bracket_0$ and $\bracket_1$ are ordinary Hermitian forms on $V_0$ and $V_1$, respectively. The form $\bracket$ is called \emph{non-degenerate} if $\bracket_0$ and $\bracket_1$ are non-degenerate, and \emph{positive definite} if $\bracket_0$ and $\bracket_1$ are positive definite. Furthermore, $\bracket$ is termed \emph{super positive definite} if $\bracket_0$ is positive definite and $\bracket_1$ is negative definite. A super positive definite super Hermitian form $\bracket$ is referred to as a \emph{Hermitian product} on $V$.

Let $V$ be a complex super vector space and $\bracket$ a super Hermitian form on $V$. For any endomorphism $T \in \End_{\CC}(V)$, we define its \emph{adjoint} $T^{\dagger}$ by the relation
\begin{align}
\bra Tv,w\ket = (-1)^{p(v)p(T)}\bra v,T^{\dagger}w\ket
\end{align}
for all homogeneous $v,w \in V$, and then defined on all of $V$ by linearity.

The concept of unitarity for supermodules over $\gg$ is defined relative to real forms $\gg^{\RR}$, which we will establish. For any $s \in \{2,4\}$, we denote by $\aut_{2,s}^{\RR}(\gg)$ the set of automorphisms $\theta$ of $\gg$  considered as a real super vector space such that $\theta \vert_{\gg_{0,1}}\neq \id_{\gg_{0,1}}$ and
\begin{align}
\theta^{2}\big\vert_{\gg_{0}} = \id_{\gg_{0}}, \quad \theta^{2}\big\vert_{\gg_{1}} = \begin{cases}
    + \id_{\gg_{1}} \quad & s = 2, \\ 
    - \id_{\gg_{1}} \quad & s = 4.
\end{cases}
\end{align}
We set 
\begin{equation}
    \begin{aligned}
        \aut_{2,4}(\gg) &:= \{ \theta \in \aut_{2,4}^{\RR}(\gg) : \theta \ \text{is} \ \CC\text{-linear}\}, \\ 
        \autbar_{2,s}(\gg) &:= \{ \theta \in \aut_{2,s}^{\RR}(\gg) : \theta \ \text{is} \ \text{conjugate-linear}\}.
    \end{aligned}
\end{equation}

 A \emph{real structure} on $\gg$ is a conjugate-linear Lie superalgebra morphism $\phi : \gg \to \gg$ such that $\phi \in \autbar_{2,2}(\gg)$. Equivalently, $\phi$ is a conjugate-linear involution. The space of fixed points $\gg^{\phi}$ is called \emph{real form} of $\gg$. 
 
 However, the real forms $\autbar_{2,2}(\gg)$ of $\gg$ and the set $\aut_{2,4}(\gg)$ are related. 

\begin{proposition}[{\cite{Fioresi_real_forms}}] \label{prop::real_form_omega} There exists a unique $\omega \in \overline{\aut}_{2,4}(\gg)$ (up to inner automorphisms of $\gg$), a positive system $\Phi^{+}$ and suitable root vectors $e_{\pm \alpha}$ for $\alpha \in \Phi^{+}$ such that 
\[
\omega(e_{\pm \alpha}) = - e_{\mp \alpha} \ \forall \alpha \ \text{even simple}, \quad \omega(e_{\pm \alpha}) = \pm e_{\mp \alpha} \ \forall \alpha \ \text{odd simple}.
\]
Moreover, the following two assertions hold:
\begin{enumerate}
    \item[a)] $\omega$ induces a bijection \[
\autbar_{2,2}(\gg) \setminus \{ \theta : \theta\vert_{\even} = \omega\vert_{\even}\} \to \aut_{2,4}(\gg), \qquad \theta \mapsto \omega^{-1} \circ \theta.
\]
\item[b)] For the Killing form $(\cdot,\cdot)$ on $\gg$, we have $\overline{(X,Y)} = (\omega(X),\omega(Y))$ for all $X,Y \in \gg$, and $(\cdot,\omega(\cdot))$ is positive definite.
\end{enumerate}
\end{proposition}

\begin{remark}
    The positive system $\Phi^{+}$ is the \emph{distinguished positive system}, meaning that there is exactly one odd positive root that cannot be expressed as the sum of two other positive roots.
\end{remark}

Consequently, up to equivalence, there is a bijective correspondence \cite{Serganova_real, Pellegrini, Fioresi_real_forms, Chuah_real}
\begin{align}
    \{\text{real forms} \ \gg^{\RR} \ \text{of} \ \gg\} \leftrightarrow \{ \theta \in \aut_{2,4}(\gg)\},
\end{align}
where the real forms of $\gg$ are precisely the subspaces of fixed points of elements in $\aut_{2,4}(\gg)$. In what follows, we do not distinguish between isomorphic real forms, or equivalently, we do not distinguish between involutions on $\gg$ that are conjugate by $\gg$-automorphisms.

We now fix a real form $\gg^{\RR}$ of a basic classical Lie superalgebras $\gg$, \emph{i.e.}, $\gg^{\RR}$ is the subspace of fixed points of some $\theta \in \aut_{2,4}(\gg)$. We denote by $\sigma := \omega \circ \theta \in \autbar_{2,2}(\gg)$ the associated conjugate-linear involution on $\gg$ (see Proposition \ref{prop::real_form_omega} above). We say $\theta$ is a \emph{Cartan automorphism} on $\gg$ or $\gg^{\RR}$, if 
\begin{align}
    B_{\theta}(\cdot,\cdot) := -(\cdot,\theta(\cdot)) \label{eq::inner_product}
\end{align}
is an inner product on $\gg^{\RR}$. Given $\theta \in \aut_{2,4}(\gg)$, there exists a unique real form $\gg^{\RR}$ such that $\theta$ restricts to a Cartan automorphism on $\gg^{\RR}$. Conversely, any real form $\gg$ has a unique Cartan automorphisms $\theta$ \cite[Theorem 1.1]{Chuah_real}. In the following, we may assume that $\theta$ associated to $\gg^{\RR}$ is a Cartan automorphism.

With the notion of a real form established, we can now define unitarizable supermodules.

\begin{definition}
    Let $\gg^{\RR}$ be a real form of $\gg$, and let $\HH$ be a complex $\gg^{\RR}$-supermodule. The supermodule $\HH$ is called a unitarizable $\gg$-supermodule (or unitarizable $\gg^{\RR}$-supermodule) if there exists a Hermitian product $\bracket$ on $\HH$ such that $X^{\dagger} = -X$ for all $X \in \gg^{\RR}$. Explicitly, this means 
    \[
    \langle Xv,w\ket = -(-1)^{p(v)p(X)}\bra v,Xw\ket
    \]
    for all $v,w \in \HH$ and $X \in \gg^{\RR}$.
\end{definition}

The following proposition is a standard result of central importance. 

\begin{proposition} \label{prop::unitarizable_completely_reducible}
    Unitarizable $\gg$-supermodules are completely reducible $\gg$-supermodules.
\end{proposition}

The proof is similar to the Lie algebra case and will be omitted.\\ \\

\subsection{Clifford superalgebras}
\label{subsec::Clifford_superalgebras} 

\subsubsection{Clifford and exterior superalgebras.} Fix a parabolic subalgebra $\pp = \ll \ltimes \uu$, and let $\gg := \ll \oplus \ss$ be the induced decomposition of $\gg$. Recall that $\ss = \uu \oplus \ubar$, where $\ubar$ is the nilradical of the opposite parabolic subalgebra $\overline{\pp} = \ll \ltimes \ubar$. Let $T(\ss)$ denote the tensor algebra over the super vector space $\ss$, and let $I(\ss)$ be the two-sided ideal generated by elements of the form 
\begin{align}
v \otimes w + (-1)^{p(v)p(w)}w \otimes v - 2 (v,w)1_{T(\ss)}
\end{align}
for any $v, w \in \ss$ and where $(\cdot, \cdot)$ denotes the Killing form of $\gg$ restricted to $\ss$. The \emph{Clifford superalgebra} is defined as the quotient $C(\ss) := T(\ss) / I(\ss)$. In other words, if we naturally identify $\ss$ as a subspace of $C(\ss)$, the Clifford superalgebra is generated by $\ss$ with the relations
\begin{align}
vw + (-1)^{p(v)p(w)}wv = 2(v, w)1_{T(\ss)},
\end{align}
where $vw$ represents the Clifford multiplication. The Clifford superalgebra $C(\ss)$ naturally inherits a $\ZZ_{2}$-grading from the tensor superalgebra $T(\ss)$. This arises from the natural $\ZZ \times \ZZ_{2}$-grading on $T(\ss)$, where the degree of $v_{1} \otimes \dots \otimes v_{n}$ is defined as $(n, p(v_{1}) + \dots + p(v_{n}))$.

Additionally, the Clifford superalgebra is characterized by a universal property.

\begin{lemma}
    Let $\calA$ be an associative superalgebra with unit $1_{\calA}$. Assume there exists a morphism of super vector spaces $\phi: \ss \to \calA$ with $\phi(v)\phi(w) + (-1)^{p(v)p(w)}\phi(w)\phi(v) = 2(v,w)1_{\calA}$ for any $v,w \in \ss$. Then $\phi$ extends uniquely to a superalgebra morphism $\phi: C(\ss) \to \calA$, denoted by the same symbol.
\end{lemma}

The proof of the lemma is straightforward and will be omitted. As super vector spaces, the Clifford superalgebra and the exterior superalgebra are isomorphic under the Chevalley identification. The \emph{exterior superalgebra} is $\bigwedge \ss := T(\ss) /J(\ss)$, where $J(\ss)$ is the two-sided ideal generated by 
\begin{align}
v \otimes w + (-1)^{p(v)p(w)}w \otimes v
\end{align}
for any $v,w \in \ss$. In particular, under the natural embedding $\ss \hookrightarrow \bigwedge \ss$, the Clifford superalgebra is generated by $\ss$ with relations 
\begin{align}
v \wedge w + (-1)^{p(v)p(w)} w \wedge v = 0,
\end{align}
with $\wedge$ being the exterior multiplication.

On $\bigwedge \ss$, we have two natural operators. For any $v \in \ss$, let $\epsilon(v)$ denote the left exterior multiplication on $\bigwedge \ss$. Moreover, the derivation  
\be \label{eq::contraction}
\iota(v)(v_{1}\otimes \ldots \otimes v_{l}) := \sum_{k = 1}^{l}(-1)^{k-1}(-1)^{p(v)(p(v_{1})+\ldots + p(v_{k-1}))}(v,v_{k}) v_{1}\otimes \ldots \otimes \hat{v}_{k}\otimes \ldots \otimes v_{l}
\ee
on $T(\ss)$ leaves $J(\ss)$ invariant and descends to a derivation on $\bigwedge \ss$. We define $\gamma : \ss \to \End(\bigwedge \ss)$ by $\gamma(v) := \epsilon(v) + \iota(v)$, which satisfies $\gamma(v)\gamma(w) + (-1)^{p(v)p(w)}\gamma(w)\gamma(v) = 2(v,w)$ for any $v,w \in \ss$. Thus, by the universal property of the Clifford superalgebra, this realizes $\bigwedge \ss$ as a $C(\ss)$-supermodule, \emph{i.e.}, $\gamma : C(\ss) \to \End(\bigwedge \ss)$. 

\begin{theorem}[\cite{kang2021dirac}] \label{thm::quantization_map} The map $\eta : C(\ss) \to \bigwedge \ss$ with $\eta(v) := \gamma(v) 1_{\bigwedge V}$ is a super vector space isomorphism. Moreover, the inverse map is given by the quantization map $\sum_{n} Q_{n} : \bigwedge \ss \to C(\ss)$ with 
\[
Q_{n} (v_{1} \wedge \ldots \wedge v_{n}) := \frac{1}{n!}\sum_{\sigma \in S_{n}} p(\sigma;v_{1}, \ldots, v_{n}) v_{\sigma(1)}\ldots v_{\sigma(n)},
\]
where
\[
p(\sigma; v_{1},\ldots, v_{n}) = \sgn(\sigma)\prod_{1 \leq i < j \leq n, \ \sigma^{-1}(i) > \sigma^{-1}(j)} (-1)^{p(v_{i})p(v_{j})}
\]
\end{theorem}
\begin{remark}
\label{rmk:isotropic_Q}
    If $v_1, \dots, v_n$ span an isotropic subspace of $\ss$, we have that $v_{\sigma(1)} \dots v_{\sigma(n)} = p(\sigma;v_1, \dots, v_n) v_1 \dots v_n$ and hence in this case
    \[
        Q_n(v_1\wedge \dots \wedge v_n) = v_1 \dots v_n.
    \]
\end{remark}
We will now provide an explicit realization of the Clifford algebra $C(\ss)$ which will be used in a later stage. To this end, we first fix a basis $\eta_{1},\ldots,\eta_{m}$ of $\ubar_{0}$ and a basis $x_{1},\ldots,x_{n}$ of $\ubar_{1}$. We shall denote by $\bigwedge \ubar$ the exterior superalgebra over $\ubar$ according to the previous definition, where we recall that 
\be
\eta_i \wedge \eta_j = - \eta_j \wedge \eta_i, \quad x_i \wedge x_j =  x_j \wedge x_i, \quad  x_i \wedge \eta_j = - \eta_j \wedge x_i.
\ee
We now define the $\CC$-linear operators $\{\partiale\}$ and $\{\partialo\}$ acting on $\ubar$ as
\be
\label{eq:del_action}
\partiale (x_{k}) := -\delta_{ik}, \quad \partiale (\eta_{l}) := 0, \quad \partialo (x_{k}) := 0, \quad \partialo (\eta_{l}) := \delta_{jl}.
\ee

Note that the unexpected minus sign appearing in the above equation is justified by the choice $(\ou_{i},u_{j}) = \delta_{ij}$ together with the identification made.
Further, upon identifying the above operators with the basis elements in $\uu$, \emph{i.e.}, $\partialo \leftrightsquigarrow {\frac{1}{2}} \bar{\eta}_i$ and $\partiale \leftrightsquigarrow {\frac{1}{2}} \bar{x}_i$, there is a natural action of $\ss = \uu \oplus \bar \uu$ on $\bigwedge \ubar$, where $\ubar$ acts by multiplication operators $\eta_j \wedge \cdot$, $x_i \wedge \cdot $, and $\ubar$ acts by $\frac{\partial}{\partial x_{i}}, \frac{\partial}{\partial \eta_{j}}$ in the obvious manner. \\
Under the above identifications, using that $(\partial_{\eta_i} , \eta_j) = (\eta_j , \partial_{\eta_i}) = \delta_{ij}$ and $(\partial_{x_i} , x_j) = - (x_j, \partial_{x_i})$, a direct calculation yields the following lemma.


\begin{lemma} \label{lemm::Clifford_differential} The Clifford superalgebra $C(\ss)$ is isomorphic to the superalgebra generated by $\{x_{i}, \frac{\partial}{\partial x_{j}} : 1 \leq i,j \leq m\}$ and $\{\eta_{i}, \frac{\partial}{\partial \eta_{j}} : 1 \leq i,j \leq n\}$ with quadratic relations
\begin{align*}
 \begin{split}
    &\quad x_i x_j - x_j x_i = 0, \quad \eta_i \eta_j + \eta_j \eta_i = 0, \quad 
    \frac{\partial}{\partial x_i} x_j - x_j \frac{\partial}{\partial x_i} =
    -\delta_{ij}, \quad
    \frac{\partial}{\partial \eta_i} \eta_j + \eta_j \frac{\partial}{\partial
    \eta_i} = \delta_{ij},\\
    &x_i \eta_j + \eta_j x_i = 0, \quad x_i \frac{\partial}{\partial \eta_j} +
    \frac{\partial}{\partial \eta_j} x_i = 0, \quad \frac{\partial}{\partial
    x_i} \eta_j + \eta_j \frac{\partial}{\partial x_i} = 0, \quad \frac{\partial}{\partial x_i}
\frac{\partial}{\partial \eta_j} + \frac{\partial}{\partial
    \eta_j}\frac{\partial}{\partial x_i} = 0.
\end{split}
\end{align*}
\end{lemma}
\subsubsection{Embedding of $\ll$ into $C(\ss)$.} Next, we aim at constructing an explicit embedding of $\ll$ into the Clifford algebra $C(\ss)$.
To this end, we recall that we fixed a parabolic subalgebra $\pp = \ll \ltimes \uu$ such that $\gg$ decomposes as $\gg = \ll \oplus \ss$ with respect to $(\cdot,\cdot)$. The restriction of $(\cdot,\cdot)$ to $\ll$ and $\ss$, denoted by $(\cdot,\cdot)_{\ll}$ and $(\cdot,\cdot)_{\ss}$ respectively, is still non-degenerate. We define the \emph{orthosymplectic superalgebra} of $\ss$ to be 
\be
\osp(\ss) := \{ T \in \End(\ss) : (T(v),w)_{\ss} + (-1)^{p(T)p(v)}(v,T(w))_{\ss} = 0 \ \text{for all} \ v,w \in \ss\},
\ee
equipped with the natural supercommutator.

\begin{lemma} \label{lemm::l_in_osp}
    The adjoint action of $\ll$ on $\ss$ induces a morphism of superalgebras $\nu : \ll \to \osp(\ss)$.
\end{lemma}

\begin{proof}
    The natural adjoint action of $\ll$ on $\ss$ defines a representation 
$
\nu : \ll \to \osp(\ss), 
$
since for any $X \in \ll$ and any $u,u' \in \ss$ we have
\begin{align*}
    (\ad_{X}u,u') &= ([X,u],u') = (X,[u,u']) = -(-1)^{p(u)p(u')}(X,[u',u]) \\ &= -(-1)^{p(u)p(u')}([X,u'],u) = -(-1)^{p(u)p(u')+p(u)p([X,u'])}(u,\ad_{X}u') \\ &= -(-1)^{p(X)p(u)}(u,\ad_{X}u'),
\end{align*}
which concludes the verification.\end{proof}

For the orthosymplectic $\ZZ_{2}$-graded representation $\nu : \ll \to \osp(\ss)$, we define the \emph{moment map} $\mu : \ss \times \ss \to \ll$ to be the bilinear map given by 
\be
(x, \mu(v,w))_{\ll} = (\nu(x)v,w)_{\ss}
\ee
for all $v,w \in \ss$ and $x \in \ll$. The moment map is even and skew-supersymmetric, hence it descends to a map defined on $\bigwedge^2 (\ss)$. Post-composing with $\nu : \ll \rightarrow \osp(\ss)$, yields a map $\mu : \bigwedge^2 (\ss) \rightarrow \osp (\ss)$, which we call the moment map associated to $\ll$, and denote by the same symbol as above with a mild notational abuse. 

\begin{proposition}[{\cite[Proposition 2.13]{Meyer}}] \label{prop::moment_map}
    The moment map $\mu : \bigwedge \nolimits^{2}(\ss) \to \osp(\ss)$ associated to $\ll$ satisfies
    \[
    \mu(x,y)(z) = (y,z)_\ss x - (-1)^{p(y)p(z)}(x,z)_\ss y
    \]
    for all $x,y,z \in \ss$. Moreover, $\mu$ is an isomorphism of super vector spaces, and it satisfies for any $T \in \osp(\ss)\subset \ss \otimes \ss^\ast$
    \[
    \mu^{-1}(T) = \frac{1}{2} \sum_{i = 1}^{2s}T(e_{i}^{\ast})\wedge e_{i}, \quad Q_{2}(\mu^{-1}(T)) = \frac{1}{4} \sum_{i = 1}^{2s}(T(e_{i}^{\ast})e_{i} - (-1)^{p(e_{i})p(T(e_{i}^{\ast}))}e_{i}T(e_{i}^{\ast})),
    \]
    where $\{e_{i}\}$ is a basis of $\ss$ with dual basis $\{e^{\ast}_{i}\}$ and $s = m+n$.
\end{proposition}
Note that in the previous proposition one looks at $\osp(\ss)$ as a certain subset of $\ss \otimes \ss^\ast$, and it makes sense to consider an action of $T \in \osp(\ss) $ on ${e_i^\ast}$, upon using $(\cdot, \cdot)_\ss$.  \\
We now define $\nu_{\ast} : \ll \to \bigwedge \nolimits^{2}(\ss)$ as the composition of $\nu$ and the inverse of the moment map $\mu^{-1}$: in the remainder of this section we will provide an explicit characterization of this map, that will be used later on in the paper. 


Recall that $\ss = \ss_0 \oplus \ss_1$ with $\ss_0 = \uu_0 \oplus \ubar_0$ and
$\ss_1 = \uu_1 \oplus \ubar_1$. We let $b_1, \dots, b_{s_0}$ be a basis
of $\uu_0$, and $\psi_1, \dots, \psi_{s_1}$ be a basis of $\uu_1$, and accordingly we denote by $\overline{b}_{1}, \dots, \overline{b}_{s_{0}}$ and $\overline{\psi}_{1}, \dots, \overline{\psi}_{s_1}$ the bases of $\overline{\uu}_0$ and $\overline{\uu}_1$ respectively, so that one has
$(\overline{b}_i, b_j) = \delta_{ij}$, $(\overline{\psi}_i, \psi_j) = \delta_{ij}$ and $(\ss_0, \ss_1) = 0$. In particular, note $(\overline{b}_{i},b_{j}) = (b_{j},\overline{b}_{i})$ and $(\overline{\psi}_{i},\psi_{j}) = -(\psi_{j},\overline{\psi}_{i})$.

We set $\ss^{\ast} = \underline{\mbox{Hom}}(\ss, \CC)$, where $\underline{\Hom}(\cdot, \cdot)$ denotes the inner Hom, that is the set of all linear maps from $\ss$ to $\CC$. We equip $\ss^{\ast}$ with the natural $\ZZ_{2}$-grading, and we define $x^{\ast}(y) := (x,y)_{\ss}$ for any $x^{\ast} \in \ss^{\ast}$ and $y \in \ss$. 
We may identify $\ss$ with $\ss^{\ast}$ under $(\cdot, \cdot)_{\ss}$. A direct calculation yields the following relations:
\begin{align} \label{eq::dual_basis}
    b_{i}^{\ast} &= \overline{b}_{i}, \quad (\overline{b}_{i})^{\ast} = b_{i}, \quad \psi_{j}^{\ast} = \overline{\psi}_{j}, \quad (\overline{\psi}_{j})^{\ast} = -\psi_{j}, \qquad 1 \leq i \leq s_{0}, \ 1 \leq j \leq s_{1}.
\end{align}
Upon decomposing $\ss = \uu \oplus \ubar$, Equation \eqref{eq::dual_basis} can be rewritten as 
\be \label{eq::dual_general}
u^{\ast} = \ou, \qquad (\ou)^{\ast} = (-1)^{p(\ou)}u = (-1)^{p(u)}u
\ee
for all $u \in \uu$ and $\ou \in \ubar$.

We first identify a suitable basis of $\osp(\ss) \subset \ss \otimes \ss^{\ast} = \End(\ss)$. For that, we define the usual dual basis by $y^{\vee}(x) := \delta_{xy}$ for any basis elements $x,y$, and then extend by linearity. Given such basis elements $x, y \in \ss$, elements of the form $x\otimes y^\vee$ provide a basis of $\End(\ss)$.


For any two basis elements $z_1,
z_2 \in \ss$, a direct calculation yields
\begin{equation}
\begin{aligned}
    (x^{\ast} \otimes y^{\vee}(z_1), z_2) &= 
    \begin{cases}
        1 & \text{if } x = z_2 \text{ and } y = z_1, \\
        0 & \text{otherwise},
    \end{cases} \\
    (z_1, y^{\ast} \otimes {x}^{\vee}(z_2)) &= 
    \begin{cases}
        (-1)^{p(y)} & \text{if } x = z_2 \text{ and } y = z_1, \\
        0 & \text{otherwise}.
    \end{cases}
\end{aligned}
\end{equation}

As a consequence, we find that $x^{\ast} \otimes y^{\vee} - (-1)^{p(x)p(y)}
y^{\ast} \otimes x^{\vee} \in \osp(\ss)$. 

\comment{
Given $x_1 \otimes y_1^{\vee} -
(-1)^{p(x)p(y)} y_1 \otimes x_1^{\vee}$ and $x_2 \otimes y_2^{\vee} -
(-1)^{p(x)p(y)} y_2 \otimes x_2^{\vee}$, we compute the super-commutator applied
to some basis element $z$ of $\ss$
\begin{equation}
\begin{split}
    & [x_1 \otimes y_1^{\vee} - (-1)^{p(x_1)p(y_1)} y_1 \otimes x_1^{\vee}, x_2 \otimes y_2^{\vee} -
    (-1)^{p(x_2)p(y_2)} y_2 \otimes x_2^{\vee}](z) \\ 
    =& x_1 \cdot (\overline{y_1}, x_2)
    (\overline{y_2}, z) - (-1)^{p(x_2)(x_2)} x_1 \cdot (\overline{y_1}, y_2)
    (\overline{x_2}, z) - (-1)^{p(x_1)p(y_1)} y_1 \cdot (\overline{x}_1, x_2)
    (\overline{y_2}, z)\\
    +& (-1)^{p(x_1)p(x_2)p(y_1)p(y_2)} y_1 \cdot(\overline{x}_1, y_2)
    (\overline{x}_2, z) - (-1)^{p(x_1)p(x_2)p(y_1)p(y_2)} \left(
         x_2 \cdot (\overline{y_2}, x_1)
    (\overline{y_1}, z) \right. \\ -& \left. (-1)^{p(x_1)(x_1)} x_2 \cdot (\overline{y_2}, y_1)
    (\overline{x_1}, z) - 
    (-1)^{p(x_2)p(y_2)} y_2 \cdot (\overline{x}_2, x_1) 
    (\overline{y_1}, z)\right.\\  +& \left. (-1)^{p(x_1)p(x_2)p(y_1)p(y_2)} y_2
    \cdot(\overline{x}_2, y_1)
    (\overline{x}_1, z)
    \right)
\end{split}
\end{equation}
}

By Proposition \ref{prop::moment_map} we have an isomorphism of super vector spaces $\osp(\ss) \rightarrow \bigwedge^2 (\ss)$, that in turn induces a Lie superalgebra morphism $\osp(\ss) \rightarrow C(\ss)$. Realizing $C(\ss)$ as in Lemma \ref{lemm::Clifford_differential} and using Proposition \ref{prop::moment_map}, the basis vectors of $\osp(\ss)$ are mapped as follows (up to an overall $\frac{1}{2}$ normalization): 
\be
\begin{aligned}
    &b_i \otimes (b_j)^{\vee} - \overline{b}_j \otimes (\overline{b}_i)^{\vee} 
    \mapsto {\frac{1}{2}} \left( \del{\eta_i} {\eta_j} - {\eta_j}\del{\eta_i} \right), \; \;
     b_i \otimes (\overline{b_j})^{\vee} - b_j \otimes (\overline{b_i})^{\vee} 
    \mapsto {\frac{1}{2}} \left( \del{\eta_i} \del{\eta_j} - \del{\eta_j} \del{\eta_i} \right), \\
    &\overline{b_i} \otimes (b_j)^{\vee} - \overline{b_j} \otimes (b_i)^{\vee} 
    \mapsto {\frac{1}{2}} \left( {\eta_i} {\eta_j} - {\eta_j} {\eta_i} \right),  \;\,
    b_i \otimes (\psi_j)^{\vee} - \overline{\psi_j} \otimes (\overline{b}_i)^{\vee} 
    \mapsto {\frac{1}{2}} \left( \del{\eta_i} {x_j} - {x_j} \del{\eta_i} \right), \\
    &\overline{b_i} \otimes (\psi_j)^{\vee} - \overline{\psi_j} \otimes (b_i)^{\vee} 
    \mapsto {\frac{1}{2}} \left( {\eta_i} {x_j} - {x_j} {\eta_i} \right), \; \;
    b_i \otimes (\overline{\psi}_j)^{\vee} + \psi_j \otimes (\overline{b}_i)^{\vee} 
    \mapsto -{\frac{1}{2}} \left( \del{\eta_i} \del{x_j} - \del{x_j} \del{\eta_i} \right), \\ 
    & \overline{b}_i \otimes (\overline{\psi}_j)^{\vee} + \psi_j \otimes (b_i)^{\vee} 
    \mapsto -{\frac{1}{2}} \left( {\eta_i} \del{x_j} - \del{x_j}{\eta_i} \right), \; \;
    \psi_i \otimes (\psi_j)^{\vee} - \overline{\psi}_j \otimes (\overline{\psi}_i)^{\vee} 
    \mapsto {\frac{1}{2}} \left( \del{x_i} {x_j} + {x_j} \del{x_i} \right), \\
    & \psi_i \otimes (\overline{\psi}_j)^{\vee} + \psi_j \otimes (\overline{\psi}_i)^{\vee} 
    \mapsto - {\frac{1}{2}} \left( \del{x_i} \del{x_j} +  \del{x_j} \del{x_i} \right), \; \; 
    \overline{\psi}_i \otimes (\psi_j)^{\vee} + \overline{\psi}_j \otimes (\psi_i)^{\vee} 
    \mapsto {\frac{1}{2}} \left( {x_i} {x_j} + {x_j} {x_i} \right),
\end{aligned}
\ee
where the generators $\eta_j$ and $x_i$ together with their derivations $\partialo$ and $\partiale$, are given as in Lemma \ref{lemm::Clifford_differential}. 
As described in Lemma \ref{lemm::l_in_osp}, we have a Lie superalgebra morphism $\nu : \ll \to \osp (\ss)$, which in turn defines an embedding $\ll \hookrightarrow C(\ss)$. Since $\uu$ and
$\overline{\uu}$ are preserved by the action of $\ll$ on $\ss$, the image of $\ll$ is contained in the span

\begin{multline}
    \bigg\langle \eta_i \frac{\partial}{\partial \eta_j} - \frac{\partial}{\partial \eta_j} \eta_i, \ 
    \eta_i \frac{\partial}{\partial x_k} - \frac{\partial}{\partial x_k} \eta_i, \\
    x_k \frac{\partial}{\partial \eta_i} - \frac{\partial}{\partial \eta_i} x_k, \ 
    x_l \frac{\partial}{\partial x_k} + \frac{\partial}{\partial x_k} x_l : 
    1 \leq k,l \leq s_1, \ 1 \leq i,j \leq s_0 \bigg\rangle.
\end{multline}

Concretely, for any $X \in \ll$, we have $[X, \uu] \subset \uu$ and $[X,\ubar] \subset \ubar$, such that 
\begin{equation}
\begin{split}
\label{eq:X_bracket_expansion}
[X,\bbar_{i}] &= \sum_{k = 1}^{s_0}\alpha_{ik}(X)\bbar_{k} + \sum_{l = 1}^{s_{1}}\beta_{il}(X)\overline{\psi}_{l}, \\ 
[X,\psibar_{j}] &= \sum_{k = 1}^{s_0}\gamma_{jk}(X)\bbar_{k} + \sum_{l = 1}^{s_{1}}\delta_{jl}(X)\psibar_{l}.
\end{split}
\end{equation}
for some complex coefficients $\alpha_{ik}(X), \beta_{il}(X), \gamma_{jk}(X)$ and $\delta_{jl}(X)$, that can be determined by applying $(\cdot, {b}_{k})$ and $(\cdot,{\psi}_{l})$ on both sides of the previous expressions as to get
\begin{equation} \label{eq::coefficients}
\begin{aligned}
    \alpha_{ik}(X) &= (X,[\bbar_{i}, {b}_{k}]), \quad \beta_{il}(X) = (X,[\bbar_{i},{\psi}_{l}]), \\
    \gamma_{jk}(X) &= (X,[\psibar_j, {b}_{k}]), \quad \delta_{jl}(X) = (X,[\psibar_{j}, {\psi}_{l}]).
\end{aligned}
\end{equation}
Finally, letting $X_i$ be a basis for $\ll$, we let the structure constant $f^k_{ij}$ and $c^{k}_{ij}$ be defined by
\be
    \begin{split}
        [X_i, \bbar_j] = \sum\limits_k f^k_{ij} \bbar_k + \sum\limits_{k'} c^{k'}_{ij}
        \overline{\psi}_k, 
    \end{split}
\ee
and similarly, we let $g^k_{ij}$ and $d^{k}_{ij}$ be defined by 
\be
    \begin{split}
        [X_i, \psibar_j] = \sum\limits_{k} g_{ij}^{k} \bbar_{k} +
        \sum\limits_{k'} d^{k'}_{ij} \psibar_{k'}. 
    \end{split}
\ee
Upon using these, if follows that
\be
\begin{split}
    \phantom{aaaa} \ad_{X_i} =& \sum\limits_{k = 1}^{s_0} \left( \sum\limits_{j = 1}^{s_0} f^k_{ij} (\overline{b}_k \otimes (\overline{b}_j)^{\vee} -
    {b_j} \otimes ({b_k})^{\vee}) + \sum\limits_{j' = 1}^{s_1}
    g^k_{ij'} (\bbar_k \otimes
    (\overline{\psi}_{j'})^{\vee} + {\psi}_{j'} \otimes ({b}_k)^{\vee}) \right) \\ 
    &+ \sum\limits_{k' = 1}^{s_1} \left( \sum\limits_{j = 1}^{s_0}
    c^{k'}_{ij}(\psibar_{k'} \otimes (\bbar_j)^{\vee} - 
        b_j \otimes
    (\psi_{k'})^{\vee}) + \sum\limits_{j' = 1}^{s_1}
d^{k'}_{ij'} (\psibar_{k'} \otimes
    (\psibar_{j'})^{\vee} - {\psi}_{j'} \otimes ({\psi}_k)^{\vee})
    \right)
\end{split}
\ee
In turn, upon the identifications above, one has 
    \be
\begin{split}
    \nu_{*}(X_i) =& \frac{1}{2} \sum\limits_{k = 1}^{s_0} \left( \sum\limits_{j = 1}^{s_0}
    f^k_{ij} (\eta_k \del{\eta_j} -
    \del{\eta_j} \eta_k) -  \sum\limits_{j' = 1}^{s_1} g^k_{ij'} (\eta_k 
    \del{x_{j'}} - \del{x_{j'}} \eta_k) \right) \\ 
    &+ \frac{1}{2} \sum\limits_{k'} \left( \sum\limits_{j = 1}^{s_0}
    c^{k'}_{ij}(x_{k'} \del{\eta_j} - \del{\eta_j} x_{k'}) -  \sum\limits_{j' =
    1}^{s_1}d^{k'}_{i{j'}} (x_{k'} \del{x_{j'}} + \del{x_{j'}} x_{k'})
    \right).
\end{split}
\ee
In turn, for general $X \in \ll$, upon using Equation \eqref{eq::coefficients} one gets the following expression:
\be \label{eq::nu_ast_X}
\begin{split}
\nu_*(X) &= 
\sum_{k = 1}^{s_0} \Bigg( 
    \sum_{j = 1}^{s_0} (X, [\bbar_j, {b}_k]) \eta_k \del{\eta_j} 
    - \sum_{j' = 1}^{s_1} (X, [\psibar_{j'}, {b}_k]) \eta_k \del{x_{j'}}
\Bigg) \\
&+ \sum_{k' = 1}^{s_1} \Bigg( 
    \sum_{j = 1}^{s_0} (X, [\bbar_j, {\psi}_{k'}]) x_{k'} \del{\eta_j}
    - \sum_{j' = 1}^{s_1} (X, [\psibar_{j'}, {\psi}_{k'}]) x_{k'} \del{x_{j'}}
\Bigg) \\
&+\frac{1}{2}\left( -\sum_{k = 1}^{s_0} (X, [\bbar_k, {b}_k]) 
    + \sum_{k' = 1}^{s_1} (X, [\psibar_{k'}, {\psi}_{k'}]) \right).
    \end{split}
\ee
Denoting the bases of $\uu$ and $\ubar$ by $u_{1}, \ldots, u_{s}$ and $\ou_{1}, \dots, \ou_{s}$ respectively, we may rewrite Equation \eqref{eq::nu_ast_X} in the following compact fashion.
\begin{lemma} \label{lemm::form_nu}
    The map $\nu_\ast : \ll \rightarrow \bigwedge^2 (\ss) \subset C(\ss)$ is given by 
\begin{equation*}
\label{eq:form_nu}
\begin{split}
\nu_{\ast}(X) &= \frac{1}{2} \sum\limits_{j,k = 1}^{s} (X, [\ou_j, u_k])(-1)^{p(u_j)} \ou_k u_j + \begin{cases}
        \rho^\uu(X) \ &\text{if} \ X \in \hh, \\
        0 \ &\text{else.}
    \end{cases}\\
    &= \frac{1}{2}\sum\limits_{j,k = 1}^{s} (X, [u_k, \ou_j] )(-1)^{p(u_j)} u_j \ou_k - \begin{cases}
        \rho^\uu(X) \ &\text{if} \ X \in \hh, \\
        0 \ &\text{else.}
    \end{cases}
\end{split}
\end{equation*}
\end{lemma}
\begin{proof} We consider the first equality. For that, it is immediate that
    \begin{align*}
    &\frac{1}{2}\left( -\sum_{k = 1}^{s_0} (X, [\bbar_k, {b}_k]) 
    + \sum_{k' = 1}^{s_1} (X, [\psibar_{k'}, {\psi}_{k'}]) \right) = -\frac{1}{2} \sum\limits_{k = 1}^s (-1)^{p(u_k)} (X, [\ou_k, u_k]),
    \end{align*}
    and it remains to prove the equality
    \[
    -\frac{1}{2} \sum\limits_{k = 1}^s (-1)^{p(u_k)} (X, [\ou_k, u_k]) =  \begin{cases}
        \rho^\uu(X) \ &\text{if} \ X \in \hh, \\
        0 \ &\text{else.}
    \end{cases}
    \]
    
    To this end, let $\alpha_{1},\dots,\alpha_{s}$ denote the set of positive roots such that $\uu = \bigoplus_{k=1}^{s}\gg^{\alpha_{k}}$. In particular, we have $\ubar = \bigoplus_{k=1}^{s}\gg^{-\alpha_{k}}$. By Proposition \ref{prop::structure_theory_BCLSA}, the weight spaces of $\gg$ are one-dimensional.

    We may assume, without loss of generality, that $u_{i} \in \gg^{\alpha_{i}}$ and $\ou_{j} \in \gg^{-\alpha_{j}}$. Then, by Proposition \ref{prop::structure_theory_BCLSA}, we have $[\ou_{i},u_{i}] \in \hh$, and $(X,[\ou_{i},u_{i}])$ is trivial unless $X \in \hh$ by consistency. If $X \in \hh$, the invariance of $(\cdot,\cdot)$ and the root space decomposition yield: 
\[
    (X,[\ou_{i},u_{i}]) = ([X, \ou_i], u_i) = -(\alpha_i(X) \ou_i, u_i) = -\alpha_i (X) (\ou_i, u_i) = - \alpha_i (X), 
\]
such that 
\[
-\frac{1}{2}\sum_{k=1}^{s}(-1)^{p(u_{i})}(X,[\ou_{k},u_{k}]) = \frac{1}{2} \sum_{k=1}^{s_{0}}\alpha_{k}(X) - \frac{1}{2} \sum_{l=s_{0}+1}^{s}\alpha_{l}(X) = \rho^{\uu_{0}}(X) - \rho^{\uu_1}(X) = \rho^{\uu}.
\]

Finally, the second equality is a straightforward computation using the definition of the Clifford superalgebra $C(\ss)$ and will be omitted.
\end{proof}

\section{Cubic Dirac operators and Dirac cohomology}

In this section, we briefly introduce cubic Dirac operators $D(\gg, \ll)$ for Levi subalgebras $\ll$ of basic classical Lie superalgebras $\gg$ and discuss their main properties. Later on, in Theorem \ref{theorem::decomp_Dirac} and Lemma \ref{lemm::square_C_Cbar}, we prove that the cubic Dirac operator $D(\gg, \ll)$ possesses a convenient decomposition into $\ll$-invariant summands, which will be investigated in a later section. Next, in Subsection \ref{subsec::DiracCohomo}, we introduce the oscillator supermodules $M(\ss), \overline{M}(\ss)$ and characterize it as an $\ll$-supermodule in view of Lemma \ref{lemm::form_nu}. Given any $\gg$-supermodule $M$, there is a natural action of the cubic Dirac operator $D(\gg, \ll)$ on $M\otimes \overline{M}(\ss)$, which allows us to introduce a cohomology theory $M \mapsto H_{\Dirac} (M)$, called the Dirac cohomology of the supermodule $M$, see Definition \ref{def::Dirac_cohom}. In the last part of the section, we focus on the Dirac cohomology of supermodules admitting an infinitesimal character. In particular, we prove a super-analog of the Casselmann--Osborne Lemma in Theorem \ref{thm::Casselmann_Osborne}. Finally, in the last subsection, we briefly discuss some homological properties of Dirac cohomology.

\subsection{Cubic Dirac operators} \label{subsec::Cubic_Dirac_operators}

\subsubsection{Definition and first properties.}
Fix a parabolic subalgebra $\pp = \ll \ltimes \uu$, and let $\gg := \ll \oplus \ss$ be the induced decomposition of $\gg$. Recall $\ss = \uu \oplus \ubar$ is even dimensional. For convenience, let $2s = 2s_{0} + 2s_{1} := \dim(\ss)$ with $2s_{0} = \dim(\ss_{0})$ and $2s_{1} = \dim(\ss_{1})$ with $s,s_{0},s_{1} \in \ZZ_{+}$. Moreover, we recall that the restriction of $(\cdot,\cdot)$ to $\ss$ gives a non-degenerate supersymmetric invariant bilinear form, denoted by $(\cdot,\cdot)_{\ss}$, that allows to identify $\ss$ and its dual space $\ss^{\ast}$. In the following, we will fix an orthogonal basis $\{X_{1},\ldots,X_{2s}\}$ of $\ss$ with dual basis $\{X_{1}^{\ast},\ldots,X_{2s}^{\ast}\}$.

Basic classical Lie superalgebras $\gg$ are examples of \emph{quadratic Lie superalgebras} \cite{kang2021dirac}, which have the property that there exists a unique element $\phi \in \bigl(\bigwedge \nolimits ^{3}\gg\bigr)_{0}$, called \emph{fundamental} $3$\emph{-form}, such that the following holds for all $X,Y,Z \in \gg$:
\begin{enumerate}
    \item[a)] $(\phi, X \wedge Y \wedge Z) = - \frac{1}{2}([X,Y],Z)$.
    \item[b)] $[X,Y] = 2 \iota(X)\iota(Y)\phi$.
    \item[c)] $\phi^{2} = \frac{1}{24} \str \ad(\Omega_{\gg})$.
\end{enumerate}
In a), we extended the supersymmetric non-degenerate invariant bilinear form $(\cdot,\cdot)$ for $\gg$ to $\bigwedge \nolimits^{3} \gg$. The fundamental $3$-form $\phi$ is uniquely determined by its projection $\phi_{\ss}$ to $\ss$, along the decomposition $\gg = \ll \oplus \ss$ \cite[Remark 4.1]{kang2021dirac}. With respect to the above fixed basis, $\phi_{\ss}$ reads
\be
\phi_{\ss} = -\frac{1}{12} \sum_{1 \leq i,j,k \leq 2s}(-1)^{p(X_{i})p(X_{j})+p(X_{k})p(X_{k})} ([X_{i},X_{j}],X_{k}) X_{i}^{\ast}\wedge X_{j}^{\ast} \wedge X_{k}^{\ast},
\ee
and it lies in $(\bigwedge\nolimits ^{3} \ss)_{0}^{\ll}$, meaning that $\Phi_{\ss}$ is $\ll$-invariant under the natural action given by the commutator \cite[Section 4]{kang2021dirac}. Since the $X_{i}$'s are orthogonal, we also have
\be
\phi_{\ss} = -\frac{1}{12} \sum_{1 \leq i,j,k \leq 2s}(-1)^{p(X_{i})p(X_{j})+p(X_{k})p(X_{k})} ([X_{i},X_{j}],X_{k}) X_{i}^{\ast}X_{j}^{\ast}X_{k}^{\ast}.
\ee
in $C(\ss)$. In a fashion analogous to \cite{kostant1999cubic}, we use $\phi_\ss$ to give the following definition, as suggest in \cite[Section 5]{kang2021dirac}. 

\begin{definition}
    The \emph{cubic Dirac operator} $\Dirac$ is the element in $\UE(\gg)\otimes C(\ss)$ given by
    \[
    \Dirac := \sum_{i = 1}^{2s} X_{i} \otimes X_{i}^{\ast} + 1 \otimes \phi_{\ss}.
    \]
\end{definition}

A direct calculation shows that $\Dirac$ is independent of the choice of basis, hence the definition is well-posed. 
We now define a diagonal embedding $\alpha : \ll \to \UE(\gg) \otimes C(\ss)$ by 
\be \label{diag_emb}
X \mapsto X \otimes 1 + 1 \otimes \nu_{\ast}(X), \qquad X \in \ll.
\ee
The explicit form of $\nu_{\ast}$ is given in Lemma \ref{lemm::form_nu}. We will denote the image of $\ll$, $\UE(\ll)$ and $\mathfrak{Z}(\ll)$ under $\alpha$ by $\ll_{\Delta}, \UE(\ll_{\Delta})$ and $\mathfrak{Z}(\ll_{\Delta})$, respectively. Furthermore, the image of the quadratic Casimir $\Omega_{\ll}$ of $\ll$ will be denoted by $\Omega_{\ll,\Delta}$. Finally, the quadratic Casimir of $\gg$ will be denoted by $\Omega_{\gg}$. \\
The following results recollect some important properties of the cubic Dirac operator, namely that $D(\gg, \ll)^2$ is $\ll$-invariant and has a nice square.

\begin{lemma}[{\cite[Lemma 6.1]{kang2021dirac}}]
    The cubic Dirac operator $\Dirac$ is $\ll$-invariant under the $\ll$-action on $\UE(\gg)\otimes C(\ss)$, which is induced by the adjoint action on both factors, \emph{i.e.}, $[X, \Dirac] = X\Dirac + (-1)^{p(X)}\Dirac X = 0$ for all $X \in \ll$.
\end{lemma}

\begin{theorem}[{\cite[Theorem 1.3]{kang2021dirac}}] \label{thm::square_Dirac}
The cubic Dirac operator $\Dirac$ has square 
\[
\Dirac^{2} = \Omega_{\gg} \otimes 1 - \Omega_{\ll,\Delta} + c (1\otimes 1),
\]
where $c$ is a constant given by $c = \frac{1}{24}(\tr \ad_{\gg}(\Omega_{\gg}) - \tr \ad_{\ll}(\Omega_{\ll}))$.
\end{theorem}

For the remainder of this article, we refer to the square of the Dirac operator $\Dirac^{2}$ as the \emph{Laplace operator}, denoting it by 
\[
\Delta := \Dirac^{2}.
\]

In Theorem \ref{thm::square_Dirac}, the constant $c$ has an explicit formulation in terms of the Weyl vectors $\rho$ and $\rho^{\ll}$ by applying an argument similar to the one in \cite[Proposition 1.84]{kostant1999cubic}, and a modified formula of Freudenthal and de Vries for Lie superalgebras \cite[Theorem 1]{Meyer}. In particular, the following holds.

\begin{lemma} \label{lemm::Constant_C}
    In terms of Weyl vectors, the constant $c = \frac{1}{24}(\tr \ad_{\gg}(\Omega_{\gg}) - \tr \ad_{\ll}(\Omega_{\ll}))$ is given by 
    \[
    c = (\rho, \rho) - (\rho^{\ll}, \rho^{\ll}).
    \]
\end{lemma}

\subsubsection{Decomposition of $\Dirac$.} We now aim at decomposing the Dirac operator into smaller $\ll$-invariant pieces. To this end, let us fix a basis $u_{1}, \ldots, u_{s}$ of $\uu$, with dual basis $u_{1}^{\ast} = \overline{u}_{1}, \ldots, u_{s}^{\ast} = \ou_{s}$ of $\uu^{\ast} \cong \ubar$. Then $\ss$ has basis and dual basis given by
\begin{equation}
    \begin{aligned}
        X_{1} &= u_{1}, \ldots, X_{s} = u_{s}, \quad X_{s+1} = \ou_1, \ldots, X_{2s} = \ou_s, \\
        X_{1}^{\ast} &= \ou_1, \ldots, X_{s}^{\ast} = \ou_s, \quad X_{s+1}^{\ast} = (-1)^{p(u_{1})}u_{1}, \ldots, X_{2s}^{\ast} = (-1)^{p(u_{s})}u_{s},
    \end{aligned}
\end{equation}
where we use Equation\eqref{eq::dual_general}. In particular, 
\be
\sum_{i = 1}^{2s} X_{i} \otimes X_{i}^{\ast} = \sum_{i = 1}^{s} u_{i} \otimes \ou_{i} + \sum_{i = 1}^{s} (-1)^{p(u_{i})}\ou_{i} \otimes u_{i} =: A + \Abar,
\ee
where $A$ and $\Abar$ denote the two above summands, respectively. The following Lemma shows that also $A$ and $\Abar$ are $\ll$-invariant. 

\begin{lemma}
    The elements $A, \Abar \in \UE(\gg) \otimes C(\ss)$ are $\ll$-invariant.
\end{lemma}

\begin{proof}
    It is enough to prove the statement for $A = \sum_{i = 1}^{s} u_{i} \otimes \ou_i$, as the proof for $\Abar$ is analogous. First, note that any $X \in \ss$ can be written as 
    \[
    X = \sum_{i = 1}^{2s} (X_{i}^{\ast},X)X_{i} = \sum_{i = 1}^{2s} (X,X_{i}) X_{i}^{\ast},
    \]
    and $[\ll,\uu] \subseteq \uu$, $[\ll,\ubar] \subseteq \ubar$. Consequently, for any $X \in \ll$, we have 
    \[
    [X,u_{i}] = \sum_{j = 1}^{s} (\ou_{j}, [X,u_{i}])u_{j}, \qquad [X,\ou_{i}] =  \sum_{j = 1}^{s} ([X,\ou_{i}], u_{j})\ou_{j}.
    \]
    Fix some $X \in \ll$, and consider
    \[
    [\alpha(X),A] = \sum_{i = 1}^{s} [X,u_{i}]\otimes \ou_{i} + \sum_{i = 1}^{s} (-1)^{p(\nu_{\ast}(X))p(u_{i})}u_{i} \otimes [X,\ou_{i}].
    \]
    Using invariance of $(\cdot,\cdot)$ and supersymmetry, the first summand can be rewritten as
    \begin{align*}
        \sum_{i = 1}^{s} [X,u_{i}]\otimes \ou_{i} &= \sum_{i,j = 1}^{s} (\ou_{j},[X,u_{i}])u_{j} \otimes \ou_{i} = \sum_{i,j = 1}^{s} u_{j} \otimes (\ou_{j},[X,u_{i}])\ou_{i} \\ &= -\sum_{i,j = 1}^{s} (-1)^{p(\nu_{\ast}(X))p(u_{j})}u_{j} \otimes ([X,\ou_{j}],u_{i})\ou_{i} \\ &= - \sum_{j = 1}^{s} (-1)^{p(\nu_{\ast}(X))p(u_{j})}u_{j} \otimes [X,\ou_{j}],
    \end{align*}
    which forces $[\alpha(X),A] = 0$. We conclude that $A$ is $\ll$-invariant.
\end{proof}

Next, we decompose the fundamental $3$-form $\phi_{\ss}$. The spaces $\uu, \ubar$ are isotropic subspaces with respect to $(\cdot,\cdot)$, \emph{i.e.}, we have for any $i,j,k$ $([u_{i},u_{j}],u_{k}) = 0$ and $([\ou_i,\ou_j],\ou_k) = 0$. As a result, we may decompose $\phi_{\ss}$ as $\phi_{\ss} = a + \overline{a}$, where
\begin{equation} \label{eq::definition_a_and_abar}
\begin{aligned}
    {a} &= -{\frac{1}{4}} \sum\limits_{i,j,k = 1}^{s} (-1)^{p(u_i)p(u_j) + p(u_k) + p(u_i) + p(u_j)} 
    ([\overline{u}_i, \overline{u}_{j}], u_k) u_i \wedge u_j \wedge \overline{u}_k, \\
    \overline{a} &= -{\frac{1}{4}} \sum\limits_{i,j,k = 1}^{s} (-1)^{p(u_i)p(u_j)} 
    ([u_i, u_j], \overline{u_k}) \overline{u}_i \wedge \overline{u}_j \wedge u_k.
\end{aligned}
\end{equation}
Here, the combinatorial factor $6$ arises from summing over all permutations of $X_i, X_j, X_k$, each contributing with the same sign, while the factor $\frac{1}{2}$ accounts for summation over all pairs $i, j$. Moreover, we used that $(\ou_i)^{\ast} = (-1)^{p(u_i)} u_i$, as in Equation \eqref{eq::dual_general}. 

We now want to express $a$ and $\overline{a}$ as elements in the Clifford superalgebra $C(\ss)$. For that, we first need the following technical lemma.

\begin{lemma}
\label{lemm:l_invariant}
    Let $u_1, \ldots, u_{s}$ be a basis of $\uu$ with dual basis $\ou_1, \ldots, \ou_{s}$. Then the element
    \[
        \Xi := \sum\limits_{i,j = 1}^{s} (-1)^{p(u_i) + p(u_j)} ([\overline{u}_i, \overline{u}_j], u_i)u_j
    \]
    vanishes identically, \emph{i.e.}, $\Xi \equiv 0.$ 
\end{lemma}
\begin{proof} The vanishing of $\Xi$ will follow from $\ll$-invariance. Indeed, let $\hh$ be the Cartan subalgebra of $\gg$ and assume $\Xi$ is $\ll$-invariant. Since the $\hh$-invariant elements of $\gg$ are precisely $\hh$, and $\hh \subset \ll$, any $\ll$-invariant element must therefore be in $\hh$. Given that $\uu \cap \hh = 0$, we conclude that $\Xi = 0$.

We are then left to prove that $\Xi$ is $\ll$-invariant. For this,
    first, we show that the map $\psi: \ss^{\otimes 4} \to \uu$, defined by $x \otimes y \otimes z \otimes w \mapsto x (y, [z,w])$, is $\ll$-equivariant. Here, $\otimes$ denotes the $\ZZ_2$-graded tensor product, and $\ss \otimes \ss$ is the $\ll$-supermodule with $\ll$-action given by (see Section \ref{subsec::conventions})
    \[
    X (v \otimes w) := Xv \otimes w + (-1)^{p(X)p(v)} v \otimes Xw, \quad X \in \ll, \ v, w \in \ss.
    \]
    Let $X \in \ll$, then one computes
    \[\begin{split}
        \psi(X(x \otimes y \otimes z \otimes w)) &= 
        [X,x]([z,w],y) + x (-1)^{p(X)p(x)} \bigl(\left( [X,y], [z,w]) \right.\\ & \left. + (-1)^{p(X)p(y)} (y, [[X,z], w]) + (-1)^{p(X)(p(y) + p(z))}(y,[z, [X,w]]) \right).
        \end{split}
    \]
    As a consequence, proving equivariance reduces to check that
    \[
    \begin{split}
        &([X,y],[z,w]) + (-1)^{p(X)p(y)}(y, [[X,z],w]) + (-1)^{p(X)(p(y) + p(z))}(y,[z, [X,w]])\\
        =& ([X,y],[z,w]) + (-1)^{p(X)p(y)}(y, [X,[z,w]]) = 0,
    \end{split}
    \]
    where we used
    \[
        [X,[z,w]] = [[X,z],w] + (-1)^{p(X)p(z)}[z, [X,w]]
    \]
    by the super Jacobi identity, and
    \[
        ([X,y],[z,w]) + (-1)^{p(X)p(y)}(y, [X,[z,w]]) = 0
    \]
    by the $\ll$-invariance of $(\cdot, \cdot)$.

    Next, we construct an $\ll$-invariant element in $\ss^{\otimes 4}$, which is mapped under $\psi$ to an $\ll$-invariant element in $\uu$ by $\ll$-equivariance. 

   We first claim that the map $\tau: \overline{\uu} \to \uu^\ast$, defined by $\overline{u} \mapsto (\overline{u}, \cdot)$, is an $\ll$-equivariant isomorphism. Whilst it is clear that the map is an isomorphism of super vector spaces, it remains to show $\ll$-equivariance. For this, let $f \in \uu^\ast$, $v \in \uu$. Then we have:
    \begin{align*}
    \tau([X, f])(v) &= ([X, f], v) = -(-1)^{p(X)p(f)} (f, [X, v]) = -(-1)^{p(X)p(f)} \tau(f)([X,v]) \\ &= (X\tau(f))(v).
    \end{align*}

    Now, consider the invariant element $\operatorname{id} \otimes \operatorname{id} \in \End(\uu) \otimes \End(\uu)$. Under the identification $\End(\uu) \cong \uu \otimes \uu^\ast$ with a basis $(u_i)_{i=1, \dots, s}$ and dual basis $(u_i^\ast)_{i=1, \dots, s}$, this can be expressed as:
    \[
    \label{eq:id_times_id}
    \sum\limits_{i,j = 1, \dots, s} u_i \otimes u_i^\ast \otimes u_j \otimes u_j^\ast.
    \]
    Applying the natural $\ll$-supermodule isomorphism of $\uu \otimes \uu^\ast \otimes \uu \otimes \uu^\ast$, given by $x \otimes \nu \otimes y \otimes \mu \mapsto (-1)^{p(\nu)p(y)}x \otimes y \otimes \mu \otimes \nu$, the above is mapped to
    \[
    \sum\limits_{i,j = 1, \dots, s} (-1)^{p(u_i)p(u_j)} u_i \otimes u_j \otimes u_i^\ast \otimes u_j^\ast.
    \]
    In turn, under the map $\operatorname{id} \otimes \operatorname{id} \otimes \tau^{-1} \otimes \tau^{-1}: \uu \otimes \uu \otimes \uu^\ast \otimes \uu^\ast \to \uu \otimes \uu \otimes \overline{\uu} \otimes \overline{\uu}$, the previous becomes 
    \be
    \label{eq:invariant_element}
    \sum\limits_{i,j = 1}^{s} (-1)^{p(u_i)p(u_j)} u_i \otimes u_j \otimes \overline{u}_i \otimes \overline{u}_j,
    \ee
    which gives the desired element in $\ss^{\otimes 4}$ under inclusion. Applying now $\psi$ to \eqref{eq:invariant_element}, we obtain the $\ll$-invariant element:
    \[
    \begin{split}
        \sum\limits_{i,j = 1}^{s} (-1)^{p(u_i)p(u_j)} u_i (u_j, [\overline{u}_i, \overline{u}_{j}]) &= \sum\limits_{i,j = 1}^{s} (-1)^{p(u_i)p(u_j)}(-1)^{p(u_j)(p(u_i) + p(u_j))} ([\overline{u}_i, \overline{u}_{j}], u_j) u_i \\
        &= \sum\limits_{i,j=1}^s (-1)^{p(u_i)p(u_j) + p(u_j)} ([\overline{u}_{j}, \overline{u}_i], u_j) u_i.
    \end{split}
    \]
    Notably, $([\overline{u}_{j}, \overline{u}_i], u_j)$ is non-zero only if $p(u_j) = p(u_j) + p(u_i)$, leading to $p(u_i) = 0$. It follows that
    \[
    \Xi = \sum\limits_{i,j=1}^s (-1)^{p(u_{j}) + p(u_{i})} ([\overline{u}_j, \overline{u}_i], u_j) u_i
    \]
    is $\ll$-invariant. This finishes the proof. \end{proof}

Having this technical result available, we are ready to give the following characterization of the elements $a$ and $\overline{a}$ inside $C(\ss)$.

    \begin{lemma}
    In the Clifford superalgebra $C(\ss)$, the elements $a$ and $\overline{a}$ are given by
    \begin{align*}
{a} &= -\frac{1}{4} \sum_{1 \leq i,j \leq s} 
(-1)^{p(u_{i})p(u_{j}) + p(u_i) + p(u_j)} 
[\overline{u}_{i}, \overline{u}_{j}] \, u_{i} u_{j}, \\
\overline{a} &= -\frac{1}{4} \sum_{1 \leq i,j \leq s} (-1)^{p(u_{i})p(u_{j})}[u_{i},u_{j}]\overline{u}_{i}\overline{u}_{j}.
    \end{align*}
\end{lemma}

\begin{proof}
   We only prove the expression for ${a}$, as the proof for $\overline{a}$ is analogous. First, we note that under the quantization map of Theorem \ref{thm::quantization_map} the following holds:
   \[
   u_{i} \wedge u_j \wedge \ou_{k} \mapsto u_{i}u_{j}\ou_{k} + (-1)^{p(u_{k})}(\delta_{ik} (-1)^{p(u_{i})p(u_{j})}u_{j}-\delta_{jk}u_{i}).
   \]
   Thus, in the Clifford superalgebra $C(\ss)$ we can write
   \begin{equation*}
   \begin{split}
        {a} =& -\frac{1}{4} \sum\limits_{i,j k= 1}^s (-1)^{p(u_i)p(u_j) + p(u_i) + p(u_j) + p(u_k)} ([\overline{u}_i, \overline{u}_{j}], u_k) u_i  u_j  \overline{u}_k \\ 
        &+ \frac{1}{2} \sum\limits_{i,j = 1}^{s} (-1)^{p(u_{i}) + p(u_{j})} ([\overline{u}_i, \overline{u}_{j}], u_i) u_j \\
        =& -\frac{1}{4} \sum\limits_{i,j= 1}^s (-1)^{p(u_i)p(u_j) + p(u_i) + p(u_j)} [\overline{u}_i, \overline{u}_{j}] u_i  u_j,
   \end{split}
   \end{equation*}
    where the second summand vanishes by Lemma \ref{lemm:l_invariant}.
\end{proof}

\begin{remark} \label{rmk::different_form_a_abar}
    We can rewrite $a$ and $\bar{a}$ as
    \begin{align*}
{a} &= -\frac{1}{4} \sum_{1 \leq i,j \leq s}
(-1)^{p(u_{i})p(u_{j})} 
u_{i} u_{j}[\ou_{i}, \ou_{j}], \\
\bar{a} &= -\frac{1}{4} \sum \limits_{1 \leq i,j \leq s} (-1)^{p(u_i)p(u_j)+p(u_{i})+p(u_{j})} \ou_i \ou_j [u_i, u_j].
\end{align*}
This follows from a straightforward computation in $C(\ss)$ using the properties of $(\cdot,\cdot)$ and will be omitted.
\end{remark}

Further, a direct but lengthy calculation yields the following lemma.

\begin{lemma} \label{lemm::invariance_a_abar}
    The elements $a, \overline{a} \in C(\ss)$ are invariant under the adjoint action of $\ll$. 
\end{lemma}

Altogether, relying on the above results, we can define the following $\ll$-invariant elements $C,\Cbar \in \UE(\gg)\otimes C(\ss)$:
\be \label{eq::C_Cbar}
C := A + 1\otimes a, \qquad \Cbar := \Abar + 1\otimes \overline{a},
\ee
which can be used to decompose the cubic Dirac operator. We summarize this discussion in the following theorem, whose proof is an immediate consequence of the above Lemmas. 

\begin{theorem} \label{theorem::decomp_Dirac} The cubic Dirac operator has the following decomposition in $\ll$-invariant elements
\[
D(\gg, \ll) = C + \Cbar.
\]    
\end{theorem}

We conclude this subsection showing that each of the $\ll$-invariant summand of $D(\gg, \ll)$ is nilpotent.
\begin{lemma} \label{lemm::square_C_Cbar}
    The square of $C, \Cbar$ is 
    \[
    C^{2} = 0, \qquad \Cbar^{2} = 0.
    \]
\end{lemma}
\begin{proof}
    The idea of the proof is similar to the one of Proposition 2.6 in \cite{Dirac_cohomology_Lie_algebra_cohomology}. We define 
    \[
    E := \frac{1}{2} \sum_{i = 1}^{s} 1 \otimes u_{i}\ou_{i}
    \]
    and compute the commutator $[E,C] = [E,A] + [E, 1 \otimes a]$:
    \begin{align*}
        [E, A] &= EA-AE \\ &= \frac{1}{2}\sum_{i,j = 1}^{s} u_{i}\otimes (u_{j}\ou_{j}u_{i}-u_{i}u_{j}\ou_{j}) \\ &=
        \frac{1}{2}\sum_{i,j = 1}^{s} u_{i}\otimes (u_{j}\ou_{j}u_{i} + (-1)^{p(u_{i})p(u_{j})}u_{j}u_{i}\ou_{j}) \\ &=
        \frac{1}{2}\sum_{i,j = 1}^{s} u_{i}\otimes (u_{j}\ou_{j}u_{i} + (-1)^{p(u_{i})p(u_{j})}u_{j}(-(-1)^{p(u_{i})p(u_{j})}\ou_{j}u_{i}+ 2\delta_{ij})) \\ &= \sum_{j = 1}^{s} u_{j} \otimes u_{j},
    \end{align*}
    which equals $A$ under the natural identification $u_{j}$ with $\ou_{j} = u_{j}^{\ast}$. Analogously, $[E,1 \otimes a] = 1 \otimes a$, \emph{i.e.}, $[E,C] = C$. Moreover, a direct calculation yields $[E,\Cbar] = -\Cbar$ and therefore $[E,C^{2}] = 2C^{2}$, $[E,\Cbar^{2}] = -2\Cbar^{2}$.

   However, set $D:= \Dirac$ and $\Dbar := C - \Cbar$ such that $C = \frac{1}{2}(D+\Dbar)$ and $\Cbar = \frac{1}{2}(D-\Dbar)$. Then
   \[
   4C^{2} = (D+\Dbar)^{2} = D^{2} + \Dbar^{2} = (D-\Dbar)^{2} = 4\Cbar^{2},
   \]
since $D\Dbar + \Dbar D = [E,D^{2}] = 0$ by Theorem \ref{thm::square_Dirac}. This forces $C^{2} = \Cbar^{2} = 0$. 
\end{proof}
\subsection{Dirac cohomology}  \label{subsec::DiracCohomo}

We define Dirac cohomology with respect to a cubic Dirac operator $\Dirac$ for some parabolic subalgebra $\pp = \ll \ltimes \uu$ and study its homological properties. Furthermore, we state Vogan's Theorem to examine the Dirac cohomology of supermodules with infinitesimal character by formulating a Casselman--Osborne Lemma.  

\subsubsection{Oscillator supermodule} 
There exists a natural simple supermodule for the Clifford superalgebra $C(\ss)$, the \emph{oscillator supermodule}, which we will construct.

The super vector space $\ss$ decomposes in its even and odd parts as $\ss = \ss_{0} \oplus \ss_{1}$, where in turn $\ss_{0} = \uu_{0} \oplus \ubar_{0}$ and $\ss_{1} = \uu_{1} \oplus \ubar_{1}$. Note that these decompositions are not direct sum decompositions with respect to $(\cdot,\cdot)$. 

We consider the Clifford superalgebra
$
C(\ss) = C(\ss_{0}) \otimes C(\ss_{1}),
$
and treat the Clifford algebra $C(\ss_{0})$ and the Weyl algebra $C(\ss_{1})$ separately.

First, we consider the Clifford algebra $C(\ss_{0})$. The subspaces $\uu_{0}$ and $\ubar_{0}$ of $\ss_0$ define isotropic and complementary subspaces of $\ss_{0}$ with respect to $(\cdot,\cdot)_{\ss}$. We fix a basis $u_{1},\ldots, u_{s_0}$ of $\uu_{0}$ and $\overline{u}_{1},\ldots, \overline{u}_{s_{0}}$ of $\ubar_{0}$ such that $(\ou_{j},u_{i}) = (u_{i},\overline{u}_{j}) = \delta_{ij}$ for all $1 \leq i,j \leq s_{0}$. We may also identify $\ubar_{0}$ with the dual space $\uu_{0}^{\ast}$ under $(\cdot,\cdot)_{\ss}$, such that the dual basis $u_{i}^{\ast}$ is $\overline{u}_{i}$ (\emph{cf.}~Equation \eqref{eq::dual_general}).  

We are interested in
\be
\spin := \bigwedge \uu, \qquad \spinbar := \bigwedge \ubar_{0}.
\ee
Without loss of generality, we focus on $\spinbar$, which is relevant for later applications. However, the following discussion applies to both $\spin$ and $\spinbar$.

On $\bigwedge \ubar_0$, we have a natural action of $\ss_{0}$, where $\overline{u} \in \ubar_{0}$ acts as the left exterior multiplication $\epsilon(\overline{u})$, and $u \in \uu_{0}$ acts as the contraction $\iota(u)$ defined in Equation \eqref{eq::contraction}. By the universal property of the Clifford superalgebra $C(\ss_{0})$, this extends to an action of $C(\ss_{0})$ on $\spinbar$, which realizes $\spinbar$ as a $C(\ss_{0})$-module, called \emph{spin module}. The following lemma is standard. 

\begin{lemma} \label{lemm::spin_simple}
    The spin module $\spinbar$ is the unique simple $C(\ss_{0})$-module, up to isomorphism. Additionally, $\spinbar$ contains a highest weight vector with respect to $\Phi(\ll;\hh)^{+}$, whose weight is $\rho^{\uu_0}$.
\end{lemma}

\begin{remark}
    Equivalently, we may consider $\spinbar$ as the left-ideal in $C(\ss_{0})$ generated by the element ${u} := {u}_{1}\ldots {u}_{s_{0}}$ such that $\spinbar = (\bigwedge \ubar_{0}){u}$ and the action is given by (left) Clifford multiplication. Note that the Clifford product and the exterior product coincide on $\uu_{0}$ and $\ubar_{0}$, since they are isotropic. 
\end{remark}

The spin module $\spinbar$ is equipped with a non-degenerate Hermitian form $\bracket_{\spinbar}$, such that $u_{i}$ and $\overline{u}_{i}$ are adjoint to each other. We fix a real form $\gg^{\RR}$ of $\gg$ defined with respect to a Cartan automorphism $\theta \in \aut_{2,4}(\gg)$, such that 
\begin{align}
    B_{\theta}(\cdot,\cdot) := -(\cdot,\theta(\cdot))
\end{align}
defines an inner product on $\gg^{\RR}$ (\emph{cf.}~Section \ref{subsec::unitarity_real_forms}). We uniquely extend the inner product $B_{\theta}(\cdot,\cdot)$ to a Hermitian form on $\gg$. Restricting this form to $\ss$, we denote it by the same symbol, $B_{\theta}(\cdot,\cdot)$, by abuse of notation.  We may assume that $u_{1}, \ldots, u_{s_{0}}, \overline{u}_{1}, \ldots, \overline{u}_{s_{0}}$ is an orthonormal basis of $(\ss_{0},B_{\theta}(\cdot,\cdot))$. Then
\be
1 = B_{\theta}(u_{i},u_{i}) = -(u_{i},\theta(u_{i})) = (u_{i},u_{i}^{\ast}), \qquad 1 = B_{\theta}(\overline{u}_{i},\overline{u}_{i}) = -(\overline{u}_{i},\theta(\overline{u}_{i})) = (\overline{u}_{i},\overline{u}_{i}^{\ast}),
\ee
and $-\theta(u_{i}) = u_{i}^{\ast} = \overline{u}_{i}$ for all $1 \leq i \leq s_{0}$.

On $T^{n}(\ubar_{0})$, we consider the bilinear form 
\begin{align}
\langle v_{1} \otimes \ldots \otimes v_{n}, w_{1} \otimes \ldots \otimes w_{n}\rangle_{\bigwedge} := \sum_{\sigma \in S_{n}} \tilde{B}(p(\sigma;v_{1},\ldots,v_{n})v_{\sigma(1)}\otimes \ldots \otimes v_{\sigma(n)}, w_{1} \otimes \ldots \otimes w_{n}),
\end{align}
where 
\begin{align}
\tilde{B}(v_{1} \otimes \ldots \otimes v_{n}, w_{1} \otimes \ldots \otimes w_{n}) := \prod_{i = 0}^{n-1}B_{\theta}(v_{n-i},w_{1+i}).
\end{align}
A direct calculation shows that $\langle\cdot,\cdot\rangle_{\bigwedge}$ descends to a Hermitian form on $\spinbar$, denoted by $\langle \cdot,\cdot\rangle_{\spinbar}$ in what follows. By construction, the following holds true.

\begin{lemma}\label{lemm::adjoint_spin}
\begin{enumerate}
    \item[a)] If we consider $\spinbar$ as a super vector space with obvious $\ZZ_{2}$-grading, $\bracket_{\spinbar}$ is a super positive definite super Hermitian form, that is,
    \[
\bra v,w \ket_{\spinbar} = (-1)^{p(v)p(w)}\overline{\bra w,v\ket}_{\spinbar}, \quad v,w \in \spinbar, 
    \]
    $\bra v,w \ket_{\spinbar} = 0$ whenever $p(v) \neq p(w)$, and $\bra\cdot,\cdot\ket_{\spinbar}$ is positive definite on $\overline{S}^{\gg,\ll}_{0}$ and $-i$-times positive definite on $\overline{S}^{\gg,\ll}_{1}$.
    \item[b)] The adjoint of $u_{i}$ with respect to $\langle \cdot,\cdot\rangle_{\spinbar}$ is $\theta(u_{i}) = -\overline{u}_{i}$, and the adjoint of $\overline{u}_{i}$ is $\theta(\overline{u}_{i}) = -u_{i}$.
\end{enumerate}
\end{lemma}

Second, we consider the \emph{Weyl algebra} $C(\ss_{1})$. To this end, we note that $(\cdot,\cdot)\vert_{\ss_{1}}$ is a symplectic form on $\ss_1$, and $\ss_{1} = \uu_{1} \oplus \ubar_{1}$ is a complete polarization—that is, $\uu_{1}$ and $\ubar_{1}$ are maximal isotropic subspaces. We refer to these spaces collectively as $X$ and $Y$, assigning one to each. Further, we fix a basis $e_{1}, \ldots, e_{s_{1}}$ of $X$ with dual basis $f_{1}, \ldots, f_{s_{1}}$ of $Y$, such that the Weyl algebra $\Weyl := C(\ss_{1})$ over $\ss_{1}$ is generated by $e_{k}$ and $f_{l}$.

The Weyl algebra acts naturally on $\CC[X] \cong \Sym(X)$, where elements of $X$ acts by multiplication and the action of $Y$ is given as follows:
\be
f_{i}\cdot e_{j} := (f_{i}, e_{j})_{\ss}, \qquad 1\leq i, j \leq s_{1}.
\ee
We call $\CC[X]$ \emph{oscillator module}. Any element in $\CC[X]$ that is annihilated by all $f_{i}$ is necessarily constant. We conclude that the maximal
proper submodule of $\CC[X]$ is zero, and $\CC[X]$ is a simple module over $\Weyl$. 

\begin{lemma} \label{lemm::Weyl_Simple}
    The oscillator module $\CC[X]$ is a simple $\Weyl$-module. 
\end{lemma}

To treat $X = \uu_{1}$ and $X = \ubar_{1}$ separately, we introduce the following notation:
\be
M(\ss_{1}) := \CC[\uu_{1}] = \Sym(\uu_{1}), \qquad \overline{M}(\ss_{1}) := \CC[\ubar_{1}] = \Sym(\ubar_{1}). 
\ee

For simplicity and clarity, we consider only $\overline{M}(\ss_{1})$ in the following. The oscillator module $M(\ss_{1})$ can be treated analogously, and in particular, all results hold for $M(\ss_{1})$ as well.

We introduce an appropriate notation following Section \ref{subsec::Clifford_superalgebras}. Fix a basis $\partial_{1},\ldots, \partial_{s_{1}}$ of $\uu_{1}$, and a basis $x_{1},\ldots, x_{s_{1}}$ of $\ubar_{1}$ such that \be
(x_{k},\partial_{l}) = \frac{1}{2} \delta_{kl}.
\ee 
Then the Weyl algebra $\Weyl$ can be identified with the algebra of differential operators with polynomial coefficients in the variables $x_{1}, \dotsc, x_{s_{1}}$, by identifying $\partial_{k}$ with the partial derivative $\partial/\partial x_{k}$ for all $k = 1, \dotsc, s_{1}$. In particular, $\Weyl$ forms a Lie algebra with the following commutator relations:
\begin{align}
    [x_{k}, x_{l}]_{W} = 0, \qquad
    [\partial_{k}, \partial_{l}]_{W} = 0, \qquad
    [x_{k}, \partial_{l}]_{W} = \delta_{kl},
\end{align}
for all $1 \leq k, l \leq s_{1}$. As a Lie algebra of differential operators, the action of the Weyl algebra on $\overline{M}(\ss_{1}) := \CC[x_1,\dotsc, x_{s_{1}}]$ is natural.

We give $\overline{M}(\ss_{1})$ a $\ZZ_{2}$-grading by declaring $\overline{M}(\ss_{1})_{0}$ to be the subspace generated by homogeneous polynomials of even degree, and $\overline{M}(\ss_{1})_{1}$ to be the subspace generated by homogeneous polynomials of odd degree. 

Furthermore, $\overline{M}(\ss_{1})$ carries a Hermitian form $\langle \cdot,\cdot \rangle_{\overline{M}(\ss_{1})}$, also known as \emph{Bargmann--Fock Hermitian form} or \emph{Fischer--Fock Hermitian form}, that is uniquely determined by
\begin{align}
\langle \prod_{k=1}^{s_{1}}x_{k}^{p_{k}},\prod_{k=1}^{s_{1}}x_{k}^{q_{k}}\rangle_{\overline{M}(\ss_{1})}=\begin{cases}\prod_{k=1}^{s_{1}}p_{k}! \qquad &\text{if} \ p_{k}=q_{k} \ \text{for all} \ k, \\
0 \qquad &\text{otherwise}.\end{cases}
\end{align}
The form is positive definite and consistent, \emph{i.e.}, one has $\langle \overline{M}(\ss_{1})_{0},\overline{M}(\ss_{1})_{1}\rangle_{\overline{M}(\ss_{1})} = 0$. In particular, for any $v, w \in \overline{M}(\ss_{1})$, the generators of $ \Weyl $ satisfy the following relations for all $1 \leq k \leq s_{1}$:
\be
\langle \partial_k v, w \rangle_{\overline{M}(\ss_{1})} = \langle v, x_k w \rangle_{\overline{M}(\ss_{1})}, \qquad \langle x_k v, w \rangle_{\overline{M}(\ss_{1})} = \langle v, \partial_k w \rangle_{\overline{M}(\ss_{1})}.
\ee
We deduce the following lemma.
\begin{lemma} \label{lemm::adjoint_M(s_1)} For $\bracket_{\overline{M}(\ss_{1})}$, the adjoint of $x_{k}$ is $\partial_{k}$, and the adjoint of $\partial_{k}$ is $x_{k}$ for all $1 \leq k \leq s_{1}$.
\end{lemma}

Finally, combining the previous constructions, we define the \emph{oscillator supermodules} over $C(\ss)$ as 
\be 
M(\ss) := \spin \otimes M(\ss_{1}), \qquad \overline{M}(\ss) := \spinbar \otimes \overline{M}(\ss_{1}).
\ee 
Here, we equip $\spin$ and $\spinbar$ with the $\ZZ_{2}$-grading induced by the natural $\ZZ_{2}$-grading of $T(\uu_{0})$ and $T(\ubar_0)$, such that $\spin = \spineven \oplus \spinodd$ and $\spinbar = \Sbar^{\gg,\ll}_{0} \oplus \Sbar^{\gg,\ll}_{1}$. This makes $M(\ss)$ and $\overline{M}(\ss)$ into $C(\ss)$-supermodules by posing $M(\ss)_{0,1} := S^{\gg,\ll}_{0,1} \otimes M(\ss_{1})$ and $\overline{M}(\ss)_{0,1} := \overline{S}^{\gg,\ll}_{0,1} \otimes \overline{M}(\ss_{1})$. 

We conclude this section by describing the properties of $\overline{M}(\ss)$. The supermodule $M(\ss)$ can be treated analogously. First, we note $\overline{M}(\ss)$ is $\hh$-semisimple. We define the set of $\hh$-weights of $\overline{M}(\ss)$ by $\mathcal{P}_{\overline{M}(\ss)} := \{\mu \in \hh^{\ast} : \overline{M}(\ss)^{\mu} \neq \{0\}\}$, where $\overline{M}(\ss)^{\mu}$ is the weight space of weight $\mu$. Then the sets of $\hh$-weights of $\overline{M}(\ss)$ is more precisely
\be \label{eq::set_of_weights_M(s)}
\mathcal{P}_{\overline{M}(\ss)} =\{ \rho^{\uu} - \ZZ_{+}[A] : A \subset \Phi^{+} \setminus \Phi(\ll;\hh)^{+}\},
\ee
where $\ZZ_{+}[A] := \sum_{\xi \in A}\ZZ_{+}\xi$.

By Lemma \ref{lemm::spin_simple} and Proposition \ref{lemm::Weyl_Simple}, one immediately has the following.
 
\begin{lemma}
    The $C(\ss)$-supermodule $\overline{M}(\ss)$ is simple.
\end{lemma}

Second, $\overline{M}(\ss)$ carries a natural non-degenerate Hermitian form 
\be
\begin{split}
\langle v \otimes P,w \otimes Q \rangle_{\overline{M}(\ss)} := \langle v,w \rangle_{\spinbar}\langle P,Q \rangle_{\overline{M}(\ss_{1})}.
\end{split}
\ee
for $v\otimes P, w \otimes Q \in \overline{M}(\ss)$. The properties of the form are given in the following straightforward lemma.

\begin{lemma} \label{lemm::adjoint_M(s)} The Hermitian form $\langle \cdot,\cdot \rangle_{\overline{M}(\ss)}$ on $\overline{M}(\ss)$ is non-degenerate, supersymmetric and consistent. Moreover, the adjoint of any basis element $u \in \uu$ is
   \[
   u^{\dagger} = -(-1)^{p(u)}\overline{u}.
   \]
\end{lemma}


Finally, we consider $M(\ss)$ and $\overline{M}(\ss)$ as $\ll$-supermodules under the Lie algebra morphism $\nu_{\ast} : \ll \to C(\ss)$ introduced in Section \ref{subsec::Clifford_superalgebras}, and described explicitly in Lemma \ref{lemm::form_nu}.
As a first result, we show that $M(\ss)$ and $\overline{M}(\ss)$ are completely reducible as $\ll$-supermodules. This follows directly from Proposition \ref{prop::unitarizable_completely_reducible}, provided we show that $M(\ss)$ and $\overline{M}(\ss)$ are unitarizable $\ll$-supermodules. 

We consider unitarity with respect to a fixed Cartan automorphism $\theta \in \aut_{2,4}(\gg)$ defining the real form $\gg^{\RR}$, such that $B_{\theta}(\cdot,\cdot)$ is an inner product on $\gg^{\RR}$. In particular, $\gg^{\RR} = \ll^{\RR} \oplus \ss^{\RR}$. Then, by Lemma \ref{lemm::adjoint_spin} and the definition of $\bracket_{\overline{M}(\ss)}$, it is enough to show 
\be
\bra \nu_{\ast}(X) (v\otimes P), (w \otimes Q) \ket_{\overline{M}(\ss)} = -(-1)^{p(X)p(v)} \bra v \otimes P, \nu_{\ast}(X) (w \otimes Q)\ket_{\overline{M}(\ss)}
\ee
for any homogeneous $X \in \ll^{\RR}$ and $v\otimes P,w \otimes Q \in \overline{M}(\ss)$. Similarly, for $M(\ss)$.

However, this is immediate by the explicit form of $\nu_{\ast}(X)$ given in Lemma \ref{lemm::form_nu} and the application of Lemma \ref{lemm::adjoint_spin} and Lemma \ref{lemm::adjoint_M(s_1)}. We conclude the following proposition. 

\begin{proposition} \label{prop::completely_reducible} The supermodules $M(\ss)$ and $\overline{M}(\ss)$ are unitarizable $\ll$-supermodules. In particular, $M(\ss)$ and $\overline{M}(\ss)$ are completely reducible as $\ll$-supermodules.
\end{proposition}

As an $\ll$-supermodule, $M(\ss)$ and $\overline{M}(\ss)$ have another elegant description as the exterior superalgebra over $\uu$ and $\ubar$, respectively.

\begin{proposition} \label{prop::identification}
    There are $\ll$-supermodule isomorphisms 
    \begin{align*}
    M(\ss) &\cong \bigwedge \uu \otimes \CC_{-\rho^{\uu}} \cong \bigwedge \uu_{0} \otimes \Sym (\uu_1) \otimes \CC_{-\rho^{\uu}}, \\
    \overline{M}(\ss) &\cong \bigwedge \ubar \otimes \CC_{\rho^{\uu}} \cong \bigwedge \ubar_{0} \otimes \Sym (\ubar_1) \otimes \CC_{\rho^{\uu}}, 
    \end{align*}
    where the action of $\ll$ on $M(\ss)$ and $\overline{M}(\ss)$ is induced by $\nu_{\ast}$ as in Lemma \ref{lemm::form_nu}, and the action of $\ll$ on $\bigwedge \uu_{0} \otimes \Sym (\uu_1) \otimes \CC_{-\rho^{\uu}}$ and $\bigwedge \ubar_{0} \otimes \Sym (\ubar_1) \otimes \CC_{\rho^{\uu}}$ is induced by the adjoint action.
\end{proposition}
\begin{proof}
    We prove the statement for $\overline{M}(\ss)$. The proof for $M(\ss)$ is analogous and will be omitted.
    
    We consider $\ubar$ as an $\ll$-supermodule under the adjoint action, which naturally makes the
    exterior superalgebra $\bigwedge \ubar = \bigwedge \ubar_0 \otimes \Sym (\ubar_1)$ into an $\ll$-supermodule.
    Let $X \in \ll$ and let $\eta_{i_0} \wedge \dots \wedge \eta_{i_n} \otimes  
 x_{j_0} \dots x_{j_m} \in \bigwedge \ubar_0 \otimes \Sym (\ubar_1)$. Then $X$ acts as
 \begin{equation*}
    \begin{split}
    X(\eta_{i_0} \wedge \dots \wedge \eta_{i_n} \otimes  x_{j_0} \dots x_{j_m}) = 
&\sum\limits_{t = 0}^{n} (-1)^{t} [X, \eta_{i_k}] \eta_{i_0} \wedge \dots \hat{\eta}_{i_t} \dots \wedge \eta_{i_n} \otimes x_{j_0} \dots x_{j_m}\\
    &+ (-1)^{n + 1} \sum\limits_{t' = 0}^{m} [X, x_{j_{t'}}]\eta_{i_0} \wedge \dots \wedge \eta_{i_n} \otimes x_{j_0} \dots \hat{x}_{j_{t'}} \dots x_{j_m},
    \end{split}
 \end{equation*}
 where $\hat{x}$ and $\hat{\eta}$ indicate that the corresponding term is omitted.
 
Now, consider the action of $\ll$ on $\overline{M}(\ss)$ induced by $\nu_{\ast}$, as given explicitly in Lemma \ref{lemm::form_nu}. By Theorem \ref{thm::quantization_map}, there exists an isomorphism $Q: \bigwedge \ubar \rightarrow C(\ss)$ of super vector spaces, providing a basis of $\overline{M}(\ss)$ consisting of elements of the form 
\[
    Q(\eta_{i_0} \wedge \dots  \wedge \eta_{i_n} \otimes  x_{j_0} \dots x_{j_m}) = \eta_{i_0} \dots \eta_{i_n} x_{j_0} \dots x_{j_m},
 \]
 using isotropy of $\ubar$ (\emph{cf.}~Remark \ref{rmk:isotropic_Q}). Here, $\overline{M}(\ss)$ is regarded as a natural quotient of $C(\ss)$.  We have:
 \begin{align*} \label{eq::computation_identifcation}
 \begin{split}
    \nu_{\ast}(X) & Q (\eta_{i_0} \wedge \dots \wedge \eta_{i_n} \otimes  
 x_{j_0} \dots x_{j_m}) \\
 = &\sum\limits_{k = 1}^{s_0} \left( \sum\limits_{t = 0}^n (X, [\bbar_{i_t}, {b}_k]) \eta_k (-1)^{t}  \eta_{i_0} \dots \hat{\eta}_{i_t} \dots x_{j_m} \right.\\
 & + \left. \sum\limits_{t' = 0}^{m}  (X, [\psibar_{t'}, {b}_k]) \eta_k (-1)^{n + 1} \eta_{i_1} \dots \hat{x}_{j_{t'}} \dots x_{j_m} \right)\\
 &+ \sum\limits_{k' = 1}^{s_1} \left( \sum\limits_{t = 0}^n (X, [\bbar_{i_t}, {\psi}_{k'}]) x_{k'} (-1)^{t}  \eta_{i_1}  \dots \hat{\eta}_{i_t} \dots x_{j_m} \right. \\ &+ \left. 
 \sum\limits_{t' = 0}^{m}  (X, [\psibar_{t'}, {\psi}_{k'}]) x_{k'} (-1)^{n + 1} \eta_{i_1}  \dots \hat{x}_{j_{t'}} \dots x_{j_m} \right) \\
 &+ \rho^{\uu}(X) \eta_{i_1} \dots x_{j_n} \\
 & = Q \left(\sum\limits_{t = 0}^n (-1)^t [X, \eta_{i_{t}}] \eta_{i_1} \wedge \dots \hat{\eta}_{i_t} \dots x_{j_m} \right. \\ &+ \left.  \sum\limits_{t' = 0}^{m} (-1)^{n + 1} [X, x_{i_{t'}}] \eta_{i_1} \wedge \dots \hat{x}_{j_{t'}} \dots x_{j_m} + \rho^{\uu}(X) \eta_{i_0} \wedge \dots x_{j_m} \right),
 \end{split} 
 \end{align*}
 where $\rho^{\uu}(X)$ is only present if $X \in \hh$.
 Concerning the signs, recall, that by \eqref{eq:del_action} we have the somewhat surprising action $\del{x_i}{x_k} = -\delta_{ik}$.
Here, we use the notation of Section \ref{subsec::Clifford_superalgebras}, and the proof of Lemma \ref{lemm::form_nu} to obtain $\rho^{\uu}$ together with \eqref{eq:X_bracket_expansion} in the final step. We conclude that both actions coincide up to a twist by $\CC_{\rho^{\uu}}$. This finishes the proof. 
\end{proof}

The superspaces $\uu$ and $\ubar$ are finite-dimensional $\ll$-supermodules, and the adjoint action of $\ll$ preserves the natural $\ZZ$-grading of $\bigwedge \uu$ and $\bigwedge \ubar$. Combining Proposition \ref{prop::completely_reducible} and Proposition \ref{prop::identification}, we obtain the following corollary.

\begin{corollary}
    The $\ll$-supermodules $M(\ss)$ and $\overline{M}(\ss)$ decompose completely into finite-dimensional simple $\ll$-supermodules.
\end{corollary}

For the remainder of this article, we focus on $\overline{M}(\ss)$, as this supermodule is best suited for studying highest weight $\gg$-supermodules.

\subsubsection{Dirac cohomology: definition and first results}
For any $\gg$-supermodule $M$, the cubic Dirac operator $\Dirac$ acts naturally on the $\UE(\gg)\otimes C(\ss)$-supermodule $M\otimes \overline{M}(\ss)$ by componentwise action. In particular, as $\Dirac$ and $\ll$ commute, the kernel $\ker \Dirac$ naturally carries the structure of an $\ll$-supermodule. This gives rise to a definition of Dirac cohomology analogous to \cite{huang2005dirac}, as suggested in \cite{kang2021dirac}.

\begin{definition} \label{def::Dirac_cohom}
The \emph{Dirac cohomology} $H_{\Dirac}(M)$ of a $\gg$-supermodule $M$ is the $\ll$-supermodule
\[
H_{\Dirac}(M) := \ker \Dirac / \ker \Dirac \cap \Im \Dirac.
\]
\end{definition}

The Dirac cohomology $H_{\Dirac}(M)$ has a natural decomposition induced by the $\ZZ_{2}$-grading of $\overline{M}(\ss)$. Accordingly, we decompose the Dirac operator as $\Dirac = \Dirac^{+} + \Dirac^{-}$, where
\begin{align}
\begin{split}
\Dirac^{+} &:= \Dirac\big\vert_{M \otimes \overline{M}(\ss)_{0}} : M \otimes \overline{M}(\ss)_{0} \to M \otimes \overline{M}(\ss)_{1}, \\
\Dirac^{-} &:= \Dirac\big\vert_{M \otimes \overline{M}(\ss)_{1}} : M \otimes \overline{M}(\ss)_{1} \to M \otimes \overline{M}(\ss)_{0},
\end{split}
\end{align}
and define $\DC^{+}(M) := \mathrm{H}_{\Dirac^{+}}(M)$ and $\DC^{-}(M) := \mathrm{H}_{\Dirac^{-}}(M)$, so that 
\be
\DC(M) := \DC^{+}(M) + \DC^{-}(M).
\ee

We are particularly interested in admissible $(\gg, \ll)$-supermodules, \emph{i.e.}, $\gg$-supermodules that are $\ll$-semisimple. For these supermodules, the Dirac cohomology is $\ll$-semisimple by Proposition \ref{prop::completely_reducible}, as $\Dirac$ commutes with the action of $\ll$.

\begin{lemma} \label{lemm::completely_reducible_DC}
    Let $M$ be an admissible $(\gg, \ll)$-supermodule. Then $\DC(M)$ is a semisimple $\ll$-supermodule, that is, $\DC(M)$ is completely reducible as an $\ll$-supermodule.
\end{lemma}

Given the nice square of the Dirac operator in Theorem \ref{thm::square_Dirac}, the Dirac cohomology reveals its full potential when the supermodule under consideration possesses an infinitesimal character. This will be explored in the subsequent section. Before addressing this topic, we establish that $\Dirac$ naturally induces a cohomology on the space $(\UE(\gg) \otimes C(\ss))^{\ll}$ of $\ll$-invariants in $\UE(\gg) \otimes C(\ss)$, as shown in \cite{kang2021dirac}.

We equip $\UE(\gg) \otimes C(\ss)$ with the $\ZZ_{2}$-grading induced by $C(\ss)$, that is, $(\UE(\gg)\otimes C(\ss))_{0,1} := \UE(\gg) \otimes C(\ss)_{0,1}$. The diagonal embedding $\alpha : \ll \to \UE(\gg) \otimes C(\ss)$, as defined in equation \eqref{diag_emb}, endows $\UE(\gg) \otimes C(\ss)$ with the structure of an $\ll$-supermodule via the adjoint action. Hence, it makes sense to consider the space of $\ll$-invariants in $\UE(\gg) \otimes C(\ss)$:
\be \label{l_inv_space}
(\UE(\gg) \otimes C(\ss))^{\ll} := \{ A \in \UE(\gg)\otimes C(\ss) : [\alpha(X),A] = 0 \ \text{for all} \ X \in \ll\},
\ee
which is still endowed with the obvious $\ZZ_{2}$-grading. This enters the following construction, due to \cite{kang2021dirac}. Let $\hat d$ be the operator acting on $\UE(\gg) \otimes C(\ss)$ as 
\be
\hat{d}(A) := [\Dirac, A] = \Dirac A - (-1)^{p(A)}A \Dirac
\ee
for any homogeneous $A \in \UE(\gg) \otimes C(\ss)$. Since $\Dirac$ is $\ll$-invariant, the operator $\hat{d}$ can be restricted to an operator 
\be
d : (\UE(\gg) \otimes C(\ss))^{\ll} \rightarrow (\UE(\gg) \otimes C(\ss))^{\ll},
\ee
 which is proved to be nilpotent, \emph{i.e.}, $d^2 =0$, and hence it makes sense to define its cohomology to be the quotient  $\ker d / \Im d$, see \cite{kang2021dirac}. Recalling that we denoted by $\mathfrak{Z}(\ll_{\Delta})$ the image under the diagonal embedding $\alpha$ of $\mathfrak{Z}(\ll)$, the center of the universal enveloping superalgebra $\UE(\ll)$, one has the following characterization for the cohomology of $d$. 

\begin{thm}[{\cite[Theorem 6.2]{kang2021dirac}}] \label{thm::ker_d_center} The cohomology of $d$ is isomorphic to $\mathfrak{Z}(\ll_{\Delta})$, \emph{i.e.}, 
\[
\ker d = \mathfrak{Z}(\ll_{\Delta}) \oplus \Im d.
\]  
\end{thm}

\subsubsection{Infinitesimal characters and Dirac cohomology} \label{subsec::infchar}
We now aim to extend the above discussion to supermodules admitting an infinitesimal character, which we will now introduce.\\
Infinitesimal characters are specific algebra homomorphisms $\chi : \ZG \to \CC$, which can be formulated explicitly using the Harish-Chandra homomorphism for Lie superalgebras. To construct the Harish-Chandra homomorphism, we consider the following isomorphism of super vector spaces
\begin{equation}
\mathfrak{U}(\gg) \cong \mathfrak{U}(\hh) \oplus (\nn^{-}\mathfrak{U}(\gg) + \mathfrak{U}(\gg)\nn^{+}),
\end{equation}
which is an immediate consequence of the Poincaré--Birkhoﬀ--Witt Theorem (PBW Theorem) for $\gg$. The associated projection $p: \mathfrak{U}(\gg) \to \mathfrak{U}(\hh)$ is called the \emph{Harish-Chandra projection}, and its restriction to $\mathfrak{Z}(\gg)$ defines an algebra homomorphism:
$
p\big\vert_{\mathfrak{Z}(\gg)}: \mathfrak{Z}(\gg) \to \mathfrak{U}(\hh) \cong \text{S}(\hh)$.
The \emph{Harish-Chandra homomorphism} $\HC: \ZG \to S(\hh)$ is the composition of the algebra homomorphisms $p\big\vert_{\mathfrak{Z}(\gg)}$ with the twist $\zeta: \S(\hh)\to \S(\hh)$ defined by 
$\lambda(\zeta(f)):=(\lambda-\rho)(f)$ for all $f\in\S(\hh)$, $\lambda\in\hh^*$. For any $\Lambda \in \hh^{\ast}$, the associated map 
\be
\chi_{\Lambda}(z) := (\Lambda+\rho)(\operatorname{HC}(z))
\ee
defines an algebra homomorphism $\chi_{\Lambda} : \ZG \to \CC$, where $\rho$ denotes the Weyl vector for the fixed positive system.

\begin{definition}
We say that a $\gg$-supermodule $M$ admits an \emph{infinitesimal character} if there exists a $\Lambda \in \hh^{\ast}$ such that any element $z \in \ZG$ acts on $M$ by a scalar multiple $\chi_{\Lambda}(z) \in \CC$ of the identity. We refer to $\chi_{\Lambda}$ as the \emph{infinitesimal character} of $M$.
\end{definition}

A particularly important example of supermodules of this kind are highest weight $\gg$-supermodules. More precisely, the following holds. 

\begin{lemma}[{\cite[Lemma 8.5.3]{Musson}}] \label{lemm::action_Casimir}
    Let $M$ be a highest weight $\gg$-supermodule with highest weight $\Lambda$. Then $M$ admits an infinitesimal character $\chi_{\Lambda}$. Moreover, the quadratic Casimir $\Omega_{\gg}$ acts as the scalar multiplication by 
    $
    (\Lambda + 2\rho, \Lambda), 
    $ \emph{i.e.},
\[
\Omega_\gg =     (\Lambda + 2\rho, \Lambda) 1_{\UE(\gg)}.
\]
\end{lemma}

The Harish-Chandra homomorphism is an injective ring
homomorphism. To describe its image, we define 
$\S(\hh)^{W^{\gg}} := \{ f\in \S(\hh)\ :\ w (\lambda)(f) =
\lambda(f) \ \text{for all} \ w \in W^{\gg}, \lambda \in \hh^*\}$ and, for any $\lambda\in\hh^*$,
\be
A_{\lambda} := \{ \alpha \in \Phi_{1}^{+} : (\lambda+\rho,\alpha) = 0\} \,.
\ee
Then, the image of $\HC$ is \cite{GorelikHC, KacHC}
\be
\Im(\HC) = \bigl\{ f \in \S(\hh)^{W^{\gg}} : (\lambda+t\alpha)(f)=\lambda(f) \ 
\text{for all $t\in \CC$, $\lambda\in\hh^*$, and $\alpha \in A_{\lambda-\rho}$}\bigr\}.
\ee
As a direct consequence, we obtain $\chi_{\Lambda}= \chi_{\Lambda'}$ whenever 
\be \label{eq::relation_infinitresimal_character}
\Lambda'=w\biggl(\Lambda+\rho+\sum_{i=1}^{k}t_{i}\alpha_{i}\biggr)-\rho,
\ee
where $w \in W^{\gg}$, $t_{i}\in \CC$, and $\alpha_{1},\dotsc,\alpha_{k}\in A_\Lambda$ are linearly independent 
odd isotropic roots that satisfy $(\Lambda+\rho,\alpha_{i}) = 0$.

Having set our notation, and given the above preliminaries, we start with the following straightforward corollary of Theorem \ref{thm::ker_d_center}.

\begin{corollary} \label{cor::form_center}
    Any element $z \otimes 1 \in \UE(\gg) \otimes C(\ss)$ with $z \in \ZG$ can be written as
    \[
    z \otimes 1 = \eta_{\ll}(z) + \Dirac A + A \Dirac,
    \]
    for some $A \in (\UE(\gg) \otimes C(\ss))^{\ll}$, and some unique $\eta_{\ll}(z) \in \mathfrak{Z}(\ll_{\Delta}) \cong \mathfrak{Z}(\ll)$.
\end{corollary}

The previous Corollary \ref{cor::form_center} introduces a map $\eta_{\, \ll} : \ZG \to \mathfrak{Z}(\ll)$, which we now describe. Letting $\gg$ be decomposed as $\gg = \ll \oplus \ss$, where we note that $\hh \subset \ll$ by definition. Moreover, let $HC_{\gg,\ll}$ denote the Harish-Chandra monomorphism and $\res : S(\hh)^{W^{\gg}} \to S(\hh)^{W^{\ll}}$ be the restriction map.

\begin{proposition} \label{prop::commutative_diagram}
    The map $\eta_{\, \ll} : \ZG \to \mathfrak{Z}(\ll)$ is an algebra homomorphism. Further, the following diagram is commutative.
    \[ \begin{tikzcd}
\ZG \arrow{r}{\eta_{\ll}} \arrow[swap]{d}{HC_{\gg}} & \mathfrak{Z}(\ll) \arrow{d}{HC_{\ll}} \\%
S(\hh)^{W^{\gg}} \arrow{r}{\res}& S(\hh)^{W^{\ll}}
\end{tikzcd}
\]
\end{proposition}

\begin{proof}
    For any $z_{1},z_{2} \in \ZG$, we find $A_{1},A_{2} \in (\UE(\gg)\otimes C(\ss))^{\ll}$ such that 
    \[
    z_{1} \otimes 1 = \eta_{\ll}(z_{1}) + d(A_{1}), \qquad z_{2} \otimes 1 = \eta_{\ll}(z_{2}) + d(A_{2}).
    \]
    Multiplication yields
    \[
    z_{1}z_{2} \otimes 1 = \eta_{\ll}(z_{1})\eta_{\ll}(z_{2})+\eta_{\ll}(z_{1})d(A_{2}) + d(A_{1})\eta_{\ll}(z_{2}) + d(A_{1})d(A_{2}).
    \]
    However, by Corollary \ref{cor::form_center}, we have $d(\eta_{\ll}(z_{1})) = d(\eta_{\ll}(z_{2})) = 0$, and consequently
    \[
    z_{1}z_{2} \otimes 1 = \eta_{\ll}(z_{1})\eta_{\ll}(z_{2}) + d(\eta_{\ll}(z_{1})A_{2}+A_{1}\eta_{\ll}(z_{2})+Ad(A'))
    \]
    with $\eta_{\ll}(z_{1})A_{2}+A_{1}\eta_{\ll}(z_{2})+Ad(A') \in (\UE(\gg)\otimes C(\ss))^{\ll}$. Again, by Corollary \ref{cor::form_center}, we have
    \[
    z_{1}z_{2} \otimes 1 = \eta_{\ll}(z_{1}z_{2}) + d(A)
    \]
    for some $A \in (\UE(\gg)\otimes C(\ss))^{\ll}$, \emph{i.e.}, $\eta_{\ll}(z_{1})\eta_{\ll}(z_{2}) = \eta_{\ll}(z_{1}z_{2})$. 
    
    For the second part of the proof, we note $\hh \subset \even$, and we identify $S(\hh)^{W^{\gg}}$ and $S(\hh)^{W^{\ll}}$ with the algebraic varieties $\hh^{\ast}/W^{\gg}$ and $\hh^{\ast}/W^{\ll}$. Let $\zeta : S(\hh)^{W^{\gg}} \to S(\hh)^{W^{\ll}}$ denote the homomorphism induced by $\eta_{\ll}$ under the Harish-Chandra monomorphisms, and let $\zeta' : \hh^{\ast}/W^{\ll} \to \hh^{\ast}/W^{\gg}$ be the associated morphism of the algebraic varieties. To show that the diagram commutes, it suffices to prove that $\zeta'$ is the inclusion map, \emph{i.e.}, $\zeta'(\lambda) = \lambda$ for all $\lambda \in \hh^{\ast}/W^{\ll}$.

For any $\lambda \in \hh^{\ast}$, we define a highest weight $\gg$-supermodule $M$ as in Equation \eqref{eq::Verma_supermodules}. As stated explicitly in Proposition \ref{prop::HW_non-trivial}, $\DC(M)$ contains a non-trivial simple weight $\ll$-supermodule generated by $v_{\lambda} \otimes v_{\overline{M}(\ss)}$ of weight $\lambda + \rho^{\uu}$. However, on $\DC(M)$, we conclude from $z = \eta_{\ll}(z)$ the following:
\[
(\lambda + \rho)(z) = \big((\lambda + \rho^{\uu}) + \rho^{\ll}\big)(\eta_{\ll}(z)) = (\lambda + \rho)(\eta_{\ll}(z)),
\]
\emph{i.e.}, $\zeta'(\mu) = \mu$ for all $\mu = \lambda + \rho$. This completes the proof.
\end{proof}

As a result, we obtain a super-analog of the Casselman--Osborne Lemma (\emph{cf.}~\cite{Knapp, Milivcic}), that will play a crucial role in what follows. More precisely, we have the following theorem. 

\begin{theorem} \label{thm::Casselmann_Osborne} Let $M$ be a $\gg$-supermodule with infinitesimal character $\chi_\lambda$.
Then $z \otimes 1 \in \ZG \otimes 1$ acts as $\eta_{\,\ll}(z) \otimes 1$ on $\DC(M)$. In particular, if $V$ is an $\ll$-subsupermodule of $\DC(M)$ with infinitesimal character $\chi_{\mu}^{\ll}$ for $\mu \in \hh^{\ast}$, then $\chi_{\lambda} = \chi_{\mu}^{\ll} \circ \eta_{\ll}$. 
\end{theorem}


\begin{proof}
    Let $V$ be as above and fix some non-trivial $v \in V$. By Corollary \ref{cor::form_center}, there exists some $A \in (\UE(\gg)\otimes C(\ss))^{\ll}$ such that for any $z \in \ZG$:
    \[
    z \otimes 1 - \chi_{\lambda}(z) = (\eta_{\ll}(z) - \chi_{\mu}^{\ll}(\eta_{\ll}(z))) + d(A) + (\chi_{\mu}^{\ll}(\eta_{\ll}(z)) - \chi_{\lambda}(z)),
    \]
    where we identify $\mathfrak{Z}(\ll) \cong \mathfrak{Z}(\ll_{\Delta})$.
    Applying both sides of this identity to $v$, we obtain 
    \[
    (\chi_{\mu}^{\ll}(\eta_{\ll}(z))-\chi_{\lambda}(z))v \mod \Ker \Dirac \cap \Im \Dirac = 0.
    \]
    This concludes the proof by Proposition \ref{prop::commutative_diagram}.
\end{proof}

The previous result, combined with Proposition \ref{prop::completely_reducible}, yields the following corollary.

\begin{corollary}
    Let $M$ be a $\gg$-supermodule in $\calOp$ with infinitesimal character. Then $\DC(M)$ is a completely reducible $\ll_0$-module.
\end{corollary}

\subsubsection{Homological properties} We now briefly discuss the homological properties of the Dirac cohomology, seen as a functor from the category of $\gg$-supermodules to the category of $\ll$-supermodules. In general, it has neither a right nor a left adjoint, but it satisfies a six-term exact sequence.

\begin{lemma}
    Let $0 \to A \to B \to C \to 0$ be a short exact sequence of $\gg$-supermodules having an infinitesimal character. Then there exists a natural six-term exact sequence
\begin{equation*}
\begin{gathered}
\xymatrix@C-1em{
& \DC^{+}(A) \ar[rr] && \DC^{+}(B) \ar[dr] \\
\DC^{-}(C) \ar[ur] &&&& \DC^{+}(C) \ar[dl] \\
& \DC^{-}(B) \ar[ul] && \DC^{-}(A)  \ar[ll]
}
\end{gathered}
\end{equation*}
\end{lemma}

The proof of the above statement follows from a similar argument to the one used in the proof of Theorem 8.1 in \cite{Dirac_cohomology_Lie_algebra_cohomology}, relying on the fact that the supermodules involved have an infinitesimal character, together with Theorem \ref{thm::square_Dirac}, and that supermodule morphisms are even.

Second, we consider the \emph{Euler characteristic} of the Dirac cohomology when restricted to admissible $(\gg, \ll)$-supermodules $M$. For any such supermodule, the Dirac operator $\Dirac$ acts on $M \otimes \overline{M}(\ss)_{0}$ and $M \otimes \overline{M}(\ss)_{1}$, interchanging these two spaces. The \emph{Dirac index} $\DI(M)$ of $M$ is defined as the difference of $\ll$-supermodules
\be
\DI(M) := M \otimes \overline{M}(\ss)_{0} - M \otimes \overline{M}(\ss)_{1}.
\ee
This Dirac index $\DI(M)$ is a \emph{virtual} $\ll$-supermodule, meaning that it is an integer combination of finitely many $\ll$-supermodules -- notice that $\DI(M)$ is an element in the Grothendieck group of the Abelian category of admissible $(\gg, \ll)$-supermodules, in fact it is additive with respect to short exact sequences. In contrast, as seen above, we can use $\Dirac: M \otimes \overline{M}(\ss)_{0,1} \to M \otimes \overline{M}(\ss)_{1,0}$ to decompose the Dirac cohomology $\DC(M)$ into two parts, $\DC^{+}(M)$ and $\DC^{-}(M)$. The \emph{Euler characteristic} of $\DC(M)$ is given by the virtual $\ll$-supermodule $\DC^{+}(M) - \DC^{-}(M)$. These two virtual $\ll$-supermodules, $\DI(M)$ and the Euler characteristic, are equal if $M$ has an infinitesimal character.

\begin{proposition}\label{prop:Euler char} Let $M$ be an admissible $(\gg, \ll)$-supermodule that admits an infinitesimal character.
Then the Dirac index $\DI(M)$ is equal to the Euler characteristic of the Dirac cohomology $\DC(M)$ of $M$, \emph{i.e.}, 
\[
\DI(M) = \DC^{+}(M)-\DC^{-}(M).
\]
\end{proposition} 
\begin{proof}
  Since $M$ admits an infinitesimal character, we can decompose $M \otimes \overline{M}(\ss)$ into a direct sum of eigenspaces of $\Dirac^2$ by Theorem \ref{thm::square_Dirac} (see below in Section \ref{subsec::decomp::M_otimes_M(s)} for a proof), namely,
\[
M \otimes \overline{M}(\ss) = (M \otimes \overline{M}(\ss))(0) \oplus \sum_{r \neq 0} (M \otimes \overline{M}(\ss))(r),
\]
where $(M\otimes \overline{M}(\ss))(r)$ denotes an eigenspace with eigenvalue $r \in \CC$. The operator $\Dirac^2$ is even, so the decomposition is compatible with the decomposition into even and odd parts of $\overline{M}(\ss)$:
\[
(M \otimes \overline{M}(\ss))(r) = (M \otimes \overline{M}(\ss)_{0})(r) \oplus (M \otimes \overline{M}(\ss)_{1})(r).
\]
On the other hand, the Dirac operator $\Dirac$ preserves each eigenspace because $\Dirac$ and $\Dirac^2$ commute. However, $\Dirac$ switches parity, defining maps
\[
\Dirac(r): (M \otimes \overline{M}(\ss)_{0, 1})(r) \to (M \otimes \overline{M}(\ss)_{1, 0})(r),
\]
which are isomorphisms with inverses $1/r \Dirac(r)$ for $r \neq 0$. Consequently, we have
\[
M \otimes \overline{M}(\ss)_{0} - M \otimes \overline{M}(\ss)_{1} = (M \otimes \overline{M}(\ss)_{0})(0) - (M \otimes \overline{M}(\ss)_{1})(0).
\]
The Dirac operator $\Dirac$ acts as a differential on $\Ker(\Dirac^2)$, and the associated cohomology is the Dirac cohomology. The result follows from the Euler--Poincaré principle.
\end{proof}

\section{Dirac cohomology and highest weight supermodules} 

In this subsection, we build on the super-analog of the Casselman--Osborne Lemma (Theorem \ref{thm::Casselmann_Osborne}) as a key tool for investigating the Dirac cohomology of highest-weight $\gg$-supermodules, with a particular emphasis on finite-dimensional cases. 

Specifically, in Section \ref{subsec::generalitiesDC}, we prove that the Dirac cohomology of highest-weight $\gg$-supermodules is always non-trivial (Proposition \ref{prop::HW_non-trivial}). In Section \ref{subsec:fin_dim}, we refine this analysis to the finite-dimensional setting, where we perform explicit computations of Dirac cohomology.

More precisely, we determine the Dirac cohomology of finite-dimensional supermodules over classical Lie superalgebras of type 1 (\emph{i.e.}, $\gg = \glmn$, $\gg = A(m|n)$ or $\gg = C(n)$) with a typical highest weight (Theorem \ref{thm::DC_finite_dimensional}), and of finite-dimensional simple objects in $\mathcal{O}^\pp$ (Theorem \ref{thm::DC_OP}).


\subsection{Decomposition of $M\otimes \overline{M}(\ss)$} \label{subsec::decomp::M_otimes_M(s)}
Let $M$ be a $(\gg,\ll)$-supermodule that admits Jordan--Hölder series with simple quotients isomorphic to highest weight $\gg$-supermodules. Then any simple quotient admits an infinitesimal character $\chi_{\Lambda}$, and the quadratic Casimir $\Omega_{\gg}$ acts as the scalar multiple of $(\Lambda + 2\rho, \Lambda)$. In addition, $\overline{M}(\ss)$ is a completely reducible $\ll$-supermodule by Proposition \ref{prop::completely_reducible}. 

As a consequence, we will decompose $M \otimes \overline{M}(\ss)$ into generalized eigenspaces of the Laplace operator $\Delta = \Dirac^{2}$, given by 
\be
M \otimes \overline{M}(\ss) = \bigoplus_{x \in \CC} (M \otimes \overline{M}(\ss))(r),
\ee
where the \emph{generalized eigenspaces} are defined as 
\be
(M \otimes \overline{M}(\ss))(r) := \{ v \in M \otimes \overline{M}(\ss) : \exists n := n(v)\in \ZZ_{+} \ \text{s.t.} \ (r\cdot \id_{M \otimes \overline{M}(\ss)} - \Delta)^{n}v = 0\}.
\ee
An element $v \in (M \otimes \overline{M}(\ss))(r)$ is called a \emph{generalized eigenvector} with \emph{generalized eigenvalue} $r$. To achieve this decomposition, we first decompose $M \otimes \overline{M}(\ss)$ with respect to infinitesimal characters.

\begin{lemma} \label{lemm::generalized_character_space_decomposition}
    The $\ll$-supermodule $M \otimes \overline{M}(\ss)$ has a direct sum decomposition 
    \[
    M \otimes \overline{M}(\ss) = \bigoplus_{\nu \in \hh^{\ast}} (M\otimes \overline{M}(\ss))^{\eta_{\nu}^{\ll}},
    \]
    where $(M \otimes \overline{M}(\ss))^{\eta_{\nu}^{\ll}}$ is the generalized infinitesimal character subspace 
    \[
    (M \otimes \overline{M}(\ss))^{\eta_{\nu}^{\ll}} := \{ v \in M \otimes \overline{M}(\ss) : \text{for all } z \in \mathfrak{Z}(\ll)\ \text{exists}\ n := n(z,v) \ \text{s.t.}\ (z-\chi_{\nu}^{\ll}(z))^{n}v = 0\}. 
    \]
\end{lemma}

\begin{proof}
    The $\ll$-supermodule $M \otimes \overline{M}(\ss)$ is $\hh$-semisimple, hence, we can consider its weight space decomposition 
    \[
    M \otimes \overline{M}(\ss) = \bigoplus_{\mu \in \hh^{\ast}} (M\otimes \overline{M}(\ss))^{\mu},
    \]
    and any weight space $(M \otimes \overline{M}(\ss))^{\mu}$ is finite-dimensional. Consequently, each weight space decomposes in generalized eigenspaces for some $z \in \mathfrak{Z}(\ll)$, which concludes the proof.
\end{proof}

By Theorem \ref{thm::square_Dirac} and a direct calculation, on any $(M \otimes \overline{M}(\ss))^{\eta^{\ll}_{\nu}}$, the cubic Dirac operator $\Dirac$ acts as a scalar multiple of 
\begin{equation}
c_{\nu} := \chi_{\Lambda}(\Omega_{\gg}) - \chi_{\nu}^{\ll}(\Omega_{\ll}) - c, \label{eq::generalized_eigenvalue}
\end{equation}
where the value $c$ is given explicitly in Lemma \ref{lemm::Constant_C}.

\begin{proposition} \label{prop::decomposisition_representation_space}
    The $\ll$-supermodule decomposes into a direct sum of generalized $\Delta$-eigenspaces
    \[
    M \otimes \overline{M}(\ss) = \bigoplus_{r \in \CC} (M\otimes \overline{M}(\ss))(r).
    \]
    In particular, the following decomposition holds:
    \[
    M \otimes \overline{M}(\ss) = \ker \Delta \oplus \Im \Delta.
    \]
\end{proposition}

\begin{proof}
    The lemma is a direct consequence of Lemma \ref{lemm::generalized_character_space_decomposition} and Equation \eqref{eq::generalized_eigenvalue} with
    \[
    (M \otimes \overline{M}(\ss))(r) = \bigoplus_{\nu \in \hh^{\ast}~:~c_{\nu} = r} (M \otimes \overline{M}(\ss))^{\chi_{\nu}^{\ll}}.
    \]
This concludes the proof.\end{proof}

As a direct consequence, we provide an equivalent description of the Dirac cohomology. Let $\Dirac'$ denote the Dirac operator restricted to $\ker \Delta$. Since $(\Dirac')^{2} = 0$, we have $\Im \Dirac' \subset \Ker \Dirac'$, allowing us to define the cohomology
\be
\operatorname{H}_{\Dirac'}(M) := \Ker \Dirac' / \Im \Dirac'.
\ee
This cohomology is isomorphic to the Dirac cohomology $\DC(M)$ as $\ker\Dirac \subset \Ker \Delta$.

\begin{corollary}\label{cor::different_form_DC}
As $\ll$-supermodules, the following isomorphism holds:
\[
\operatorname{H}_{\Dirac'}(M) \cong \DC(M).
\]
\end{corollary}

Moreover, we obtain a relation between $\DC(M)$ and the decomposition of $M\otimes \overline{M}(\ss)$.

\begin{lemma}
    The following three assertions are equivalent:
    \begin{enumerate}
        \item $\DC(M) = \ker \Dirac$.
        \item $\ker \Dirac \cap \Im \Dirac = \{0\}$.
        \item $M \otimes \overline{M}(\ss) = \ker \Dirac \oplus \Im \Dirac.$
    \end{enumerate}
\end{lemma}

\begin{proof}
    The assertion b) follows from a) by definition of Dirac cohomology. To deduce c) from b), we note that 
    \[
    M \otimes \overline{M}(\ss) = \ker \Delta \oplus \Im \Delta.
    \]
    by Proposition \ref{prop::decomposisition_representation_space}, and $\ker \Dirac \subset \ker \Delta$, $\Im \Delta \subset \Im \Dirac$. We show $\ker \Dirac = \ker \Delta$ and $\Im \Dirac = \Im \Delta$. 

    Assume $v \in \ker \Delta \setminus \ker \Dirac$. Then $w := \Dirac v \neq 0$, by assumption, and $\Dirac w = 0$. However, $w \in \Ker \Dirac \cap \Im \Dirac$, which is trivial. We conclude $\ker \Dirac = \Ker \Delta$. Similarly, the equality $\Im \Dirac = \Im \Delta$ holds.

    Finally, the assertion a) follows from c) by definition again.
\end{proof}

\subsection{Generalities on Dirac cohomology of highest weight supermodules}
\label{subsec::generalitiesDC}

We now characterize the kernel of the Laplace operator for admissible highest weight $(\gg,\ll)$-supermodules and establish that the Dirac cohomology of highest weight $\gg$-supermodules is non-trivial.

Let $M$ be an admissible highest weight $(\gg, \ll)$-supermodule with highest weight $\Lambda \in \hh^{\ast}$. The explicit form of the Laplace operator $\Delta$, as provided in Theorem \ref{thm::square_Dirac}, is given by  
\be
\Delta = \Omega_{\gg} \otimes 1 - \Omega_{\ll,\Delta} + c (1\otimes 1),
\ee
where \( c = (\rho,\rho) - (\rho^{\ll},\rho^{\ll}) \). According to Lemma \ref{lemm::action_Casimir}, the action of \(\Omega_{\gg}\) on \(M\) is the scalar \((\Lambda + 2\rho, \Lambda)\) times the identity. Furthermore, as a \(\ll\)-supermodule, \(M \otimes \overline{M}(\ss)\) fully decomposes into simple highest weight \(\ll\)-supermodules. By applying Lemma \ref{lemm::action_Casimir}, we get the following result.

\begin{lemma}
    Let $M$ be an admissible highest weight $(\gg, \ll)$-supermodule with highest weight $\Lambda \in \hh^{\ast}$. Let $N$ be a simple $\ll$-constituent of $M \otimes \overline{M}(\ss)$ with highest weight $\nu \in \hh^{\ast}$. Then $\Delta$ acts on $N$ as the scalar multiple
    \[
    (\Lambda + 2\rho, \Lambda) - (\nu+ 2\rho^{\ll}, \nu) + (\rho,\rho) - (\rho^{\ll},\rho^{\ll})
    \]
    of the identity. In particular, a simple $\ll$-constituent $N$ of highest weight $\nu$ belongs to $\ker \Delta$ if and only if 
    \[
    (\Lambda + 2\rho, \Lambda)+(\rho,\rho) = (\nu + 2\rho^{\ll},\nu) + (\rho^{\ll},\rho^{\ll}).
    \]
\end{lemma}

\begin{corollary}
   Let $M$ be an admissible highest weight $(\gg, \ll)$-supermodule with highest weight $\Lambda \in \hh^{\ast}$, and let $N$ be a simple $\ll$-supermodule constituent of $M \otimes \overline{M}(\ss)$ with highest weight $\nu \in \hh^{\ast}$. Assume there exist $w \in W^{\gg}$, $t_{1}, \dots, t_{k} \in \CC$, and linearly independent odd isotropic roots $\alpha_{1}, \dotsc, \alpha_{k} \in \Phi^{+}$ satisfying $(\Lambda + \rho, \alpha_{i}) = 0$ for all $i$, such that 
\[
\nu = w\biggl(\Lambda + \rho + \sum_{i=1}^{k} t_{i} \alpha_{i}\biggr) - \rho^{\ll}.
\]
Then $N$ belongs to $\ker \Delta$.
\end{corollary}

Finally, we show that the Dirac cohomology of generic highest weight supermodules is non-trivial.

\begin{proposition}\label{prop::HW_non-trivial}
    Let $M$ be a highest weight $\gg$-supermodule with highest weight $\Lambda$. Then $\DC(M)$ is non-trivial.
\end{proposition}
\begin{proof}
   Let $v_{M}$ be the highest weight vector of $M$, and let $v_{\overline{M}(\ss)} = 1 \otimes 1 \in \overline{M}(\ss)$. In particular, $\uu$ annihilates $v_{M}$ and $v_{\overline{M}(\ss)}$, respectively, where for the latter we may identify elements of $\uu$ with $\frac{\partial}{\partial x_{i}}$ and $\frac{\partial}{\partial \eta_{j}}$, as in Section \ref{subsec::Clifford_superalgebras}.

First, we show that $v_{M} \otimes v_{\overline{M}(\ss)} \in M \otimes \overline{M}(\ss)$ lies in $\ker \Dirac$. As in Section \ref{subsec::Cubic_Dirac_operators}, we decompose $\Dirac = C + \Cbar$, where $C = A + 1 \otimes a$ and $\Cbar = \Abar + 1 \otimes \overline{a}$, with 
\[
A = \sum_{i = 1}^{s} u_{i} \otimes \ou_{i}, \qquad 
\Abar =  \sum_{i = 1}^{s} (-1)^{p(u_{i})}\ou_{i} \otimes u_{i},
\]
\[
{a} = -\frac{1}{4} \sum_{1 \leq i,j \leq s}
(-1)^{p(u_{i})p(u_{j}) + p(u_{i}) + p(u_{j})} 
[\ou_{i}, \ou_{j}] \, u_{i} u_{j},
\]
\[
\overline{a} = -\frac{1}{4} \sum_{1 \leq i,j \leq s} (-1)^{p(u_{i})p(u_{j})} [u_{i},u_{j}] \ou_{i} \ou_{j}.
\]

Since $u_{i}v_{M} = 0$ and $u_{i}v_{\overline{M}(\ss)} = 0$ for all $i = 1, \dots, s$, it is immediate that $A$, $\Abar$, and $a$ annihilate $v_{M} \otimes v_{\overline{M}(\ss)}$. To see that $\bar{a}$ annihilates $v_{M} \otimes v_{\overline{M}(\ss)}$, we note that $[u_{i},u_{j}] \in \uu$ for all $1 \leq i,j \leq s$, and we can rewrite $\bar{a}$ as (Remark \ref{rmk::different_form_a_abar})
\[
\bar{a} = -\frac{1}{4} \sum \limits_{1 \leq i,j \leq s} (-1)^{p(u_i)p(u_j)+p(u_{i})+p(u_{j})} \ou_i \ou_j [u_i, u_j].
\]
We conclude that $v_{M} \otimes v_{\overline{M}(\ss)} \in \ker \Dirac$.

Furthermore, under the identification of $\uu$ with  elements of the form $\frac{\partial}{\partial x_{i}}$ and $\frac{\partial}{\partial \eta_{j}}$ acting on $\bigwedge \ubar$, the discussion in Section \ref{subsec::Clifford_superalgebras} makes it immediate that $v_{\overline{M}(\ss)}$, and therefore $v_{M} \otimes v_{\overline{M}(\ss)}$, cannot lie in the image of $\Dirac$. This concludes the proof.
\end{proof}

\subsection{Dirac cohomology and finite-dimensional supermodules} \label{subsec:fin_dim}
In this subsection, we specify to the Dirac cohomology of finite-dimensional simple admissible $(\gg, \ll)$-supermodules, where $\gg$ is a basic classical Lie superalgebra of type $1$. In particular, we will compute the Dirac cohomology of those with typical highest weight and finite-dimensional simple objects in $\calO^{\pp}$.

In the following, throughout this subsection, we let $\gg$ be a basic classical Lie superalgebra of type $1$, \emph{i.e.}, $\gg$ is $\glmn$, $A(m\vert n)$, or $C(n)$. The simple finite-dimensional supermodules are parameterized by dominant integral weights $\lambda \in \hh^{\ast}$ with respect to some Borel subalgebra $\bb = \bb_{0} \oplus \bb_{1}$, \emph{i.e.}, those weights for which there exists a finite-dimensional simple $\even$-module with highest weight $\lambda$ for $\bb_{0}$. More precisely, $\lambda$ is dominant integral if and only if 
\be
(\lambda + \rho_{0},\alpha) > 0 \ \text{for all} \ \alpha \in \Phi_{0}^{+},
\ee
where $\Phi^{+} = \Phi_{0}^{+} \sqcup \Phi^{+}_{1}$ is the positive system with respect to $\bb$.
We denote the set of $\bb$-dominant integral weights by $P^{++}_{\bb}$, and call them also $\Phi^{+}$-dominant integral. Moreover, if $\bb$ is clear from the context, we omit the subscript and simply write $P^{++}$. 

For any $\lambda \in P^{++}_{\bb}$, we define $L_{\bb}(\lambda)$ (or simply $L(\lambda)$ when no confusion arises)  to be the simple supermodule with highest weight $\lambda$ with respect to $\bb$ such that the highest weight vector is even. The fixed notation allows us to parameterize the simple finite-dimensional $\gg$-supermodules as follows:
\be
\{L(\lambda), \Pi L(\lambda) : \lambda \in P^{++}\}.
\ee

In the following, we assume that $M := L(\lambda)$, $\lambda \in P^{++}_{\bb}$, is an admissible finite-dimensional simple $(\gg, \ll)$-supermodule, where $\bb$ is the distinguished Borel subalgebra. The \emph{distinguished Borel subalgebra} is defined with respect to the \emph{distinguished positive system} $\Phi^{+}$, that is, the positive system with the smallest number of odd roots. For this system, as a direct calculation shows, we have
\be \label{eq::inner_product_with_rho_0}
(\rho_{1}, \alpha) \begin{cases}
    = 0 \  &\text{if} \ \alpha \in \Phi_{0}, \\
    > 0 \ &\text{if} \ \alpha \in \Phi_{1}^{+}.
\end{cases}
\ee

The $\ll$-supermodule $M \otimes \overline{M}(\ss)$ decomposes completely in finite-dimensional weight $\ll$-supermodules by Proposition \ref{prop::completely_reducible}. Each finite-dimensional simple weight $\ll$-supermodule is of highest weight type with respect to some positive system $\Phi(\ll;\hh)^{+}$. The associated highest weight $\mu \in \hh^{\ast}$ is called $\Phi(\ll;\hh)^{+}$\emph{-dominant integral}. If $\nu$ is the $\Phi(\ll;\hh)^{+}$-dominant integral, we denote the associated simple $\ll$-supermodule by $L_{\ll}(\nu)$.

For the remainder, we fix a positive system $\Phi(\ll;\hh)^{+}$, which is contained in the distinguished positive system $\Phi^{+}$. In particular, if $M$ is a highest weight $\gg$-supermodule with highest weight $\lambda$ with respect to $\Phi^{+}$, then $\lambda$ is the highest weight with respect to $\Phi(\ll;\hh)^{+}$ of an $\ll$-constituent of $M$.  For any $\mu \in \hh^{\ast}$, we define
\be
W_{\lambda}^{\ll, 1} := \{w \in W^{\ll} : w(\lambda + \rho^{\ll}) \ \text{is} \ \Phi(\ll;\hh)^{+}\text{-dominant integral}\}.
\ee
Note that $w\Phi^{+}$ contains $\Phi(\ll;\hh)^{+}$ whenever $w \in W_{\lambda}^{\ll,1}$.

We are now ready to compute the Dirac cohomology. We use that $\DC(M)$ decomposes completely in simple finite-dimensional highest weight $\ll$-supermodules, and, crucially, Theorem \ref{thm::Casselmann_Osborne}. We start with two lemmas that will enter the proofs of our main result.

\begin{lemma} \label{lemm::form_nu_DC}
    Let $M$ be an admissible finite-dimensional simple $(\gg, \ll)$-supermodule with highest weight $\Lambda$. Then $\DC(M)$ decomposes in a direct sum of simple finite-dimensional highest weight $\ll$-supermodules each of with highest weight $\nu$ of the form 
    \[
    \nu = w\biggl(\Lambda+\rho+\sum_{i=1}^{k}t_{i}\alpha_{i}\biggr)-\rho^{\ll},
    \]
    for some $w \in W^{\ll}$, $t_{i} \in \CC$ and isotropic odd roots $\alpha_{i}$ in $\Phi(\ll;\hh)$ satisfying $(\Lambda + \rho,\alpha_{i}) = 0$.
\end{lemma}

\begin{proof}
    As $M$ is a highest weight $\gg$-supermodule with highest weight $\Lambda$, the Dirac cohomology contains a simple highest weight $\ll$-supermodule with highest weight $\Lambda + \rho^{\uu}$ by Proposition \ref{prop::HW_non-trivial}. If $V$ is another simple highest weight $\gg$-supermodule with highest weight $\nu$, Theorem \ref{thm::Casselmann_Osborne} dictates for any $z \in \mathfrak{Z}(\ll)$ on $\DC(M)$ the equality,
    \[
     \chi_{\Lambda+\rho^{\uu}}^{\ll}(z)= \chi_{\Lambda+\rho^{\uu}}^{\ll}(\eta_{\ll}(z)) = \chi^{\ll}_{\nu}(\eta_{\ll}(z)) =  \chi^{\ll}_{\nu}(z),
    \]
    which concludes the proof with Equation \eqref{eq::relation_infinitresimal_character} and $\rho = \rho^{\ll} + \rho^{\uu}$.
\end{proof} 

\begin{lemma} \label{lemm::multiplicity_DC}
    Let $M$ be an admissible finite-dimensional simple $\gg$-supermodule with highest weight $\Lambda$. Then any simple $\ll$-constituent in $\DC(M)$ appears with multiplicity one. 
\end{lemma}

\begin{proof}
    Let $V$ be a (non-trivial) simple $\ll$-constituent of $\DC(M)$ with highest weight $\nu$. By Lemma \ref{lemm::form_nu_DC}, the highest weight $\nu$ is of the form 
    \[
    \nu = w\biggl(\Lambda+\rho+\sum_{i=1}^{k}t_{i}\alpha_{i}\biggr)-\rho^{\ll},
    \]
    for some $w \in W^{\ll}$, $t_{i} \in \CC$ and isotropic odd roots $\alpha_{i}$ in $\Phi(\ll;\hh)$ satisfying $(\Lambda + \rho,\alpha_{i}) = 0$. We may assume $w \in W_{\mu}^{\ll,1}$ for $\mu = \Lambda + \rho^{\uu} + \sum_{i = 1}^{k}t_{i}\alpha_{i}$. To prove the lemma, we have to show that 
    \[
    w(\lambda + \rho + \sum_{i=1}^{k}t_{i}\alpha_{i}) - \rho^{\ll} = w(\lambda + \sum_{i=1}^{k}t_{i}\alpha_{i}) + \rho^{w\uu}
    \]
    appears with multiplicity one. Here, we denote by $\rho^{w\uu}$ the Weyl element with respect to $w(\Phi^{+}\setminus \Phi(\ll;\hh)^{+})$, and note that $w\rho^{\ll} = \rho^{\ll}$.  

    Assume this is not the case. Then, as we are dealing with highest weight supermodules, there exists $A \subset w(\Phi^{+}\setminus \Phi(\ll;\hh)^{+})$ and $B \subset w\Phi^{+}$ with 
    \[
    w(\Lambda + \sum_{i=1}^{k}t_{i}\alpha_{i}) + \rho^{w\uu} = (w(\Lambda +\sum_{i=1}^{k}t_{i}\alpha_{i})-\ZZ_{+}[B]) + (\rho^{w\uu}-\ZZ_{+}[A]),
    \]
    where $\ZZ_{+}[A] := \sum_{\xi \in A}\ZZ_{+}\xi$ and $\ZZ_{+}[B] := \sum_{\zeta \in A}\ZZ_{+}\zeta$. This forces $\ZZ_{+}[A] + \ZZ_{+}[B] = 0$. Now, if we are taking the inner product with $\rho_{0}$ and use the invariance of $(\cdot,\cdot)$ under the action of the Weyl group together with Equation \eqref{eq::inner_product_with_rho_0}, we deduce that neither $A$ nor $B$ can contain odd roots. On the other hand, taking then the inner product with $w(\lambda + \rho_{0})$, we obtain 
\[
\sum_{\alpha \in A} (w(\lambda+\rho_{0}),\alpha) + \sum_{\beta \in B} (w(\lambda +\rho_{0}),\beta) = 0,
\]
and each summand is strictly positive as no odd roots appear and $w(\Lambda + \rho_0)$ is dominant integral. Consequently, $A = \emptyset$ and $B = \emptyset$. This concludes the proof.
\end{proof}

The above lemmas enter the proof of the following Theorem, which is one of the main result of this section. 

\begin{theorem} \label{thm::DC_finite_dimensional}
    Let $M$ be an admissible finite-dimensional simple $(\gg, \ll)$-supermodule with typical highest weight $\Lambda$. Then 
    \[
    \DC(M) = \bigoplus_{w \in W_{\Lambda + \rho^{\uu}}^{\ll,1}} L_{\ll}(w(\Lambda + \rho) - \rho^{\ll}).
    \]
\end{theorem}

\begin{proof}
    Let $V$ be a non-trivial $\ll$-constituent in $\DC(M)$ with highest weight $\nu$. By atypicality of $\Lambda$, Lemma \ref{lemm::form_nu_DC} and Lemma \ref{lemm::multiplicity_DC}, the simple finite-dimensional $\ll$-supermodule $V$ appears with multiplicity one and has highest weight
    \[
    \nu = w(\Lambda + \rho)-\rho^{\ll}
    \]
    for some $w \in W_{\Lambda}^{\ll,1}$. Consequently, it is enough to show that each $w(\Lambda + \rho)-\rho^{\ll}$ appears as a weight in $M \otimes \overline{M}(\ss)$. However, this is immediate as \[w(\Lambda + \rho) - \rho^{\ll} = w(\Lambda) + (w(\rho)-\rho^{\ll}) = w(\Lambda) + w(\rho^{\uu})\] is a sum of extreme weights in $M$ and $\overline{M}(\ss)$, respectively, which are $\Phi(\ll;\hh)^{+}$-dominant integral.
\end{proof}

These results for general finite-dimensional admissible $(\gg, \ll)$-supermodules allow us to compute the Dirac cohomology of finite-dimensional simple objects in $\calO^{\pp}$ for some parabolic subalgebra $\pp = \ll \ltimes \uu$ with reductive $\ll_{0}$.  

\subsubsection{Dirac cohomology and finite-dimensional simple objects in $\calO^{\pp}$}

Let $\pp = \ll \ltimes \uu$ be a parabolic subalgebra with reductive $\ll_{0}$. This is the case if the Levi subalgebra is good \cite[Section 7]{DMP}. Recall that by Theorem \ref{thm::DMP}, we may consider any finite-dimensional simple $M$ as the unique simple quotient of a parabolically induced supermodule, where $\pp$ has good Levi subalgebra $\ll$.
 
By definition of $\calO^{\pp}$ in Section \ref{subsubsec::Op}, any object decomposes completely in finite-dimensional simple $\ll_{0}$-modules. In particular, if $M$ belongs to $\calO^{\pp}$, the supermodules $M \otimes \overline{M}(\ss)$ and $\DC(M)$ decompose completely in finite-dimensional $\ll_{0}$-modules by a similar argumentation as above.

Let $M := L(\Lambda) \in \calO^{\pp}$ be a finite-dimensional simple object with highest weight $\Lambda \in \hh^{\ast}$ with respect to some positive system. Fix a positive system $\Phi(\ll;\hh)^{+}$. Then, noting $\hh \subset \ll_{0}$, we may identify $\Phi(\ll_{0};\hh)$ with $\Phi(\ll;\hh)_{0}$ and $\Phi(\ll_{0};\hh)^{+}$ with $\Phi(\ll;\hh)_{0}^{+}$.

Adapting Theorem \ref{thm::Casselmann_Osborne} and Lemma \ref{lemm::form_nu_DC} to the new setting, we conclude that any $\ll_{0}$-constituent in $\DC(M)$ has highest weight 
\be
\nu = w(\Lambda + \rho^{\uu} + \rho^{\ll_{0}}) - \rho^{\ll_{0}}
\ee
for some $w \in W_{\lambda}^{\ll_{0},1} = \{ w \in W^{\gg} : w(\lambda + \rho^{\ll_{0}}) \ \text{is} \ \Phi(\ll_{0};\hh)^{+}\text{-dominant integral}\}$. No isotropic roots appear as $\ll_{0}$ is purely even. Now, a straightforward modification of the proofs of Lemma \ref{lemm::multiplicity_DC} and Theorem \ref{thm::DC_finite_dimensional} leads to the following theorem.

\begin{theorem} \label{thm::DC_OP}
    Let $M \in \calO^{\pp}$ be simple and finite-dimensional with highest weight $\Lambda$. Then 
    \[
    \DC(M) = \bigoplus_{w \in W_{\Lambda + \rho^{\uu}}^{\ll_{0},1}} L_{\ll_{0}}(w(\Lambda + \rho^{\uu} + \rho^{\ll_{0}})-\rho^{\ll_{0}}).
    \]
\end{theorem}

\section{Dirac cohomology and Kostant's cohomology}
In this section, we study the relation between Dirac cohomology and Kostant's (co)homology. For that, we briefly introduce Kostant's (co)homology in Subsection \ref{subsec::Kostant's_cohomology}. Then Proposition \ref{prop::identification} allows us to decompose the cubic Dirac operator in terms of the boundary and coboundary operators of Kostant's (co)homology. This decomposition is used to deduce an embedding of the Dirac cohomology into Kostant's (co)homology in Theorem \ref{thm::embedding}. 

Afterward, we study Dirac cohomology for unitarizable supermodules with respect to Hermitian real forms, introduced in Subsection \ref{subsec::Hermitian_real_forms}. In this case, we show that Dirac cohomology and Kostant's (co)homology are isomorphic as supermodules over the Levi subalgebra. 

Finally, as an application, we study the Dirac cohomology of simple weight supermodules in Subsection \ref{subsec::application}. These supermodules are particularly interesting because they are $\hh$-semisimple with finite-dimensional weight spaces. We establish in Theorem \ref{thm::DC_weight_supermodules} that weight supermodules have trivial Dirac cohomology unless they are highest weight supermodules.

\subsection{Kostant's $\uu$-cohomology and $\ubar$-homology} \label{subsec::Kostant's_cohomology}

We start fixing some notations. Let $M$ be a $\gg$-supermodule that is $(\gg, \ll)$-admissible, \emph{i.e.}, $M$ is a $\gg$-supermodule that is $\ll$-semisimple. Fix a parabolic subalgebra $\pp := \ll \ltimes \uu$ with opposite parabolic subalgebra $\overline{\pp} = \ll \ltimes \ubar$. Recall that $\ll$ is the Levi subalgebra and $\uu,\ubar$ are the nilradicals of $\pp$ and $\overline{\pp}$, respectively.

Recall that the exterior superalgebras $\bigwedge \uu$ and $\bigwedge \ubar$ over the super vector spaces $\uu$ and $\ubar$ inherit a natural $\ZZ$-grading induced by the $\ZZ$-grading of the tensor superalgebras $T(\uu)$ and $T(\ubar)$, respectively. More precisely, let $V$ be either $\uu$ or $\ubar$, then the \emph{exterior} $n$\emph{-power} is the super vector space
\begin{align}
\bigwedge\nolimits ^{n} V = V^{\otimes n}/\mathfrak{J}_{n},\end{align}
where $\mathfrak{J}_{n}$ is the subspace of $V^{\otimes n}$ generated by the elements of the form 
\begin{align}
    v_{1} \otimes \dots \otimes v_{n} - (-1)^{p(\sigma)} \sigma \cdot v_{1} \otimes \dots \otimes v_{n}, \qquad \sigma \in S_{n}, \ v_{i} \in V.
\end{align}
Here, $p(\sigma)$ denotes the parity of the permutation $\sigma$.

After establishing the notation, we introduce Kostant's (co)homology, see \cite[Section 6.4]{cheng2012dualities}, which is a natural generalization of the classical Lie algebra (co)homology theory, see for example \cite{Knapp} or \cite{NojaSuper} for the super case. We define the space of $p$\emph{-chains}, $C_{p}(\ubar,M)$, and the space of $p$\emph{-cochains}, $C^{p}(\uu,M)$, as
\be
C_{p}(\ubar,M) := \bigwedge\nolimits^{p} \ubar \otimes M, \qquad C^{p}(\uu,M) := \Hom_{\CC}(\bigwedge\nolimits^{p} \uu,M)
\ee
for any $p \in \ZZ_{+}$. Note that $C_{p}(\ubar,M)$ and $C^{p}(\uu,M)$ are naturally $\pp$-supermodules, and if $M$ is finite-dimensional, there is a natural identification $C_{p}(\ubar,M) \cong C^{p}(\uu,M)$. Moreover, we set $C_{\ast} := \sum_{p\geq 0}C_{p}(\ubar,M)$, and $C^{\ast} := \sum_{p \geq 0} C^{p}(\uu,M)$, which are $\pp$-supermodules by construction.

Define the \emph{boundary operator} $d := \sum_{p} d_{p} : C_{\ast}(\ubar,M) \to C_{\ast}(\ubar,M)$ by setting
\be
\label{eq:boundary}
\begin{split}
d_{p} & (x_{1}\ldots x_{p} \otimes v) := \\ & \sum_{s = 1}^{p} (-1)^{s+p(x_{s})\sum_{i = s+1}^{p}p(x_{i})}x_{1} \ldots \hat{x}_{s}\ldots x_{p} \otimes x_{s}v
\\ &+ \sum_{1 \leq s < t \leq p} (-1)^{s+t+p(x_{s})\sum_{i=1}^{s-1}p(x_{i})+p(x_{t})\sum_{j = 1}^{t-1}p(x_{j})+p(x_{s})p(x_{t})}[x_{s},x_{t}]x_{1}\ldots \hat{x}_{s}\ldots \hat{x}_{t} \ldots x_{p}\otimes v 
\end{split}
\ee
for $x_{i} \in \uu$ homogeneous and $v \in M$. Here, $\hat{x}$ indicates that the corresponding term $x$ is omitted. A straightforward calculation yields $d_{p-1}\circ d_{p} = 0$ for any $p \in \ZZ_{+}$. The $p$-th $\ubar$\emph{-homology group} with coefficients in $M$, denoted by $H_{p}(\ubar,M)$, is defined to be the $p$-th homology group of the following chain complex:
\be
\ldots \to C_{p}(\ubar,M) \xrightarrow{d_{p}} C_{p-1}(\ubar,M) \xrightarrow{d_{p-1}} \ldots \xrightarrow{d_{2}} \ubar \otimes V \xrightarrow{d_{1}} M \xrightarrow{d_{0}} 0,
\ee
that is, $H_{p}(\ubar,M) := \ker d_{p} / \Im d_{p+1}$ for $p \in \ZZ_{+}$. We refer to $H_{p}(\ubar,M)$ as \emph{Kostant's homology} for $M$.

Define the \emph{coboundary operator} $\partial := \sum_{p}\partial_{p} : C^{\ast}(\uu,M) \to C^{\ast}(\uu,M)$ by
\begin{align}
\begin{split}
    &(\partial_{p}f)(x_{1},\ldots, x_{p+1}) := \\ &\sum_{s = 1}^{p+1}(-1)^{s+1+p(x_s)(p(f)+\sum_{i=1}^{s-1}p(x_{i}))}x_{s}f(x_{1},\ldots,x_{s-1},\hat{x}_{s}, x_{s + 1}, \ldots, x_{p+1}) \\ &+ \sum_{s < t} (-1)^{s+t+p(x_{s})\sum_{i = 1}^{s-1}p(x_{i})+p(x_{t})\sum_{j = 1}^{t}p(x_{j})+p(x_{s})p(x_{t})}f([x_{s},x_{t}],x_{1},\ldots,\hat{x}_{s}, \ldots, \hat{x}_{t}, \ldots, x_{p+1}),
    \end{split}
\end{align}
for homogeneous $x_{i} \in \uu$. The coboundary operator $\partial$ satisfies $\partial_{p} \circ \partial_{p-1} = 0$ for all $p \in \ZZ_{+}$.  The $p$-th $\uu$\emph{-cohomology group} with coefficients in $M$, denoted by $H^{p}(\uu,M)$, is defined to be the $p$-th cohomology group of the following chain complex:
\be
0 \xrightarrow{\partial_{-1}} M \xrightarrow{\partial_{0}}C^{1}(\uu,M) \xrightarrow{\partial_{1}} \ldots \xrightarrow{\partial_{p-1}} C^{p}(\uu,M) \xrightarrow{\partial_{p}} C^{p+1}(\uu,M) \xrightarrow{\partial_{p+1}} \ldots,
\ee
\emph{i.e.}, $H^{p}(\uu,M) := \ker \partial_{p} /\Im \partial_{p-1}$.We refer to $H^{p}(\uu,M)$ as \emph{Kostant's cohomology} for $M$. As in \cite{huang2005dirac}, we will consider the operator $\delta = - 2d$. It is clear, that $\delta$ defines the same cohomology as $d$.

A direct calculation yields the subsequent lemma.

\begin{lemma}
    The boundary operator $d : C_{\ast}(\ubar,M) \to C_{\ast}(\ubar,M)$ and the coboundary operator $\partial : C^{\ast}(\uu, M) \to C^{\ast}(\uu,M)$ are $\ll$-supermodule morphisms. In particular, $H_{p}(\ubar,M)$ and $H^{p}(\uu,M)$ are $\ll$-supermodules.
\end{lemma}

\begin{remark}
    The boundary operators $\partial$ is even a $\pp$-supermodule morphism.
\end{remark}

Having introduced Kostant's homology and cohomology, we now describe their relation in the case $M$ is an admissible $(\gg, \ll)$-supermodule, following \cite{cheng2012dualities}. In this case, $M$ is also an admissible $(\gg,\hh)$-supermodule, that is $M = \oplus_{\mu \in \hh^{\ast}}M^{\mu}$ with $\dim(M^{\mu}) < \infty$. To any such $M$, we can assign a \emph{dual supermodule}. Let $\tau : \gg \to \gg$ be the Chevalley automorphism of $\gg$, and for any $\mu \in \hh^{\ast}$, let $(M^{\mu})^{\ast}$ denote the dual space of the finite-dimensional weight space $M^{\mu}$. The \emph{dual} of $M$ is the $\gg$-supermodule 
\be
M^{\vee} := \bigoplus_{\mu \in \hh^{\ast}}(M^{\mu})^{\ast},
\ee
with $\gg$-action given by $(X \cdot f)(v) := (-1)^{p(X)p(f)-1}f(\tau(X)\cdot v)$.

\begin{lemma}[{\cite[Theorem 6.22]{cheng2012dualities}}]
    Let $M$ be an admissible $(\gg, \ll)$-supermodule and $p \in \ZZ_{+}$. Then, as semisimple $\ll$-supermodules, we have the following isomorphisms:
   \[H_{p}(\ubar,M^{\vee}) \cong H^{p}(\uu,M).\]
   In particular, if $M$ is a simple highest weight $\gg$-supermodule, we have 
   \[
   H_{p}(\ubar,M) \cong H^{p}(\uu,M). 
   \]
\end{lemma}

\subsection{Relation to Dirac cohomology}
We construct an explicit embedding of Dirac cohomology into Kostant's cohomology for admissible $(\gg, \ll)$-supermodules. This involves identifying $M \otimes \overline{M}(\ss)$ with $M \otimes \bigwedge \ubar \otimes \CC_{\rho^{\uu}}$ as $\ll$-supermodules (\emph{cf.}~Proposition \ref{prop::identification}) and interpreting the operators $C$ and $\Cbar$ (\emph{cf.}~Theorem \ref{theorem::decomp_Dirac}) as the boundary and coboundary operators, respectively.

\begin{proposition}\label{prop::identification_C_Cbar_d_delta}
    {Let $M$ be an admissible $(\gg, \ll)$-supermodule. Then under the action of $\UE(\gg) \otimes C(\ss)$ on $M \otimes \overline{M}(\ss)$, and the identification in Proposition \ref{prop::identification}, the operators $C$ and $\Cbar$ act as $\delta = -2d$ and $\partial$, respectively. In particular, $\Dirac$ acts as $\partial - 2d = \partial + \delta$.} 
\end{proposition}
\begin{proof}
   Under the identification $ \overline{M}(\ss) \otimes M \cong \bigwedge \ubar \otimes \CC_{\rho^{\uu}} \otimes M$ of $\ll$-supermodules, we compare the actions of $C$, $\Cbar$, $d$, and $\partial$. We explicitly perform the calculation for $C$ and the boundary operator $d$, while the other cases follow by a similar line of argument.

\comment{We identify $\Hom_{\CC}(\bigwedge\ubar,
    M) \cong \bigwedge \ubar \otimes M$. Then the 
    differential $\partial$ acts on $Y_1 \wedge \dots \wedge Y_n \otimes v \in 
    \bigwedge \ubar \otimes M$ by
    \be
        \begin{split}
            \partial (Y_1 \wedge \dots & \wedge Y_n \otimes v) = \sum\limits_{i = 1}^{s} u_i  \wedge Y_1 \wedge \dots \wedge Y_n \otimes \ou_i \cdot v\\
            &+ \frac{1}{2} \sum\limits_{i = 1}^{s} \sum\limits_{t = 1}^{p} 
            (-1)^{t + p(Y_t)\sum\limits_{k = 1}^t p(Y_t) + p(u_i)p(Y_t)} u_i \wedge Y_1 \wedge \dots \wedge [\ou_i, Y_t] \wedge \dots \wedge Y_n \otimes v.
        \end{split}
        \label{eq:partial_identification}
    \ee}

By swapping the order of $\UE(\gg)$ and $C(\ss)$, we may rewrite $C = A' + a \otimes 1$ with 
\begin{align*}
    A' = \sum\limits_i (-1)^{p(u_{i})}u_i \otimes \ou_i = \sum_{i} (\ou_{i})^{\ast} \otimes \ou_{i},
\end{align*}
where we note that exchanging the factors of the tensor product introduces an extra factor $(-1)^{p(u_i)}$ in the sum, and $(\ou_{i})^{\ast} = (-1)^{p(u_{i})}u_{i}$. To deduce the action of $C$, we recall the action of $u \in \uu$ and $\ou \in \ubar$ on $Y = Y_1 \dots Y_n \in \overline{M}(\ss)$:
    \begin{equation*}
        \begin{split}
        &\ou \cdot Y := \ou \wedge Y, \\
        u \cdot Y := 2 \sum\limits_{t = 1}^n &(-1)^{t + 1 + p(u) \sum\limits_{k = 1}^{t-1}p(Y_t)}(u, Y_t) Y_1 \dots \hat{Y}_t \dots Y_n.
        \end{split}
    \end{equation*}
    \comment{
    We consider the operators $A$ and $a \otimes 1$ and $\overline{a} \otimes 1$ separately. We recall, that 
    \be
        A = \sum\limits_{i = 1}^s u_i \otimes u_i^{\ast}
    \ee

    \textcolor{red}{From my point of view, it is enough to consider either $C$ or $\Cbar$!!
    It is immediate that $A$ acts as the first summand of $\partial$.\\}}
    
    For general $\xi = Y_1 \dots Y_N \otimes v \in \overline{M}(\ss) \otimes M$, we find
    \begin{align*}
        A' \xi = \sum\limits_{i = 1}^s 2 \sum\limits_{t = 1}^n (-1)^{t + 1 + p(u_i) \sum\limits_{k = 1}^{t-1}p(Y_t)} ((\ou_i)^{\ast}, Y_t) Y_1 \dots \hat{Y}_t \dots Y_n \otimes  \ou_i v.
    \end{align*}
    Here, we use that $((\overline{u}_i)^{\ast}, Y_t)$ is only non-zero if $p(u_i) = p(Y_t)$, and hence we can replace $p(u_i)$ by $p(Y_t)$ in the exponent. Then, since $\sum_i ((\ou_i)^{\ast}, Y_t)\ou_i = (-1)^{p(Y_t)p(Y_t)}Y_t$, we see that this is equal to minus twice the first sum in the expression for $d$ in Equation \eqref{eq:boundary}. 
    
    It remains to identify the action of the cubic term ${a}$.\comment{We use the expression $a = -\frac{1}{4} \sum \limits_{i,j} (-1)^{p(u_i)p(u_j)} u_i u_j [\ou_i, \ou_j]$ (Remark \ref{rmk::different_form_a_abar}) \textcolor{red}{ATTENTION! This is $\bar{a}$ after your changes?!} into
    \begin{align*}
    \begin{split}
    -\frac{1}{4} &\sum\limits_{i,j} u_j u_i 2\sum\limits_{t = 1}^n (-1)^{t + 1 + p(u_i) \sum\limits_{k = 1}^{t-1}p(Y_t)} ([\ou_i, \ou_j], Y_t) Y_1  \dots \hat{Y}_k \dots Y_n \otimes v\\
    &= \frac{1}{2} \sum\limits_{i,j,t} (-1)^{t + (p(u_i) + p(u_j) \sum\limits_{k = 1}^{t}p(Y_t) + p(u_i) + p(Y_t)} ([\ou_i, Y_t], \ou_j) u_j u_i Y_1 \dots \hat{Y}_t \dots Y_n \otimes v
    \end{split}
    \end{align*}
    
    Now, we sum $\sum_j ([\ou_i, Y_t], \ou_j) u_j = [\ou_i, Y_t]$ and after commuting $[\ou_i, Y_t]$ into its proper place, we get the second sum in the expression \eqref{eq:partial_identification} for $\partial$.} In order to compute $a$, we fix some notation. For $Y_1 \dots Y_n$, let
    \begin{align*}
        \hat{Y}_{t,r} = \gamma(t,r) Y_1 \dots \hat{Y}_t \dots \hat{Y}_r \dots Y_n,
    \end{align*}
    where 
    \[
        \gamma(t,r) = \left\{ \begin{matrix} 1 & \textrm{if} \ t < r, \\
        -(-1)^{p(Y_t)p(Y_r)} & \textrm{if} \ t > r, \\
        0 & \textrm{if} \ t = r.
        \end{matrix} \right.
    \]
    This allows us to write
    \begin{align*}
        u_i u_j Y_1 \dots Y_n = 4 \sum\limits_{t,r} (-1)^{t + r + p(u_j) \sum\limits_{k = 1}^{t-1} p(Y_t) + p(u_i)\sum\limits_{l = 1}^{r - 1} p(Y_l)} ((\ou_i)^{\ast}, Y_t) ((\ou_j)^{\ast}, Y_r) \hat{Y}_{t,r},
    \end{align*}
    and ${a} \otimes 1$ acts on $Y_1 \dots Y_n \otimes v$ as
    \begin{align*}
        -\frac{1}{4} 4 \sum\limits_{i,j,t,r}(-1)^{t + r + p(u_j) \sum\limits_{k = 1}^{t-1} p(Y_t) + p(u_i)\sum\limits_{l = 1}^{r - 1} p(Y_l)} ((\ou_i)^{\ast}, Y_t)((\ou_j)^{\ast}, Y_r) [\ou_j, \ou_i] \hat{Y}_{t,r} \otimes v.
    \end{align*}
    Again, we can replace $p(u_i)$ by $p(Y_t)$ and $p(u_j)$ by $p(Y_r)$. Upon summing $\sum\limits_{i}  ((\ou_i)^{\ast}, Y_\kappa)\ou_i = (-1)^{p(Y_{\kappa})p(Y_\kappa)}Y_\kappa$ for $\kappa \in \{r,t\}$, we obtain
    \begin{align*}
        -\sum\limits_{t,r} (-1)^{t + r + p(Y_t) \sum_{k = 1}^t p(Y_t) + p(Y_r) \sum_{l = 1}^{r} p(Y_l) + p(Y_r)p(Y_t)} [Y_t, Y_r] \hat{Y}_{t,r} \otimes v.
    \end{align*}
   This expression remains invariant under exchanging the roles of $k$ and $l$ and vanishes for $k = l$. We conclude that it equals twice the sum restricted to $k < l$, i.e., minus twice the second sum in the expression for $d$. This finishes the proof.
\end{proof}
The subsequent corollary is immediate.

\begin{corollary} \label{cor::(co)homology_C_Cbar} Let $M$ be an admissible $(\gg, \ll)$-supermodule. Then, as $\ll$-supermodules, we have the following isomorphisms
     \[
    H(C, M\otimes \overline{M}(\ss)) \cong H_{\ast}(\ubar, M) \otimes \CC_{\rho^{\uu}}, \qquad H(\Cbar, M\otimes \overline{M}(\ss)) \cong H^{\ast}(\uu,M) \otimes \CC_{\rho^{\uu}}
    \]
\end{corollary}

We now rely on Proposition \ref{prop::identification_C_Cbar_d_delta} and Corollary \ref{cor::(co)homology_C_Cbar} to construct an embedding of the Dirac cohomology in Kostant's cohomology. As a first step, we study the relation between $\ker \Dirac, \Ker \Delta, \ker \delta$ and $\ker \partial$. More precisely, we establish that the map
\be
\ker \Dirac \to \ker \Delta \cap \ker \partial, \quad x \mapsto x
\ee
is a well-defined $\ll$-equivariant injective map. The proof is similar to the proof of Lemma 4.6 and Lemma 4.7 in \cite{Huang_Xiao_Dirac_HW}.

To this end, we use the fact that, as $\ll$-supermodules (see Proposition \ref{prop::identification}),
\[
M \otimes \overline{M}(\ss) \cong M \otimes \bigwedge \ubar \otimes \CC_{\rho^{\uu}}.
\]
The exterior superalgebra $\bigwedge \ubar$ over the super vector space $\uu$ inherits a natural $\ZZ$-grading induced by the $\ZZ$-grading of the tensor superalgebra $T(\uu)$, which induces an $\ll$-invariant increasing filtration
\begin{align}
    \{0\} \subset \bigwedge\nolimits ^{0} \ubar \subset \bigoplus_{i = 0}^{1} \bigwedge\nolimits ^{i} \ubar \subset \ldots \subset \bigoplus_{i = 0}^{s} \bigwedge\nolimits ^i \ubar = \bigwedge \ubar,
\end{align}
and an $\ll$-invariant decreasing filtration 
\begin{align}
   \bigwedge \ubar = \bigoplus_{i = 0}^{s} \bigwedge\nolimits ^i \ubar \supset \bigoplus_{i = 1}^{s} \bigwedge\nolimits ^i \ubar \supset \ldots \supset \bigwedge \nolimits^{s} \ubar \supset \{0\}.
\end{align}
These $\ll$-invariant increasing/decreasing filtrations induce $\ll$-invariant increasing/decreasing filtration of $M \otimes \bigwedge \ubar \otimes \CC_{\rho^{\uu}}$:
\begin{align}
\begin{split}
    \{ 0\} = X_{-1} \subset X_{0} \subset X_{1} \subset \ldots \subset X_{s} = M \otimes\bigwedge \ubar \otimes \CC_{\rho^{\uu}}, \\
    M \otimes\bigwedge \ubar \otimes \CC_{\rho^{\uu}} = X^{0} \supset X^{1} \supset \ldots \supset X^{s} \supset X^{s+1} = \{0\}, 
\end{split}
\end{align}
where $X_{k} := \bigoplus_{i = 0}^{k} M \otimes \bigwedge\nolimits^{i} \ubar$ and $X^{k} := \bigoplus_{i = k}^{s} M \otimes \bigwedge \nolimits^{i} \ubar$.

The following lemma is an immediate consequence of the fact that the Dirac operator $\Dirac$ is $\ll$-invariant.

\begin{lemma} \label{lemm::increasing_decreasing_filtration}
    Let $M$ be an admissible $(\gg, \ll)$-supermodule. Then the following two assertions hold:
    \begin{enumerate}
        \item[a)] $\ker \Dirac$ has an increasing $\ll$-invariant filtration 
        \[
        \{0\} = (\ker \Dirac)_{-1} \subset (\ker \Dirac)_{0} \subset (\ker \Dirac)_{1} \subset \ldots \subset (\ker \Dirac)_{s} = \ker \Dirac
        \]
        with $(\ker \Dirac)_{k} := \ker \Dirac \cap X_{k}$ for $0 \leq k \leq s$. 
        \item[b)] $\ker \Dirac$ has a decreasing $\ll$-invariant filtration 
        \[
        \ker \Dirac = (\ker \Dirac)^{0} \supset (\ker \Dirac)^{1} \supset \ldots \supset (\ker \Dirac)^{s} \supset \{0\}
        \]
        with $(\ker \Dirac)^{k} := \ker \Dirac \cap X^{k}$ for $0 \leq k \leq s$. 
    \end{enumerate}
\end{lemma}

The $\ll$-invariant increasing/decreasing filtrations of $\ker \Dirac$ induce gradings of $\ker \Dirac$, namely
\begin{align}
\begin{split}
    \gr \ker \Dirac &:= \bigoplus_{k = 0}^{s} (\Ker \Dirac)^{k}/(\Ker \Dirac)^{k+1}, \\ \operatorname{Gr} \Ker \Dirac &:= \bigoplus_{k = 0}^{s} (\Ker \Dirac)_{k}/(\ker \Dirac)_{k-1}.
    \end{split}
\end{align}
By construction, $\ll$ leaves $(\Ker \Dirac)^{k}/(\Ker \Dirac)^{k+1}$ and $(\Ker \Dirac)_{k}/(\ker \Dirac)_{k-1}$ invariant, leading to the following lemma.

\begin{lemma} \label{lemm::gr_Gr}
    The $\ll$-equivariant maps 
    \[
    \gr : \Ker \Dirac \to \gr \ker \Dirac, \quad \operatorname{Gr} : \Ker \Dirac \to \operatorname{Gr} \Ker \Dirac
    \]
    are isomorphisms.
\end{lemma}
In turns, the previous $\ll$-equivariant maps $\gr$ and $\operatorname{Gr}$ enter the proof of following lemma. 
\begin{lemma} \label{lemm::embedding_kernel_Dirac}
    Let $M$ be an admissible $(\gg, \ll)$-supermodule. Then there are injective $\ll$-supermodule homomorphisms
    \[
    f: \ker \Dirac \hookrightarrow \ker \Delta \cap \ker \delta, \qquad g : \ker \Dirac \hookrightarrow \ker \Delta \cap \ker \partial
    \]
    given by 
    \[
    f \circ \gr = \bigoplus_{k = 0}^{s}f_{k}, \quad f_{k}(x_{k}+x_{k+1}+\ldots + x_{s}) := x_{k}
    \]
    with $x_{k} + \ldots + x_{s} \in (\Ker \Dirac)^{k}/(\Ker \Dirac)^{k+1}$ and
    \[
    g \circ \operatorname{Gr} = \bigoplus_{k = 0}^{s} g_{k}, \quad g_{k}(x_{1}+ x_{2} + \ldots + x_{k}) := x_{k}
    \]
    for $x_{1} + \ldots + x_{k} \in (\Ker \Dirac)_{k}/(\Ker \Dirac)_{k-1}$.
\end{lemma}

\begin{proof}
    We only prove that $g = \bigoplus_{k = 0}^{s} g_{k} : \ker \Dirac \hookrightarrow \ker \Delta \cap \ker \partial$ is an injective $\ll$-supermodule homomorphism. The proof for $f : \ker \Dirac \hookrightarrow \ker \Delta \cap \ker \delta$ is analogously and will be omitted.  

    First, $f$ is well-defined as for any $x := x_{1}+\ldots+x_{k} \in (\ker \Dirac)_{k}$ with $x_{i} \in M \otimes \bigwedge\nolimits ^{i} \ubar$ for $0 \leq i \leq k$, we have by Proposition \ref{prop::identification_C_Cbar_d_delta} $\Dirac = \partial - 2d$ and thus
    \[
    \Dirac(x) = -2d^{1}(x_{1}) + \ldots -2d^{k}(x_{k}) + \partial^{1}(x_{1}) + \ldots + \partial^{k}(x_{k}) = 0.
    \]
    By degree reasons, we conclude $\partial^{k}(x_{k}) = 0$, as it is of degree $k+1$. This shows that $g_{k}: (\ker \Dirac)_{k}/(\Ker \Dirac)_{k-1} \to \ker \partial^{k}$ is well-defined, and hence, $g = \bigoplus_{k=0}^{s}g_{k}$ is well-defined. Moreover, by Lemma \ref{lemm::increasing_decreasing_filtration} and Lemma \ref{lemm::gr_Gr}, the map $g$ is $\ll$-equivariant.

Second, the map $g$ is injective as $g_{k}(x) = 0$ implies $x_{k} = 0$, that is, $x \in (\ker \Dirac)_{k-1}$.

Finally, the image of $g$ is $\ker \Delta \cap \ker \partial$, as for any $x = x_{1} + \ldots +x_{k} \in (\ker \Dirac)_{k}$, we have $\Delta(x_{i}) = 0$, $1 \leq i \leq k$, since $\Delta = -2(d\partial+\partial d)$ preserves the degree. 
\end{proof}

By combining Lemma \ref{lemm::embedding_kernel_Dirac} with $\Dirac = \partial - 2d = \partial + \delta$ and $\Dirac^{2} = 2(\partial \delta + \delta\partial)$, we conclude the following Lemma. 

\begin{lemma} \label{lemm::kernel_Dirac_equal_kernels} Let $M$ be an admissible $(\gg, \ll)$-supermodule. Then
    \[
    \ker \Dirac = \ker \partial \cap \ker \delta.
    \]
\end{lemma}

\begin{theorem} \label{thm::embedding}
    Let $\pp = \ll \ltimes \uu$ be a parabolic subalgebra, and let $M$ be an admissible simple $(\gg,\ll)$-supermodule. Then there exist injective $\ll$-supermodule morphisms 
    \[
    \DC(M) \hookrightarrow H^{\ast}(\uu,M), \qquad \DC(M) \hookrightarrow H_{\ast}(\ubar,M).
    \]
\end{theorem}

\begin{proof}
    We note that the Casimir operators $\Omega_{\gg}$ and $\Omega_{\ll}$ act semisimply on $M$ by assumption. Then, using Proposition \ref{prop::decomposisition_representation_space}, we have the decomposition 
    \[
    M \otimes \overline{M}(\ss) \cong \Ker \Delta \oplus \Im \Delta,
    \]
    and we may consider $\operatorname{H}_{\Dirac'}(M)$ instead of $\DC(M)$ (\emph{cf.}~Corollary \ref{cor::different_form_DC}), where $\operatorname{H}_{\Dirac'}(M) = \ker \Dirac'/\Im \Dirac'$ and $\Dirac'$ denotes the restriction of the Dirac operator to $\Ker \Delta$. 
    
    Since $\partial$ commutes with $\Delta$, as a direct calculation shows, we can restrict $\partial$ to $\ker \Delta$, denoted by $\partial'$, and define the associated cohomology $\ker \partial'/\Im \partial'$. This cohomology is naturally an $\ll$-subsupermodule of $H^{\ast}(\uu,M)$, and it is enough to show the existence of an injective $\ll$-supermodule morphism 
    \[
    \DC(M) \cong \operatorname{H}_{\Dirac'}(M) \hookrightarrow \ker \partial'/\Im \partial'.
    \]

For simplicity, we set $V := \ker \Delta$. The idea of the proof is to decompose $\ker \Dirac'$, $\ker \partial'$, $\Im \Dirac'$, and $\Im \partial'$ into suitable $\ll$-supermodules, leveraging the $\ll$-semisimplicity of $M \otimes \overline{M}(\ss)$ to compare the corresponding components.  

We start by considering the following two short exact sequences of $\ll$-supermodules, recalling that $\Dirac, \partial$ and $d$ commute with $\ll$ and switch parity:
\begin{align*}
    0 \to \ker \Dirac' \to V \to \Pi \Im \Dirac' \to 0, \qquad 0 \to \ker \partial' \to V \to \Pi \Im \partial' \to 0.
\end{align*}
Here, $\Pi$ denotes as usual the parity switching functor. By semi-simplicity, the short exact sequences split as $\ll$-supermodules:
\begin{align*}
    V \cong \ker \Dirac' \oplus \Pi \Im \Dirac', \qquad V \cong \ker \partial' \oplus \Pi \Im \partial'.
\end{align*}

Next, we decompose $\ker \partial'$. In the following, all isomorphisms are $\ll$-supermodule isomorphisms unless otherwise stated. For the decomposition, we use Lemma \ref{lemm::kernel_Dirac_equal_kernels} to see $\ker \Dirac' \subset \ker \partial'$. Then there exists an $\ll$-invariant subspace $X$ such that
\[
\ker \partial' \cong \Ker \Dirac' \oplus X,
\]
which forces in particular
\[
\Pi \Im \Dirac' \cong X \oplus \Pi \Im \partial',
\]
by the decomposition of $V$ above. On the other hand, $\Im \Dirac' \subset \ker \Dirac'$ such that we find an $\ll$-supermodule $Y$ with 
\[
\Ker \Dirac' \cong \Im \Dirac' \oplus Y,
\]
which yields $\operatorname{H}_{\Dirac'}(M) \cong Y$, and 
\[
\ker \partial' \cong \Pi X \oplus Y \oplus \Im \partial'.
\]
This induces directly an embedding of $\ll$-supermodules
\[
\operatorname{H}_{\Dirac'}(M) \cong \ker \Dirac' / \Im \Dirac' \cong  Y \hookrightarrow \Pi X \oplus Y \cong \ker \partial' / \Im \partial'.
\]
This concludes the proof.
\end{proof}

\subsection{Hermitian real forms and unitarizable supermodules} \label{subsec::Hermitian_real_forms}

We now aim at comparing the Dirac cohomology and the Kostant's $\uu$-cohomology for a special kind of supermodules, namely \emph{unitarizable} supermodules over basic classical Lie superalgebras $\gg$. As shown \cite{Carmeli_Fioresi_Varadarajan_HW},  these supermodules are particularly relevant, as they admit a geometric realization as superspaces of sections of certain holomorphic super vector bundles on Hermitian superspaces. Furthermore, as an application, we will consider weight supermodules, and show that their Dirac cohomology is trivial unless they are of highest weight type. On our way to the above results, we start reviewing Hermitian real forms in the next subsection, following in particular the results of Fioresi and collaborators, in \cite{Carmeli_Fioresi_Varadarajan_HW} and \cite{Chuah_Fioresi_real, Fioresi_real_forms}. 


\subsubsection{Hermitian real forms} We fix a real form $\gg^{\RR}$ of a basic classical Lie superalgebra $\gg$, \emph{\emph{i.e.},} $\gg^{\RR}$ is the subspace of fixed points of some $\theta \in \aut_{2,4}(\gg)$, and we denote by $\sigma := \omega \circ \theta \in \autbar_{2,2}(\gg)$ the associated conjugate-linear involution on $\gg$ (see Proposition \ref{prop::real_form_omega} above). In the following, we may assume that $\theta$ associated to $\gg^{\RR}$ is a Cartan automorphism (\emph{cf.}~Section \ref{subsec::unitarity_real_forms}).

Following \cite{Carmeli_Fioresi_Varadarajan_HW,Chuah_Fioresi_real}, we now extend the concept of Hermitian semisimple Lie algebras over $\CC$ to basic classical Lie superalgebras. 

First, the Lie subalgebra $\gg_{0}^{\RR} \subset \gg^{\RR}$ is either semisimple or reductive with one-dimensional center, since $\gg$ is basic classical. In general, we have the decomposition 
\be
\gg_{0}^{\RR} = \gg_{0}^{\RR,\s} \oplus \mathfrak{z}(\gg_{0}^{\RR}), 
\ee
where $\gg_{0}^{\RR,\s} := [\gg_{0}^{\RR},\gg_{0}^{\RR}] \subset \gg_{0}^{\RR}$ is the commutator subalgebra, \emph{i.e.}, the semisimple part, and $\mathfrak{z}(\gg_{0}^\RR)$ the center. As we shall see, there exists a notion of Hermiticity for $\gg_{0}^{\RR,\s}$.

On $\gg^{\RR}_{0}$, the Cartan automorphism $\theta \in \aut_{2,4}(\gg)$ is an involution, such that we have the following decomposition:
\begin{align}
    \gg_{0}^{\RR} = \kk^{\RR} \oplus \pp_{0}^{\RR},
\end{align}
where $\kk^{\RR}$ is the eigenspace of $\theta$ with eigenvalue $+1$, and $\pp^{\RR}_{0}$ is the eigenspace with eigenvalue $-1$. Complexification yields 
\begin{align}
    \even = \kk \oplus \pp_{0},
\end{align}
with $\kk$ and $\pp_{0}$ being the complexifications of $\kk^{\RR}$ and $\pp_{0}^{\RR}$, respectively. This is a \emph{Cartan decomposition} for $\even$ and $\gg_{0}^{\RR}$, respectively. 

\begin{definition}[{\cite{Chuah_Fioresi_real}}]
    A real form $\gg^{\RR}$ of $\gg$ is called \emph{Hermitian} if the following two conditions hold:
    \begin{enumerate}
        \item[a)] $\gg_{0}^{\RR,\s}$ is a Hermitian Lie algebra, \emph{i.e.}, $\theta_{0} := \theta\big\vert_{\gg_{0}^{\RR}}$ induces a Cartan decomposition $\gg_{0}^{\RR,\s} = \kk' \oplus \pp'$, where $\kk',\pp'$ are the $\theta_{0}\vert_{\gg_{0}^{\RR,\s}}$-eigenspaces with eigenvalue $1$ and $-1$, respectively, such that the adjoint representation of $\kk'$ on $\pp'$ has two simple components. 
        \item[b)] $\rank \even = \rank \kk$.
    \end{enumerate}
\end{definition} 

The Hermitian real forms $\gg^{\RR}$ of basic classical Lie superalgebras $\gg$ were classified in \cite{Chuah_Fioresi_real}. We give a complete list in Appendix \ref{ap::Hermitian_real_forms}. Hermitian real forms $\gg^{\RR}$ have a Cartan decomposition
\begin{align}
\gg^{\RR} = \kk^{\RR} \oplus \pp^{\RR}, 
\end{align}
where $\pp^{\RR} := \pp_{0}^{\RR} \oplus \gg_{1}^{\RR}$, and such that $\gg_{0}^{\RR} = \kk^{\RR} \oplus \pp_{0}^{\RR}$ is a Cartan decomposition for $\gg_{0}^{\RR}$. We denote the complexification of $\pp^{\RR}$ by $\pp$, and note that the Cartan decomposition extends to $\gg$: 
\begin{align} \label{eq::Cartan_decomposition_g}
    \gg = \kk \oplus \pp, \quad \pp = \pp_{0}\oplus \gg_{1}.
\end{align}
Both decompositions are compatible with $\sigma$, as $\theta$ and $\omega$ commute, and $B_{\sigma}(\cdot,\cdot) = (\cdot, \sigma(\cdot))$ is positive definite on $\kk$ and negative definite on $\pp$, which justifies the name.

For convenience, we set $\pp_{1} := \gg_{1}$. We may decompose $\pp$ into two $\kk$-stable subspaces, $\pp = \pp^{+} \oplus \pp^{-}$, which we now describe.

The even rank condition implies
$
\hh \subset \kk\subset \even \subset \gg.
$
The root system $\Phi_{c}$ for $(\hh,\kk)$ is a subset of $\Phi_{0}$. We call a root $\alpha \in \Phi_{0}$ \emph{compact} if $\alpha \in \Phi_{c}$, or equivalently, if the associated root vector lies in $\kk$; otherwise, the root is referred to as \emph{non-compact}. The set of non-compact roots is $\Phi_{n} := \Phi \setminus \Phi_{c}$, such that $\Phi = \Phi_{c} \sqcup \Phi_{n}$. In particular, all odd roots are non-compact. For a fixed positive system $\Phi^{+}$, we set
\begin{align}
    \pp^{+} := \sum_{\alpha \in \Phi_{n}^{+}}\gg^{\alpha}, \qquad \pp^{-} := \sum_{\alpha \in \Phi_{n}^{+}} \gg^{-\alpha},
\end{align}
where $\Phi_{n}^{+} := \Phi_{n} \cap \Phi^{+}$. Then $\pp = \pp^{+} \oplus \pp^{-}.$

\begin{theorem}[\cite{Chuah_Fioresi_real, Carmeli_Fioresi_Varadarajan_HW}] \label{thm::admissible_system} There exists a positive system $\Phi^{+} \subset \Phi$, called \emph{admissible}, such that the following two assertions hold:
\begin{enumerate}
    \item[a)] $\pp^{+}$ is $\kk$-stable, that is, $[\kk,\pp^{+}] \subset \pp^{+}$.
    \item[b)] $\pp^{+}$ is a Lie subsuperalgebra, that is, $[\pp^{+}, \pp^{+}] \subset \pp^{+}$.
\end{enumerate}
\end{theorem}

A complete list of admissible systems can be found in \cite{Carmeli_Fioresi_Varadarajan_HW}. In the following, we fix an admissible positive system for $\gg$, denoted by $\Phi^{+}$. Then, note that $\kk\ltimes \pp^{+}$ is a Lie subsuperalgebra, and $[\kk,\pp^{-}] \subset \pp^{-}$, $[\pp^{-},\pp^{-}] \subset \pp^{-}$. 

\subsubsection{$\pp^{-}$-cohomology, Dirac cohomology and Hodge decomposition} Having  prepared our setting in the previous subsection, we are now ready to study the cohomology of unitarizable supermodules. In particular, in the case $\mathcal{H}$ is a simple unitarizable $\gg$-supermodule, we will show that $\mathcal{H}\otimes \overline{M}(\ss)$ decomposes as $\mathcal{H}\otimes \overline{M}(\ss) = \ker \Dirac \oplus \Im \Dirac$, and hence in this case one has that $\DC(\HH) = \ker \Dirac$, see Proposition \ref{prop::decomposition_maximal_type}. Furthermore, Theorem \ref{theorem::Hodge_Dec_Like} can be seen as a Hodge decomposition-like result for the cubic Dirac operator. In particular, it shows that the Dirac cohomology of simple unitarizable supermodules is isomorphic (as $\ll$-supermodules) to the Kostant $\overline{\uu}$-cohomology (or $\ubar$-homology), up to a twist by $\CC_{\rho^\uu}$.

We fix our notation as follows. Let $\gg^{\RR}$ be a Hermitian real form of $\gg$ with associated Cartan automorphism $\theta \in \aut_{2,4}(\gg)$. Let $B_{\theta}(\cdot,\cdot)$ denote the inner product for $\gg^{\RR}$ defined in Equation \eqref{eq::inner_product}. The associated Cartan decomposition reads $\gg^{\RR} = \kk^{\RR} \oplus \pp^{\RR}$, and we consider the complexification yielding a Cartan decomposition for $\gg = \kk \oplus \pp$ (\emph{cf.}~Equation \eqref{eq::Cartan_decomposition_g}).

We fix the parabolic subalgebra $\qq := \kk \ltimes \pp^{+}$ with $\ll = \kk$ and $\uu = \pp^{+}$, which is well-defined by Theorem \ref{thm::admissible_system}. The choice of the parabolic subalgebra $\qq = \kk \ltimes \pp^{+}$ leads to the Cartan decomposition $\gg = \kk \oplus \ss$ with $\ss := \pp = \pp^{+} \oplus \pp^{-}$. The parabolic subalgebra $\qq$ is an example of a $\theta$-stable subsuperalgebra, that is,  $\theta$ preserves $\kk, \pp^{+}$ and $\pp^{-}$, which is immediate as $\pp$ is the $\theta$-eigenspace with eigenvalue $-1$ and $\kk$ is the $\theta$-eigenspace with eigenvalue $+1$. Moreover, by the definition of $\sigma$ and the action of $\omega$ on weight spaces given in Proposition \ref{prop::real_form_omega}, we conclude $\sigma(\pp^{\pm}) = \pp^{\mp}$.

 We restrict $B_{\theta}$ to $\pp^{\RR} (=\ss^{\RR})$, and fix some orthonormal basis $Z_{1}, \ldots, Z_{2s}$. Then $\pp^{+}$ and $\pp^{-}$ are spanned by 
\begin{equation} \label{eq::def_Z_i}
    u_{j} := \frac{Z_{2j-1}+iZ_{2j}}{\sqrt{2}}, \quad \overline{u}_{j} := \frac{Z_{2j-1}-iZ_{2j}}{\sqrt{2}}
\end{equation}
respectively, for $j = 1 \ldots, s$, as a direct calculation yields. In particular, $\sigma(u_{j}) = \overline{u}_{j}$ for all $j = 1, \ldots, s$. 

We study the action of $u_{j}$ and $\ou_{j}$ on $\HH \otimes \overline{M}(\ss)$ for some unitarizable $\gg$-supermodule $(\HH,\bracket_{\HH})$. For that, we associate to $\HH \otimes \overline{M}(\ss)$ the non-degenerate super Hermitian product 
\begin{equation}
\bra v\otimes P,w \otimes Q \rangle_{\HH \otimes \overline{M}(\ss)} := \bra v,w\ket_{\HH} \bra P,Q\ket_{\overline{M}(\ss)}
\end{equation}
for any $v \otimes P, w \otimes Q \in \HH \otimes \overline{M}(\ss)$. We refer to Section \ref{subsec::DiracCohomo} for an explicit realization of $\bracket_{\overline{M}(\ss)}$. By construction, we study the action componentwise.

\begin{lemma} \label{lemm::adjoint_u_H}
Let $(\HH,\bracket)$ be a unitarizable $\gg^{\RR}$-supermodule. Then the following holds for all $j = 1, \ldots, s$:
\[
u_{j}^{\dagger} = - \overline{u}_{j}
\]
\end{lemma}
\begin{proof}
    As $\HH$ is a unitarizable $\gg^{\RR}$-supermodule, the orthonormal basis $Z_{1}, \ldots, Z_{2s}$ of $\pp^{\RR}$ satisfies
    \[
    Z_{j}^{\dagger} = -Z_{j}, \quad j=1,\ldots,2s.
    \]
The statement follows with Equation \eqref{eq::def_Z_i}.
\end{proof}

For $(\overline{M}(\ss),\bracket_{\overline{M}(\ss)})$, we described the adjoint of any $u_{j}$ already in Lemma \ref{lemm::adjoint_M(s)}, namely, we have
    \[
    u_{j}^{\dag} = -(-1)^{p(u_{j})}\overline{u}_{j}
    \]
for all $j = 1, \ldots, s.$ By combining Lemma \ref{lemm::adjoint_M(s)} and Lemma \ref{lemm::adjoint_u_H}, we have proven the following lemma.

\begin{lemma} \label{lemm::adjoint_Dirac}
     Let $(\HH,\bracket_{\HH})$ be a unitarizable $\gg$-supermodule. Then the cubic Dirac operator $\Dirac$ is anti-selfadjoint with respect to $\bracket_{\HH \otimes \overline{M}(\ss)}$. In particular, 
     \[
     \ker \Dirac = \ker \Dirac^{k}
     \]
     for any $k \in \ZZ_{+}$.
\end{lemma}
In turns, this leads to the following decomposition.
\begin{proposition} \label{prop::decomposition_maximal_type}
    Let $\HH$ be a simple unitarizable $\gg$-supermodule. Then 
    \[
    \HH \otimes \overline{M}(\ss) = \ker \Dirac \oplus \Im \Dirac.
    \]
In particular, the Dirac cohomology of a simple unitarizable $\gg$-supermodule $\HH$ is
    \[
    \DC(\HH) = \ker \Dirac.
    \]
\end{proposition}

\begin{proof}
    First, the statement of Proposition \ref{prop::decomposisition_representation_space} holds more generally for any admissible $(\gg, \ll)$-supermodule which has infinitesimal character - the argument follows from similar lines. Consequently, it is enough to prove $\Im \Dirac \cap \ker \Dirac = \{0\}$. 

    Let $v \in \Im \Dirac \cap \ker \Dirac$. Then there exists some $w \in \HH \otimes \overline{M}(\ss)$ such that $\Dirac w = v$, and by positive definiteness of $(\bracket_{\HH \otimes \overline{M}(\ss)})_{0,1}$ and Lemma \ref{lemm::adjoint_Dirac}, we have
    \[
    0 \leq (-i)^{p(v)}\bra v,v \ket = (-i)^{p(v)}\bra \Dirac w, v\ket = - (-i)^{p(v)}\bra w, \Dirac v \ket = 0.
    \]
    This forces $v = 0$. 
\end{proof}


We now consider the decomposition $\Dirac = C + \Cbar$ as in Equation \eqref{eq::C_Cbar}. A direct calculation yields that the adjoint of $C$ is $-\Cbar$. More precisely, the following lemma holds. 

\begin{lemma} \label{lemm::relations_C_Cbar}
    The following assertions hold:
    \begin{enumerate}
        \item[a)] $\ker \Dirac = \ker C \cap \ker \Cbar$.
        \item[b)] $\Im C$ is orthogonal to $\ker \Cbar$ and $\Im \Cbar$, while $\Im \Cbar$ is orthogonal to $\Ker C$.
    \end{enumerate}
\end{lemma}

\begin{proof} a) As $\Dirac = C + \Cbar$, the inclusion $\ker C \cap \ker \Cbar \subset \ker \Dirac$ is clear. Assume $\Dirac v = 0$ for some $v \in \HH \otimes \overline{M}(\ss)$, \emph{i.e.}, $Cv = -\Cbar v$. Consequently, 
\[
(-i)^{p(Cv)}\bra Cv, Cv\ket_{\HH \otimes \overline{M}(\ss)} = (-i)^{p(Cv)}\bra Cv, -\Cbar v\ket_{\HH \otimes \overline{M}(\ss)} = (-i)^{p(Cv)}\bra C^{2}v,v \ket_{\HH \otimes \overline{M}(\ss)} = 0,
\]
where we use $C^{2} = 0$ by Lemma \ref{lemm::square_C_Cbar}.
In particular, $v \in \ker C$ by super positive definiteness. Analogously, $v \in \ker \Cbar$.

    b) Let $v \in \Im C$ and $w \in \ker \Cbar$. We show $\bra v,w \ket_{\HH \otimes \overline{M}(\ss)} = 0$. As $v \in \Im C$, there exists some $v' \in \HH \otimes \overline{M}(\ss)$ such that $Cv'=v$. We conclude
    \[
    \bra v,w \ket_{\HH \otimes \overline{M}(\ss)} = \bra Cv',w\ket_{\HH \otimes \overline{M}(\ss)} = - \bra v', \Cbar w\ket_{\HH \otimes \overline{M}(\ss)} = 0,
    \]
    since $w \in \ker \Cbar$. Analogously, one can prove that  $\Im \Cbar$ is orthogonal to $\Ker C$.

    We show that $\Im C$ and $\Im \Cbar$ are orthogonal. Let $v \in \Im C \cap \Im \Cbar$, then there exists $v_{C},v_{\Cbar} \in \HH \otimes \overline{M}(\ss)$ with $Cv_{C} = v$ and $\Cbar v_{\Cbar} = v$. We conclude 
    \[
    (-i)^{p(v)}\bra v,v \ket_{\HH \otimes \overline{M}(\ss)} = (-i)^{p(v)}\bra Cv_{C}, \Cbar v_{\Cbar}\ket_{\HH \otimes \overline{M}(\ss)} = -(-i)^{p(v)} \bra C^{2}v_{C},v_{\Cbar}\ket_{\HH \otimes \overline{M}(\ss)} = 0,
    \]
    by Lemma \ref{lemm::square_C_Cbar}. Hence, by super positive definiteness of $\bracket_{\HH \otimes \overline{M}(\ss)}$, we have $v = 0$.
\end{proof}

By combining all the previous results, we are now in the position to prove the following theorem, which is the main result of the present section.

\begin{theorem} \label{theorem::Hodge_Dec_Like}
    Let $\HH$ be a simple unitarizable $\gg$-supermodule. Then the following assertions hold:
    \begin{enumerate}
        \item[a)] $\HH \otimes \overline{M}(\ss) = \ker \Dirac \oplus \Im C \oplus \Im \Cbar$.
        \item[b)] $\ker C = \ker D \oplus \Im C$.
        \item[c)] $\ker \Cbar = \ker \Dirac \oplus \Im \Cbar$.
    \end{enumerate}
    In particular, we have an isomorphism of $\ll$-supermodules
    \[
    \DC(\HH) \cong H^{\ast}(\uu,\HH) \otimes \CC_{\rho^{\uu}}.
    \]
\end{theorem}
\begin{proof}
    By Proposition \ref{prop::decomposition_maximal_type}, the decomposition $\Dirac = C + \Cbar$ and Lemma \ref{lemm::relations_C_Cbar}, we have 
    \[
    \HH \otimes \overline{M}(\ss) = \ker \Dirac \oplus \Im \Dirac \subset \ker \Dirac \oplus \Im C \oplus \Im \Cbar,
    \]
    \emph{i.e.}, $\Im \Dirac = \Im C \oplus \Im \Cbar$. This proves a). The assertions b) and c) follows with a) and Lemma \ref{lemm::relations_C_Cbar}. 

    The isomorphisms are now a direct consequence of b), c) and Proposition \ref{prop::identification}.
\end{proof}

\subsection{Application: Dirac cohomology and weight supermodules} \label{subsec::application}

In this section, we prove that the Dirac cohomology of simple weight supermodules is trivial unless they are of highest weight type. This generalizes the result for reductive Lie algebras over $\CC$ in \cite{Huang_Euler_Poincaré} to basic classical Lie superalgebras.

First, any simple weight $\gg$-supermodule $M$ admits an infinitesimal character such that $\DC(M) \subset H^{\ast}(\uu, M)$ by Theorem \ref{thm::embedding}. By Proposition \ref{prop::HW_non-trivial}, we already know that $\DC(M) \neq \{0\}$ if $M$ is of highest weight type. We show that $\DC(M) = \{0\}$ unless $M$ is of highest weight type. To this end, we identify $H^{i}(\uu, M)$ with $\Ext^{i}_{\uu}(\CC, M)$ for any $i > 0$. To compute $\Ext^{i}_{\uu}(\CC, M)$, we use the subsequent lemma. 

\begin{lemma} \label{lemm::injective_resolutions}
    Let $M$ be a weight $\gg$-supermodule. Assume that there exists a positive root $\alpha$ such that $e_{-\alpha}$ acts injectively on $M$. Then there exists an injective resolution of $M$  
    \[
    0 \to M \to I_{0} \to I_{1} \to \ldots
    \]
    such that $e_{-\alpha}$ acts injectively on every $I_{i}$ for $i \in \ZZ_{+}$. 
\end{lemma}
The proof of the above follows from theorem \ref{theorem::resolution}. 
Let $M$ be a simple weight $\gg$-supermodule. Assume $M$ is not of highest weight type. By Theorem \ref{thm::DMP}, $M$ is (isomorphic to) the unique simple quotient $L_{\pp}(V)$ of a parabolically induced $\gg$-supermodule $M_{\pp}(V)$, 
where $\pp = \ll \ltimes \uu$ is a parabolic subalgebra with a good Levi subalgebra, and $V$ is a cuspidal $\ll$-supermodule. Recall that an $\ll$-supermodule is called cuspidal if for any $\alpha \in \Phi(\ll;\hh)_{0}$ the associated root vector $e_{\alpha}$ acts injectively on $V$. 

As $\pp$ is parabolic, it contains a Borel subalgebra $\bb := \hh \oplus \nn^{+}$ (\emph{cf.}~Lemma \ref{lemm::parabolic_subset}), where $\nn^{+}$ is the maximal radical of $\bb$. Moreover, $\ll \neq \hh$ by Section \ref{subsubsec::Highest_weight_supermodules}. We conclude $\ll \cap \nn^{+} \neq \{0\}$. In particular, we have proven the following lemma.

\begin{lemma}\label{lemm::existence_bijective_operator}
    There exists a root $\alpha \in \Phi_{0}^{+}(\ll, \hh)$ such that $e_{-\alpha}$ acts injectively on $M$.
\end{lemma}

Combining Lemma \ref{lemm::injective_resolutions} and Lemma \ref{lemm::existence_bijective_operator}, we conclude the subsequent theorem.

\begin{theorem} \label{thm::DC_weight_supermodules}
    Let $M$ be a simple weight $\gg$-supermodule. Then $\DC(M) = \{0\}$ unless $M$ is a highest weight $\gg$-supermodule.
\end{theorem}

\begin{proof}
    If $M$ is a highest weight module, we have $\DC(M) \neq 0$ by Proposition \ref{prop::HW_non-trivial}. Assume $M$ is not a highest weight supermodule. Then there exists some $\alpha \in \Phi(\ll, \hh)$ such that $e_{-\alpha}$ acts injectively. On the other hand, by Lemma \ref{lemm::injective_resolutions}, there exists an injective resolution of $M$
    \[
    0 \to M \to I_{0} \to I_{1} \to \ldots
    \]
    such that $e_{-\alpha}$ acts injectively on any $I_{i}$. Hence, the space of $\gg$-invariants $I_{i}^{\gg}$ is trivial for all $i = 0,1,2,\ldots$. In particular, $\{0\} = \Ext^{i}_{\uu}(\CC, M) = H^{i}(\uu, M)$. The statement follows from Theorem \ref{thm::embedding}.
\end{proof}

\appendix

\section{}

In this appendix, we briefly review results leading to lemma \ref{lemm::injective_resolutions} in the main text and provide a table of the Hermitian real forms.  

\subsection{Injective resolutions} In the following, we let $\gg$ be a Lie superalgebra. We start with the following definitions. 
\begin{definition}[Essential Homomorphism / Essential Extension]
    Let $\phi: M \rightarrow N $ be an injective morphism of $\gg$-supermodules. We say that $\phi$ is essential if $\phi (M) \cap H \neq 0$ for every non-zero sub $\gg$-supermodule $H \subset N$. Accordingly, if $\phi : M \rightarrow N $ is essential, we say that $\gg$-module $N$ is an essential extension of $\phi (M)$.
    \end{definition}
Let now $M$ be a $\gg$-supermodule and consider the following setting: let $\alpha$ be any positive root of $\gg$, and assume the root vector $e_{-\alpha}$ acts injectively on $M$. The following lemma shows that essential homomorphisms preserve injective action of $e_{-\alpha}$ as given above.
\begin{lemma} \label{lemm::essext} In the above setting, let $\phi : M \rightarrow N$ be an essential homomorphism of $\gg$-supermodules and let $e_{-\alpha}$ acts injectively on $M$ for $\alpha$ a positive root. Then $e_{-\alpha}$ acts injectively on $N$.
\end{lemma}
\begin{proof} By contradiction, assume $e_{-\alpha}$ does not act injectively on $N$. Then $e_{-\alpha} v = 0$ for some non-zero $v \in N$. Since $\phi$ is essential, there exist $u\in \UE({\gg})$ such that $u v \in \phi (M)$ and $u v \neq 0$. On the other hand, there exist another element $u^\prime \in \UE (\gg)$ such that $e_{-\alpha}^n v = u^\prime e_{-\alpha}$ for large enough $n \in \NN.$ This implies that 
$e_{-\alpha}^n (uv) = u^\prime (e_{-\alpha} v) = 0$. On the other hand, $M \cong \phi (M)$ via $\phi$, therefore $uv$ is the image of an element in $M$, where the action is injective, hence $u v = 0$, a contradiction.    \end{proof}    
The previous result turns out to be particularly useful since every $\gg$-supermodule posses an essential homomorphisms to an injective $\gg$-supermodule. More precisely, the following holds.
\begin{lemma} \label{lemm::inj_super} There exists a unique (up to isomorphism) essential homomorphism $\phi : M \rightarrow I $ such that $I$ is an injective $\gg$-supermodule.
\end{lemma}
\begin{proof} Let $\hat I$ be any injective $\aa$-supermodule containing $M$. Let $(S_{\hat I}, \subseteq)$ be the set of essential extension of $M$ contained in $\hat I$ with an ordering given by the inclusion (notice that elements in $ S_{\hat I}$ are stable under union). By Zorn lemma, take the maximal essential extension $I$ of $M$ which is contained in $\hat I$. Similarly, Zorn lemma guarantees that there exists a maximal sub-supermodule $H \subset \hat I$ such that $H \cap I =0$. This leads to the following diagram
$$
\xymatrix{
& 0 \ar[d] & & 0 \ar[dl] \\
0 \ar[r] & I \ar[d]_{i} \ar[r]^f & \hat{I} / H \ar[dl]^\pi \\
& \hat I.
}
$$
The homomorphisms $i$ and $f$ are essential by definition and by construction respectively. Moreover, the existence of the map $\pi$ such that $\pi \circ f = i$ follows from the assumption that $\hat I$ is injective. Finally, $\pi : \hat I /H \rightarrow \hat I$ is injective since $f$ is essential, hence injective as well, and its image in $\hat I$ defines an essential extension of $\hat I$. By maximality of $I$, then $f$ has to be surjective, hence an isomorphism $I \cong \hat I /H$. Thus, the following short exact sequence of $\aa$-supermodules
$$
\xymatrix{
0 \ar[r] & H \ar[r] & \hat I \ar[r] &  I \ar@/_1pc/[l]_{\pi} \ar[r] & 0 
}
$$ 
is split and, in particular, $I$ is injective as $\hat I = I \oplus H$ with $H \cap I = 0$. \\
For uniqueness, let $\varphi : M \hookrightarrow I$ be the above essential homomorphism and assume $\tilde \varphi : M \hookrightarrow \tilde I$ is another monomorphism with $\tilde{I}$ injective, to conclude that $\tilde I \cong I$. We leave the details to the readers. 
\end{proof}
The previous result justifies the following definition.
\begin{definition}[Injective Envelope] Let $M$ be a $\gg$-supermodule and let $\phi : M \rightarrow I$ be its the unique essential homomorphism to an injective $\gg$-supermodule, as in lemma \ref{lemm::inj_super}. We call the essential extension $\gg$-module $I$ the injective envelope of $M$.
\end{definition}
Putting together the two previous lemmata, one has the following.
\begin{theorem} \label{theorem::resolution} Let $M$ be a $\gg$-module and let $\alpha$ be a positive root such that the root vector $e_{-\alpha}$ acts injectively on $M$. Then $e_{-\alpha}$ acts injectively also on the injective envelope of $M$. In particular, there exists an injective resolution $M \twoheadrightarrow I^\bullet$ such that the root vector $e_{-\alpha}$ acts injectively on every $I^i \subset I^\bullet.$
\end{theorem}
\begin{proof} The first part of the theorem follows immediately from lemmata \ref{lemm::essext} and \ref{lemm::inj_super}. For the second part, set $I^i$ to be the injective envelope of $I^{i-1}, $ hence the injective action of $e_{-\alpha}$ is inherited by the resolution $I^\bullet$.
\end{proof}
\clearpage
\subsection{Hermitian real forms} \label{ap::Hermitian_real_forms}
The Hermitian real forms $\gg^{\RR}$ of basic classical Lie superalgebras $\gg$ are summarized in the following table \cite{ Carmeli_Fioresi_Varadarajan_HW, Fioresi_real_forms}, where we emphasize that $\gg^{\RR}$ is uniquely determined by the indicated real Lie subalgebra $\gg_{0}^{\RR}$.

\begin{table}[h]
\centering
\[
\renewcommand{\arraystretch}{1.5}
\begin{array}{|c|c|}
\hline
\mathfrak{g} & \mathfrak{g}_{0}^{\RR} \\ \hline
\mathfrak{sl}(m\vert n) & 
\begin{array}{l}
\mathfrak{su}(p,m-p) \oplus \mathfrak{su}(n) \oplus i\RR \\
\mathfrak{su}(p,m-p) \oplus \mathfrak{su}(r,n-r) \oplus i\RR
\end{array} \\ \hline

\begin{array}{c}
B(n\vert m), \\
D(n\vert m)
\end{array} & 
\begin{array}{l}
\mathfrak{sp}(m,\RR) \\
\mathfrak{so}(p) \oplus \mathfrak{sp}(m,\RR) \\
\mathfrak{so}^*(2n) \oplus \mathfrak{sp}(m) \\
\mathfrak{so}(q,2) \oplus \mathfrak{sp}(m,\RR)
\end{array} \\ \hline

C(m) & \mathfrak{sp}(m,\RR) \oplus \mathfrak{so}(2) \\ \hline

D(2,1;\alpha) & 
\begin{array}{l}
\mathfrak{sl}(2,\RR) \oplus \mathfrak{sl}(2,\RR) \oplus \mathfrak{sl}(2,\RR) \\
\mathfrak{su}(2) \oplus \mathfrak{su}(2) \oplus \mathfrak{sl}(2,\RR)
\end{array} \\ \hline

F(4) & 
\begin{array}{l}
\mathfrak{sl}(2,\RR) \oplus \mathfrak{so}(7) \\
\mathfrak{su}(2) \oplus \mathfrak{so}(5,2)
\end{array} \\ \hline

G(3) & \mathfrak{sl}(2,\RR) \oplus \mathfrak{g}_c \\ \hline
\end{array}
\]
\caption{Hermitian real Lie superalgebras. For $B(n \vert m)$, the values of $p$ and $q$ are $p = 2n+1$ and $q = 2n-1$, while for $D(n \vert m)$, $p = 2n$ and $q = 2n-2$. Moreover, the Lie algebra $\gg_C$ denotes the compact real form of $G_2$.
}
\end{table}

\printbibliography

\end{document}